\documentclass{cpamart1}     
\received{February 23, 2023}       
\volume{000}
\startingpage{1}

\authorheadline{I. Corwin and A. Knizel }
\titleheadline{ Stationary measure for the Open KPZ }

\usepackage{subcaption}
\usepackage{wrapfig}
\usepackage{multicol}
\usepackage{mathtools}

\usepackage{graphicx}
\usepackage{soul}
\usepackage{xcolor}
\usepackage{amssymb}
\usepackage{bigints}

\usepackage{breqn}

\usepackage{verbatim}
\usepackage{placeins}
\usepackage{cite}
\usepackage{caption}
\usepackage{stmaryrd}
\usepackage{enumerate}
\usepackage{afterpage}
\usepackage{enumitem}
\usepackage{bmpsize}
\usepackage{hyperref}
\usepackage{tabu}
\usepackage{enumitem}
\usepackage{tikz}
\usepackage{tikz-cd}
\usetikzlibrary{matrix,graphs,arrows,positioning,calc,decorations.markings}
\definecolor{dullmagenta}{rgb}{0.4,0,0.4}   
\definecolor{darkblue}{rgb}{0,0,0.4}
\def\note#1{\textup{\textsf{\color{blue}(#1)}}}

\mathtoolsset{showonlyrefs}

\newcommand\edit[1]{\textcolor{green}{#1}}

\usepackage{mathptmx}

\newtheorem{theorem}{Theorem}[section]
\newtheorem{lemma}[theorem]{Lemma}
\newtheorem{corollary}[theorem]{Corollary}
\newtheorem{proposition}[theorem]{Proposition}
\theoremstyle{definition}
\newtheorem{definition}[theorem]{Definition}
\theoremstyle{remark}
\newtheorem{remark}[theorem]{Remark}

\newtheorem{assumption}[theorem]{Assumption}

\newtheorem{conjecture}[theorem]{Conjecture}

\def\xuNunscaled{\mathbf{u}^{(N),u}}
\def\xvNunscaled{\mathbf{v}^{(N),v}}
\def\xuN{\hat{\mathbf{u}}^{(N),u}}
\def\xvN{\hat{\mathbf{v}}^{(N),v}}
\def\xu{\mathbf{u}^{u}}
\def\xv{\mathbf{v}^{v}}

\def\hscale{H^{(N)}}
\def\zscale{Z^{(N)}}
\def\e{{\varepsilon}}

\def\tzscale{\tilde{Z}^{(N)}}

\def\G{\mathcal{G}}
\def\p{\mathfrak{p}}
\def\SuppCDH{\mathcal{F}}
\def\Suppp{\mathcal{S}}
\def\SupppU{\mathcal{SU}^{d,u}}
\def\SupppV{\mathcal{SV}^{d,v}}
\def\Supp{\mathrm{S}}
\def\Supph{\mathrm{\hat{S}}}
\def\C{\mathbb{C}}
\def\R{\mathbb{R}}
\def\Z{\mathbb{Z}}
\def\E{\mathbb{E}}
\def\P{\mathbb{P}}
\def\Y{\mathbb{Y}}
\def\T{\mathbb{T}}
\def\Yh{\mathbb{\hat{Y}}}
\def\CDWtransition{\mathsf{W}}

\def\CDHtransition{\mathsf{CDH}}

\def\Cuv{\mathsf{C}_{u,v}}
\def\Cduv{\mathsf{C}_{d,u,v}}

\def\ErrLem{\widehat{\mathsf{Err}}}
\def\ErrLemtilde{\widetilde{\mathsf{Err}}}
\newcommand{\iu}{{i}}
\newcommand{\AqP}{{\mathcal A}}
\newcommand{\Error}{{\mathsf{Err}}}
\newcommand{\Err}{{\mathsf{E}}}
\DeclarePairedDelimiter\ceil{\lceil}{\rceil}

\def\r{r}

\def\tt{\rho}

\def\AA{\mathsf{A}}
\def\BB{\mathsf{B}}

\begin{document}                        


\title{Stationary measure for the open KPZ equation}

\author{Ivan Corwin }{ Department of Mathematics, Columbia University}
\author{Alisa Knizel }{Department of Statistics,
    the University of Chicago}





\begin{abstract}
We provide the first construction of stationary measures for the open KPZ equation on the spatial interval $[0,1]$ with general inhomogeneous Neumann boundary conditions at $0$ and $1$ depending on real parameters $u$ and $v$, respectively. When $u+v\geq 0$ we uniquely characterize the constructed stationary measures through their multipoint Laplace transform which we prove is given in terms of a stochastic process that we call the continuous dual Hahn process.

Our work relies on asymptotic analysis of Bryc and Weso{\l}owski's \cite{BW1} Askey-Wilson process formulas for the open ASEP stationary measure (which in turn arise from  Uchiyama, Sasamoto and Wadati's \cite{Uchiyama_2004} Askey-Wilson Jacobi matrix representation of Derrida, Evans, Hakim and Pasquier's \cite{Derrida_1993} matrix product ansatz) in conjunction with Corwin and Shen's \cite{CS} proof that open ASEP converges to open KPZ under weakly asymmetric scaling.

\end{abstract}

\maketitle






\section{Introduction}\label{sec:intro}

The  {\em open Kardar-Parisi-Zhang} (KPZ) equation models stochastic interface\\ growth on $[0,1]$ subject to inhomogeneous Neumann boundary conditions at $0$ and $1$. The equation is written as
\begin{align}\label{eq_KPZ_informal}
\partial_T H(T,X) = \tfrac{1}{2}\partial_X^2 H(T,X) + \tfrac{1}{2}\big(\partial_X H(T,X)\big)^2 +  \xi(T,X),
\end{align}
where $\xi$ is space-time white noise and for all $T>0$ we impose the boundary conditions
\begin{equation}\label{eq_KPZ_bcs}
\partial_X H(T,X)\big\vert_{X=0} = u,\qquad \partial_X
H(T,X)\big\vert_{X=1} = -v, \quad u, v\in \mathbb R.
\end{equation}
This requires a careful definition that we provide here, following \cite[Definition 2.5]{CS}.

\medskip
Let $C\big([0,\infty),C([0,1])\big)$ denote the space of continuous functions from $[0,\infty)\to C([0,1])$ where $C([0,1])$ is the space of continuous function from $[0,1]\to\R$. Let $(\Omega, \mathcal{F},\P)$ denote a probability space which supports a space-time white noise $\xi$ and a random almost-surely strictly positive function $Z_0$ taking values in $C([0,1])$ and satisfying $\sup_{X\in [0,1]} \E[Z_0(X)^p]<\infty$ for all $p>0$. For $t\geq 0$, let $\mathcal{F}_{t}$ denote the filtration generated by $Z_0$ and $\big(\xi(s,x)\big)_{s\leq t, x\in [0,1]}$. Then the {\em mild solution} to the stochastic heat equation (SHE)  inhomogeneous Robin boundary conditions is a random function $Z\in C\big([0,\infty),C([0,1])\big)$ satisfying:
\begin{itemize}[leftmargin=*]
\item Initial data: $Z(0,\cdot) = Z_0(\cdot)$ almost surely (here and below $f(\cdot)$ denotes the function $\cdot\mapsto f(\cdot)$) as random functions in $C([0,1]).$
\item Measurability: $Z(t,\cdot)$ is measurable with respect to $\mathcal{F}_t$ for all $t\geq 0$.
\item Duhamel form of SHE: For all $T>0$ and $X\in [0,1]$
\begin{align}\label{eq:mildSHE}
Z(T,X)= \int_0^1 P_{u,v}(T,X,Y) Z_0(Y)dY + \int_0^1 \int_0^\infty P_{u,v}(T-S,X,Y)Z(S,Y) \xi(dS,dY)
\end{align}
where the integral against $\xi$ is in the sense of It\^o, and $P_{u,v}(T,X,Y)$ is the Gaussian heat kernel on $[0,1]$ with inhomogeneous Robin boundary conditions, i.e., for all $T>0$ and $ X,Y\in (0,1)$
$$
\partial_T P_{u,v}(T,X,Y) = \partial_X^2 P_{u,v}(T,X,Y),
$$
with $P_{u,v}(0,X,Y) = \delta_{X=Y}$ and, for all $T>0$ and $Y\in [0,1]$,
\begin{align}
\partial_X P_{u,v}(T,X,Y) \big\vert_{X=0} &= (u-\tfrac{1}{2}) P_{u,v}(T,0,Y),\\
\partial_X P_{u,v}(T,X,Y) \big\vert_{X=1} &= -(v-\tfrac{1}{2}) P_{u,v}(T,1,Y).
\end{align}
\end{itemize}
The existence, uniqueness and strict positivity (i.e., provided that $Z_0$ is almost surely strictly positive then almost surely $\big(Z(t,x)\big)_{t\geq 0,x\in [0,1]}$ is likewise strictly positive) for the solution of the SHE are proved in \cite[Proposition 2.7]{CS} and \cite[Proposition 4.2]{Parekh}.

The  Hopf-Cole solution to the open Kardar-Parisi-Zhang with inhomogeneous Neumann boundary conditions parameterized by  $u,v\in \R$ and initial data $H_0=\log Z_0 \in C([0,1])$ is then the random function $H\in C\big([0,\infty),C([0,1])\big)$ defined on the same probability space as above by the equality
\begin{equation}\label{eq:KPZdefn}
H(t,x):= \log Z(t,x),\quad \forall t\geq 0, x\in [0,1].
\end{equation}
Owing to the strict positivity, this logarithm is well-defined.

\medskip
Informally one writes the SHE as the solution to the following stochastic PDE
\begin{equation}\label{eq_SHE_informal}
\partial_T Z(T,X) = \tfrac{1}{2}\partial_X^2 Z(T,X)+  \xi(T,X) Z(T,X)
\end{equation}
for $T\geq 0$ and $X\in [0,1]$ with boundary conditions that for all $T>0$,
\begin{equation}\label{eq_SHE_bcs}
\partial_X Z(T,X)\big\vert_{X=0} =(u-\tfrac{1}{2})  Z(T,0),\quad \partial_X Z(T,X)\big\vert_{X=1} = -(v-\tfrac{1}{2}) Z(T,0).
\end{equation}
Justifying the above Hopf-Cole notion of solution to the KPZ equation has a long history going back in the full-line case to \cite{BC95}. For the above open KPZ equation, \cite{GH19} uses regularity structures to show that this Hopf-Cole solution arises when one smoothes the noise $\xi$ (in which case all equations make classical sense) and then renormalizes the solution as the smoothing is removed. Note that
in going from the SHE \ref{eq_SHE_bcs} to KPZ equation boundary condition \ref{eq_KPZ_bcs} we have removed a factor of $1/2$. This is simply a convention used to match the parameterization of the KPZ boundary conditions present in \cite{GH19} and \cite{AHL_2020__3__87_0}.

\medskip

Our aim in this work is to provide a characterization of stationary
solutions to the above open KPZ equation. The solution to the open KPZ equation is a Markov process in time with state space given by $C([0,1])$. This process does not have stationary probability measures in the usual sense since there is an overall drift and diffusion of the height function (in a similar spirit to how the SSRW does not have a stationary probability measure). However, as we will show, the open KPZ increment Markov process $\big(H(T,X)-H(T,0))_{T\geq 0, X\in [0,1]}$ does have stationary probability measures. Precisely, we say that a probability measure $\mu_{u,v}$ on $C([0,1])$ is {\em stationary} for the open KPZ increment process if the following holds: For all time $T\geq 0$, the law of $\big(H(T,X)-H(T,0)\big)_{X\in [0,1]}$ equals $\mu_{u,v}$ where $H(T,X)$  is the Hopf-Cole solution to the open Kardar-Parisi-Zhang with inhomogeneous Neumann boundary conditions parameterized by  $u,v\in \R$ and initial data $H_0\in C([0,1])$ whose law is $\mu_{u,v}$. Rather than working directly with the stationary measure $\mu_{u,v}$, we will often find it easier to think of a random function $H_{u,v}\in C([0,1])$ whose law is $\mu_{u,v}$, e.g. the canonical process on the probability space $(C([0,1]),\mathcal{F},\mu_{u,v})$ with $\mathcal{F}$ the Borel sigma-algebra for $C([0,1])$.

Theorem \ref{thm_main} provides the first construction of stationary probability
measures $\mu_{u, v}$ for the open KPZ increment process for all choices of $u$ and $v$. For $u+v\geq 0$, we completely characterize $\mu_{u,v}$ via a duality --- its multi-point Laplace transform is explicitly given in terms of a Markov process that we call the {\em continuous dual Hahn process}.
A simple case of these formulas shows that for $u,v>0$ and $c\in (0,2u)$,
\begin{equation}\label{eq:Hincform}
\E\left[e^{-cH_{u,v}(1)}\right] = e^{c^2/4}\,\cdot\,\dfrac{\int\limits_0^{\infty} e^{-r^2}\,\cdot \, \frac{\big|\Gamma(\frac{c}{2}+u+ir)\Gamma(-\frac{c}{2}+v+ir)\big|^{2}}{\big|\Gamma(2ir)\big|^2} dr}{\int\limits_0^{\infty} e^{-r^2}\,\cdot \,\frac{\big|\Gamma(u+ir)\Gamma(v+ir)\big|^{2}}{\big|\Gamma(2ir)\big|^2} dr}.
\end{equation}
%
The notation on the left-hand side needs a bit of explanation. As noted above, we are using $H_{u,v}(X;\omega)$ to denote the canonical process associated with the probability space $(C([0,1]),\mathcal{F},\mu_{u,v})$. The expectation $\E$ simply denotes integrating against the measure $\mu_{u,v}$. In other words it could be written as $\int_{C([0,1])} e^{-c \omega(1)}d\mu_{u,v}(\omega)$. However, as is often the case in working with random variables versus their measures, we find it more clear to simply think of $H_{u,v}$ as a random function with law $\mu_{u,v}$. In \eqref{eq:Hincform}, $H_{u,v}(1)$ records the net height change across the interval $[0,1]$. For $u,v>0$ and  $c\in (0,2u)$ the integral on the right-hand side involves a continuous integrand. In the case where either $u$ or $v$ are negative, the formula has an extension (that follows from our main result, Theorem \ref{thm_main}, below) involving a continuous integrand plus a sum of discrete atoms.

The Laplace transform formulas for $\mu_{u,v}$ were inverted after the first version of this paper was posted. In the mathematics literature this came in  work of \cite{BKWW} while in the physics literature it came in work of \cite{BLD}. The inversions provide a satisfying probabilistic description for the stationary
measures: $\mu_{u,v}$ is equal to the distribution of $2^{-1/2} W + Y$ where $W$ and $Y$ are independent stochastic processes that we now briefly describe. $W\in C([0,1])$ is a standard Brownian motion. $Y\in C([0,1])$ is given by a reweighing of a Brownian motion of variance $1/2$ as follows.  Write the law of $Y$ as $\mathbb{P}_{Y}$ and the law of Brownian motion with variance $1/2$ as $\mathbb{P}_{BM}$. Then, the Radon-Nikodym derivative $\frac{d \mathbb{P}_{Y}}{d \mathbb{P}_{BM}}(\beta)$ for $\beta\in C([0,1])$ is proportional to
\begin{equation}
e^{-2 v \beta(1)}\left(\int_0^1 e^{-2\beta(t)}dt\right)^{-u-v}.
\end{equation}
This description has been proven rigorously in \cite[Proposition 1.7]{BKWW} provided $u+v\geq 0$ and $\min(u,v)>-1$, and is conjectured in \cite{BLD} to hold for all values of $u$ and $v$. For $u+v=0$, $Y$ reduces to a Brownian motion with drift and one sees that $2^{-1/2} W + Y$ has the law of standard Brownian motion with drift $u=-v$ (i.e., the law of the random function on $[0,1]$ given by $X\mapsto B(X)+uX$ where $B$ is a standard variance 1 Brownian motion).

Let us also note one further development since this paper was originally posted. For the half-space KPZ equation, \cite{BC22} constructed what is conjectured to be the full set of stationary measures. The approach taken therein is quite different than here and proceeds through studying the half-space log-gamma polymer model. Interestingly, the above sort of structure for the stationary measure (as a reweighing of simple random-walk type objects) can be seen directly already at the level of the log-gamma polymer. As such, it would be interesting to find a more direct proof of the above open KPZ stationary measure description in which this structure is already apparent at the level of a discretization of the process.

\medskip

The aim of the rest of this introduction is to state our  main result, Theorem \ref{thm_main}. This requires introducing two other Markov processes --- the open ASEP on an interval and the continuous dual Hahn process. We proceed with those first.

\subsection{The open ASEP}\label{sec:intrasep}
Fix six parameters $q\in [0,1)$, $\alpha, \beta>0$, $\gamma, \delta\geq 0$, and $N\in \Z_{\geq 1}$.  Open ASEP is a continuous-time Markov process taking values in the state space $\{0,1\}^{\llbracket 1,N\rrbracket}$. The state at time $t$ is denoted by $\tau(t)=\big(\tau_1(t),\ldots, \tau_N(t)\big)$;  sites $x\in \{1,\ldots N\}$ where $\tau_x(t)=1$ are said to be {\em occupied} by a particle, and those where $\tau_x(t)=0$ are {\em unoccupied}. The process is defined via the rates of its transitions as follows:
Particles jump left or right from occupied sites to unoccupied sites within $\llbracket 1,N\rrbracket$ at rate $q$ or $p=1$, respectively; at the left boundary, sites become occupied (if presently unoccupied) at site $1$ at rate $\alpha$ and become unoccupied (if presently occupied) at rate $\gamma$; at the right boundary, particles become occupied (if unoccupied) at site $N$  at rate $\delta$ and become unoccupied (if occupied) at rate $\beta$. All moves are from independent exponential clocks. As it is easy to write down the generator of this process from the above description we do not labor this point (we also do not make use of this).

The open ASEP has a unique stationary probability measure $\pi^{\mathrm{ASEP}}_N (\tau)$, with the dependence on the other parameters $q, \alpha, \beta, \gamma, \delta$ implicit. In other words, $\pi^{\mathrm{ASEP}}_N$ uniquely satisfies $\mathcal{L} \pi^{\mathrm{ASEP}}_N=0 $ where $\mathcal{L}$ is the generator of open ASEP. Note that in this paper we only use the term {\em stationary} to refer to this sort of temporal statistical stationarity, not any sort of spatial shift-invariance (which anyway does not make much sense in this context).
We will denote the expectation of a function $f:\{0,1\}^{\llbracket 1,N\rrbracket}\to \R$ under $\pi_{N}(\tau)$ by
\begin{equation}\label{eq:langrang}
\big\langle f\big\rangle_N:= \!\!\!\!\sum_{\tau\in \{0,1\}^{\llbracket 1,N\rrbracket}} f(\tau)\cdot \pi^{\mathrm{ASEP}}_N(\tau).
\end{equation}
This is in accordance with notation used in much of the physics and mathematics literature around this model.
For $\tau$ (defined on some probability space) distributed according to $\pi^{\mathrm{ASEP}}_N$ we write
\begin{equation}\label{eq:hNdef}
h_N(x)=\sum_{i=1}^{x} (2\tau_i-1)
\end{equation}
for the associated random height function. Extend this to a continuous function on $[0,N]$ by linear interpolation.

From the occupation variable process $\tau(t)$ defined above, we define the ASEP height function Markov process $h_N(t,x)$. The subscript indicates the lattice size $N$, and the time $t$ and spatial location $x$ are both arguments. The dependence of $h_N(t,x)$ on the other parameters $q,\alpha,\beta,\gamma,\delta$ will be generally suppressed.
The height function is defined for $t\geq 0$ and $x\in \llbracket 0, N\rrbracket$ as
\begin{equation}\label{eq_hN}
h_N(t,x) := h_N(t,0) + \sum_{i=1}^{x} (2\tau_i(t)-1),\qquad  h_N(t,0):=-2\mathcal{N}_N(t)
\end{equation}
where the net current $\mathcal{N}_N(t)$ equals the number of
particles to enter into site $1$ from the left reservoir minus the
number of particles to exit from site $1$ into the left reservoir, up
to time $t$. The height function definition is extended to $x\in
[0,N]$ by linear interpolation---see  Figure \ref{Fig:openASEP}. Just
as for the open KPZ equation, the open ASEP height function process
will not have a stationary measure. However, its the increment process
(which is essentially just the $\tau$ process) will: If $h_N(x)$ is
randomly chosen as in \eqref{eq:hNdef}, then starting the ASEP height
function process from that initial data we immediately get that the law of $h_N(t,\cdot)-h_N(t,0)$ as a function of $\cdot$ will be $t$-independent.

\begin{figure}[h]
\centering
\scalebox{0.6}{\includegraphics{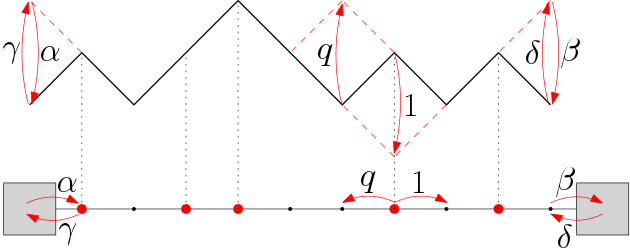}}
 \captionsetup{width=\linewidth}
\caption{Open ASEP with system size $N=10$ and its height function
  $h_N(t, x)$. Red arrows indicate some possible moves with rates labeled.}
\label{Fig:openASEP}
\end{figure}

%

It is convenient to work with a particular parameterization for open
ASEP. Consider  the functions
\begin{equation}\label{eqnkappa}
\kappa^{\pm}(q,x,y):=\dfrac{1}{2
  x}\left(1-q-x+y\pm\sqrt{(1-q-x+y)^2+4 xy} \right).
\end{equation}
Let us define $(q,A,B,C,D)$ in terms of $(q,\alpha,\beta,\gamma,\delta)$ as
\begin{equation}\label{eq_ABCD}
A=\kappa^+(q,\beta,\delta), \quad B=\kappa^-(q,\beta,\delta),\quad C=\kappa^+(q,\alpha,\gamma) ,\quad D=\kappa^-(q,\alpha,\gamma)
\end{equation}
Given $q$, \eqref{eq_ABCD} provides a bijection between $\big\{(\alpha,\beta,\gamma,\delta):  \alpha, \beta>0,\, \gamma, \delta\geq 0\big\}$ and $\big\{(A,B,C,D): A, C>0,\, B,D\in (-1,0]\big\}$.

In order to make contact with the open KPZ equation we have to assume that the rates vary with $N$ and two parameters $u,v\in \R$ as follows.
\begin{assumption}\label{assumption1}
Let
\begin{equation}\label{eq:assumerates}
q=\exp\Big(-\frac{2}{\sqrt{N}}\Big), \quad A=q^{v},\quad B=-q,\quad C=q^{u},\quad D=-q.
\end{equation}
Solving for $\alpha,\beta,\gamma$ and $\delta$ in terms of $q,u$ and $v$,
\begin{equation}\label{eq:explicitparam}
\alpha = \frac{1}{1+q^u},\qquad \beta = \frac{1}{1+q^v},\qquad \gamma = \frac{q^{u+1}}{1+q^u},\qquad \delta = \frac{q^{v+1}}{1+q^v}.
\end{equation}
Let $h_N(X)$ be a random height function defined as in \eqref{eq:hNdef} whose law is the push-forward of the ASEP stationary measure  $\pi^{\mathrm{ASEP}}_N$. Define a diffusive scaling of $h_N(X)$, keeping track of $N$ in the super-script and keeping track of the parameters $u$ and $v$ in the subscript: For $X\in [0,1]\cap \Z/N$ let
\begin{equation}\label{eqn_hscale}
H^{(N)}_{u,v}(X) := N^{-1/2} h_N(NX)
\end{equation}
and then linear interpolate to all $X\in [0,1]$. Finally, let $\mu^{(N)}_{u,v}$ denote the law of $H^{(N)}_{u,v}$, i.e., the stationary measure itself.
\end{assumption}

The scaling of $q$ in \eqref{eq:assumerates} in conjunction with the
height function scaling in \eqref{eqn_hscale} is called {\em weak
  asymmetry scaling}. The conditions on $A,B,C$ and $D$ in
\eqref{eq:assumerates} correspond to $\alpha,\beta,\gamma$ and
$\delta$ which satisfy {\em Liggett's condition} \cite{Liggett75, Liggett77} that $\alpha +\gamma/q =1$ and $\beta+\delta/q=1$. Moreover, from \eqref{eq:explicitparam} we see that $\alpha,\beta,\gamma$ and $\delta$ satisfy {\em triple point scaling}, which means that as $N\to \infty$,
\begin{equation}
\label{eq_trip_exp}
\begin{aligned}
&\alpha = \frac{1}{2} + \frac{u}{2} N^{-1/2} + o(N^{-1/2}), &\beta = \frac{1}{2} + \frac{v}{2} N^{-1/2} +o(N^{-1/2}),\\
&\gamma = \frac{1}{2} - \frac{u}{2} N^{-1/2} + o(N^{-1/2}), &\delta = \frac{1}{2} - \frac{v}{2} N^{-1/2} + o(N^{-1/2}).
\end{aligned}
\end{equation}

\subsection{Continuous dual Hahn process}\label{sec:introCDH}
The open KPZ stationary measures that we construct can be characterized via a duality with another stochastic process which we call the {\em continuous dual Hahn process} (denoted below by $\T_s$). This is a special limit of the Askey-Wilson processes constructed by Bryc and Weso{\l}owski \cite{BW1}; see Section \ref{sec:AW}. The continuous dual Hahn process depends on two parameters $u,v\in\R$ which are assumed throughout to satisfy the relation $u+v>0$. The definition of this process is simplest (and also appears in \cite{Bryc}) when $u,v>0$ and thus for the sake of this introduction we will only define it in that case here. Section \ref{sec:tangentAWP} addresses the considerably more complicated general case of $u+v>0$.


We will only define the continuous dual Hahn process $\T_s$ for $s\in [0,\Cuv)$  (see subsequent work of \cite{BrycCDH} for an extension to all of $\mathbb{R}$) where
\begin{equation}\label{eq:Cuv}
\Cuv :=
\begin{cases}
2&\textrm{if }u\leq 0\textrm{ or }u\geq 1,\\
2u&\textrm{if }u\in(0,1).
\end{cases}
\end{equation}
Formulas become more involved outside $[0,\Cuv)$ and will not be needed.

For $u,v>0$ and $s\in [0,\Cuv)$ define a measure $\p_s$ with density given by
\begin{equation}
\p_s(r):=\dfrac{(v+u)(v+u+1)}{8\pi}\cdot
      \frac{\Big|\Gamma\Big(\dfrac{s}{2}+v+\iu\dfrac{\sqrt{r}}{2}\Big)\cdot\Gamma\Big(-\dfrac{s}{2}+u+\iu\dfrac{\sqrt{r}}{2}\Big)\Big|^2}{\sqrt{r}
        \cdot \big|\Gamma(\iu \sqrt{r})\big|^2}{\bf 1}_{r>0}.
\end{equation}
This family of infinite measures will turn out to be preserve by our Markov process and necessary in the statement of our main results.

Following \cite{doi:10.1137/0511064}, we define the orthogonality probability measure for the continuous dual Hahn orthogonal polynomials as follows: For $a\in \R$ and $b=\bar c\in \C\setminus \R$ with $\textup{Re}(b)=\textup{Re}(c)>0$ let
  \begin{equation}
\CDHtransition(x; a, b, c):=\dfrac{1}{8\pi}\,\cdot\,\dfrac{\bigg\rvert \Gamma\left(a+\iu  \dfrac{\sqrt x}{2}\right)\cdot\Gamma\left(b+\iu
  \dfrac{\sqrt x}{2}\right)\cdot\Gamma\left(c+\iu  \dfrac{\sqrt x}{2}\right)\bigg \rvert^2}{\Gamma(a+b)\cdot \Gamma(a+c)\cdot\Gamma (b+c)\cdot \sqrt x\cdot\left\rvert \Gamma\left(\iu \sqrt x\right)\right\rvert^2}{\bf 1}_{x> 0}.
  \end{equation}
For $s,t\in [0,\Cuv)$ with $s<t$ and $m,r\in (0,\infty)$, define a measure $\p_{s,t}(m,\cdot)$ with density in $r$ by
\begin{equation}
\p_{s, t}(m, r):=\CDHtransition\left(r; u-\dfrac{t}{2}, \dfrac{t-s}{2}+\iu\dfrac{\sqrt m}{2}, \dfrac{t-s}{2}-\iu\dfrac{\sqrt m}{2}\right).
\end{equation}

The continuous dual Hahn process (with $u,v>0$) $\{\T_s\}_{s\in [0,\Cuv)}$ is the Markov process
with state-space $\R_{>0}$ and transition probabilities given by $\p_{s,t}$. Lemma \ref{lem:consistancy} verifies that the $\p_{s,t}$ satisfy the Chapman-Kolmogorov equation. That lemma also verifies that if $\T_0$ is started according to the infinite distribution $\p_0$ then the infinite distribution of $\T_s$ is $\p_s$ for all $s\in [0,\Cuv)$.

\subsection{Statement of the main result}
For $d\in \Z_{\geq 1}$ we will assume  that
\begin{align}
\vec{X} &= (X_0,\ldots, X_{d+1})\textrm{ where }0=X_0<X_1<\cdots<X_d\leq X_{d+1}=1,\\
\vec{c} &= (c_1,\ldots, c_d)\textrm{ where }c_1,\ldots, c_d>0,\label{eq:Xcs}\\
\vec{s} &= (s_1> \cdots>s_{d+1})\textrm{ where }s_k = c_k+\cdots +c_d \textrm{ and }s_{d+1} = 0.
\end{align}

We are now prepared to state our main theorem. Part
\eqref{it:tightness} shows the existence of stationary measures for the open KPZ increment process as limits of stationary measures for the WASEP increment process (i.e., with parameters scaled as in \eqref{eq:assumerates}). Part  \eqref{it:coupling} records a coupling between these stationary measures in which height differences are stochastically ordered relative to a certain ordering of the boundary parameters. Part (\ref{it:brownian}) records the simple Brownian case for the stationary measure which occurs when $u+v=0$. Part \eqref{it:duality}
records a duality which comes from switching the roles of $u$ and $v$. Part \eqref{it:explicit} provides a unique characterization of
the WASEP derived stationary measures for the open KPZ increment process provided that $u+v>0$ (the $u+v=0$ case was already addressed in part \eqref{it:brownian}). This characterization is given by a remarkable duality formula which relates the Laplace transform of the KPZ stationary measure to  the continuous dual Hahn process introduced above and in greater generality in Section \ref{sec:tangentAWP}. This shows that provided $u+v>0$, there is a unique limit point in part \eqref{it:tightness}.

Section \ref{sec:discussion} gives an outline of the key ideas and logic that go into the proof of these results.
%
%

\begin{theorem}\label{thm_main}
Assume that open ASEP satisfies Assumption \ref{assumption1} for all $N$.
\begin{enumerate}[leftmargin=*]
  \item\label{it:tightness} {\bf Tightness and construction of
      WASEP-stationary measures:}  For any $u,v\in \R$, the
    $N$-indexed sequence of laws of $\mu^{(N)}_{u,v}$ (recall from Assumption \ref{assumption1}) are tight
    in the space of measures on $C([0,1])$ and all
    subsequential limits $\mu_{u,v}$ are stationary measures for the open KPZ increment process and are almost surely H\"older $\alpha$ for all $\alpha<1/2$. Call any such subsequential limit a {\em WASEP-stationary measure} for the open KPZ increment process.
  \item\label{it:coupling} {\bf Coupling:} For any $M\in \Z_{\geq 2}$, $u_1\leq \cdots
    \leq u_M$ and $v_1\geq \cdots \geq v_M$, assume that $\left\{\mu_{u_i,v_i}\right\}_{i=1}^{M}$ are WASEP-stationary measures that arise in part \eqref{it:tightness} along the same subsequence as $N\to \infty$. Then there exists a probability space which supports $M$ random functions   $\left\{H_{u_i,v_i}\right\}_{i=1}^{M}$ in $C([0,1])$ such that marginally each $H_{u_i,v_i}$ has distribution $\mu_{u_i,v_i}$ and such that for all $0\leq X\leq X'\leq 1$ and $1\leq i\leq j\leq M$,
    \begin{equation}\label{eqnHcoupling}
    H_{u_i,v_i}(X')-H_{u_i,v_i}(X)\leq H_{u_j,v_j}(X')-H_{u_j,v_j}(X).
    \end{equation}
    \item\label{it:brownian} {\bf Brownian case:} For $u+v=0$, there is a unique WASEP-stationary measure $\mu_{u,-u}$
    for the open KPZ increment process that coincides with the law of standard Brownian motion of drift $u=-v$.
    \item\label{it:duality} {\bf Duality:}  For any $u,v\in \R$ let $\mu_{u,v}$ and $\mu_{v,u}$ be a pair of
      WASEP-stationary measures for the open KPZ equation which arise
      in part \eqref{it:tightness} along the same subsequence as $N\to \infty$. Then the corresponding stochastic processes $X\mapsto H_{u,v}(X)$ and $X\mapsto H_{v,u}(1-X)-H_{v,u}(1)$ have the same law in $C([0,1])$.
    \item\label{it:explicit}{\bf Explicit Laplace
        transform characterization:} For $u,v\in \R$ with $u+v>0$, the measures $\mu^{(N)}_{u,v}$ converge to a unique limit  $\mu_{u,v}$  as $N\to \infty$ (hence there is a unique WASEP-stationary measure). This limit $\mu_{u,v}$ is supported on $C([0,1])$ and is determined by its multi-point  Laplace transform formula:
        For any $d\in \Z_{\geq 1}$, $\vec{X},\vec{c}$ and $\vec{s}$ as in \eqref{eq:Xcs}, provided $s_1<\Cuv$, see \eqref{eq:Cuv},
        \begin{align}\label{eq_mainthmformphi}
    \displaystyle\E\Big[e^{-\sum\limits_{k=1}^{d} c_k H_{u,v}(X_k)}\Big] = \frac{\displaystyle\E\Big[e^{\frac{1}{4}\sum\limits_{k=1}^{d+1} (s_k^2-\T_{s_k})(X_k-X_{k-1})}\Big]}{\displaystyle\E\Big[e^{-\frac{1}{4}\T_0}\Big]} =:\phi_{u,v}(\vec{c},\vec{X})\qquad
\end{align}
where on the left-hand side $H_{u,v}$ has law $\mu_{u,v}$ and on the right-hand side where $\T_s$ is the continuous dual Hahn process started with $\T_0$ according to the infinite distribution $\p_0$ (see Section \ref{sec:introCDH} and \ref{sec:tangentAWP}). In particular, this implies that $H^{(N)}_{u,v} \Rightarrow H_{u,v}$ as stochastic processes in $C([0,1])$.

\end{enumerate}
\end{theorem}

As remarked earlier, the Laplace transform formula can be
inverted. This was achieved after the posting of this paper by
\cite{BKWW} and \cite{BLD}, see also \cite{BKuz} where in particular the
relationship between the results in \cite{BKWW} and \cite{BLD} is discussed. The inversion relies on the spectral decomposition of the heat kernel with an exponential potential (known as Liouville quantum mechanics in physics). The condition $u+v>0$ that we assume corresponds to the {\em fan region} for open ASEP/KPZ. While we do not currently have formulas for $u+v<0$ (the {\em shock region}), the description given in \cite{BLD} offers a plausible conjecture for the stationary measure with those parameters. With our methods, it should also be possible to access the stationary measure for the open KPZ increment process on an interval $[0,M]$ for any $M>0$. Taking the $M\to \infty$ limit should make contact with the KPZ equation in a half-space. See \cite{BLD, BKLD20} for some discussion on this limit procedure and \cite{BC22} for an alternative approach to construct the half-space KPZ increment process stationary measures.

Another important remark is that while our tightness result implies existence of  stationary measures for the open KPZ  increment process, it does not imply uniqueness. Even for parameters where we prove that WASEP-stationary measures are unique (i.e. uniqueness of the limit points of the scaled open WASEP stationary measures), we do not rule out the existence of other stationary measures for the open KPZ increment process with the same boundary parameters. However, based on related results in the literature, we conjecture that the stationary measures we have constructed are unique for all choices of $u$ and $v$.

\begin{conjecture}\label{conj_unique}
Fix any $u,v\in \R$. Consider any two random functions $H_0,\tilde H_0\in C([0,1])$, supported on the same probability space. On this probability space, define a space-time white noise $\xi$ with the time variable ranging over $\R$ and let $H(T,X;-T_0)$ and $\tilde H(T,X;-T_0)$ denote the solutions to the open KPZ equation started at time $-T_0$ with initial data $H_0$ and $\tilde H_0$, respectively. Then the following {\em one force one solution principle} holds: For any $S<S'$ the random functions $(T,X)\mapsto H(T,X)-H(T,0)$ and $(T,X)\mapsto \tilde H(T,X)-\tilde H(T,0)$ in $C([S,S'],C([0,1]))$ converge almost surely to the same limit as $T_0\to \infty$. In particular, for any fixed $u,v\in R$ there exists a unique stationary measure for the open KPZ increment process.
\end{conjecture}
For the KPZ increment process with periodic boundary conditions \cite{hairer2018} showed uniqueness of the Brownian bridge stationary measure while \cite{gubinelli2018infinitesimal} constructed the infinitesimal generator and estimated its spectral gap, establishing $L^2$ exponential ergodicity. Further, \cite{rosati2019synchronization} demonstrated the one force one solution principle. Let us also mention related work of \cite{EKMS,BCKFullLine,alex2019stationary} for the stochastic Burgers equation, and work on the mixing time of open ASEP \cite{PhysRevLett.95.240601,gantert2020mixing,Schmid,LabbeLacoin,bufetov2020cutoff}.

Finally, we remark that Theorem \ref{thm_main} \eqref{it:coupling} and
\eqref{it:brownian} combine to show that increments of $H_{u,v}\in C([0,1])$ with law $\mu_{u,v}$ (for any WASEP-stationary measure coming for Theorem \ref{thm_main} \eqref{it:tightness}) are stochastically sandwiched between Brownian motions of different drifts. Take $M=3$ and let $u_1=-v, u_2=u, u_3=u$ and $v_1=-u, v_2=v, v_3=v$. Then Theorem \ref{thm_main} \eqref{it:brownian} implies that along every subsequence of $N\to \infty$, $\mu^{(N)}_{u_1,v_1}$ converges to  $\mu_{u_1,v_1}$ which is the law of a standard Brownian motion $B_{-v}$ of drift $-v$; similarly $\mu^{(N)}_{u_3,v_3}$ converges $\mu_{u_3,v_3}$ which is the law of a standard Brownian motion $B_u$ of drift $u$. There is a subsequence along which $\mu^{(N)}_{u_2,v_2}$ converges to the limit $\mu_{u_2,v_2}$. Thus,  Theorem \ref{thm_main} \eqref{it:coupling}  implies that for all $0\leq X\leq X'\leq 1$,
\begin{equation}\label{eq:browniancomparison}
B_{-v}(X')-B_{-v}(X) \leq H_{u,v}(X')-H_{u,v}(X) \leq B_{u}(X')-B_{u}(X)
\end{equation}
with $B_{-v},H_{u,v},B_{u}$ are all random $C([0,1])$ functions defined on a common probability space with marginals given by $\mu_{u_1,v_1}$, $\mu_{u_2,v_2}$, and $\mu_{u_3,v_3}$, respectively.
The Brownian case in Theorem \ref{thm_main} follows easily from the known fact that Bernoulli product measure is the stationary measure for open ASEP when $u+v=0$. This is the only case when the stationary measure for open ASEP is simple and of product form. For general parameters it is quite complicated. For ASEP on the full line or torus, the stationary measure is product Bernoulli so for full line \cite{FQ14} and periodic \cite{hairer2018} KPZ increment processes the stationary measure is Brownian (two-sided Brownian motion with general drift, or Brownian bridge with any fixed height shift).

\subsubsection*{Outline}
Section \ref{sec:discussion} reviews the key ideas in the proof of Theorem \ref{thm_main}.
The proofs of Theorem \ref{thm_main} \eqref{it:tightness}-\eqref{it:duality} are in Section \ref{sec_thm123} and rely on Section \ref{sec_intro_WASEP} (the weak asymmetry scaling under which open ASEP height function process converges to the open KPZ equation height function process) and Section \ref{sec_couplingoverall} (coupling results for open ASEP).
The proof of Theorem \ref{thm_main} \eqref{it:explicit} is given in Section \ref{sec:proof5}. The starting point in this proof (given in Section \ref{sec:AW}) is Corollary \ref{cor:ASEPgen} (see also Proposition \ref{BrycWangFluctuations}  and \cite[Theorem 1]{BW1}) which relates the generating function for the open ASEP stationary measure to the Askey-Wilson process. Section \ref{sec:tangentAWP} defines the continuous dual Hahn process which arises as a special limit of the Askey-Wilson process. The main calculation in the proof of Theorem \ref{thm_main} \eqref{it:explicit} is Proposition \ref{prop_ASEP_gen_function} which computes the limit of the open ASEP stationary measure generating function. Combining this with some of the results in Theorem \ref{thm_main} \eqref{it:tightness}-\eqref{it:duality} we show weak convergence and that the limit of the open ASEP formula gives the open KPZ Laplace transform formula claimed in the theorem. The proof of this limit is given in Section \ref{sec:propASEPproof} and relies heavily on precise asymptotics for $q$-Pochhammer symbols. These asymptotics are stated as Proposition \ref{factorials} and proven in Section \ref{sec:factorials}.

\subsubsection*{Notation}
Lower versus upper case variables will refer to discrete versus continuous objects, respectively. For integers $a\leq b$, let $\llbracket a,b\rrbracket:= \{a,\ldots, b\}$ and  $\Z_{\geq a}=\Z\cap [a,\infty)$ (and likewise for $\Z$ replaced by $\R$ and $\geq$ replaced by $>$, $\leq $ or $<$).
We will use the standard notation for the Pochhammer and $q$-Pochhamer symbols: For $j\in \Z_{\geq 0}$ and $x\in \R$, define $[x]_j:=(x)(x+1)\cdots(x+j-1)$ with the convention that $[x]_0:=1$. For multiple arguments $x_1,\ldots x_n\in \R$, define $[x_1, \ldots, x_n]_j:=[x_1]_j\cdots [x_n]_j$. For $a, q\in \C$, with $|q|<1$, and $j\in \Z_{\geq 0}\cup\{\infty\}$, define $(a; q)_j:=(1-a)(1-a q)\cdots(1-a q^{j-1})$ and $(a_1, \dots, a_n; q)_j=(a_1;q)_j\cdots (a_j;q)_j$. We will often omit the dependence on $q$ and write $(a)_j$ or $(a_1,\dots,a_n)_j$. We also use the notation $\Gamma(x_1,\ldots,x_n):= \Gamma(x_1)\cdots \Gamma(x_n)$. We will denote $\iu := \sqrt{-1}$.


\section{Key ideas in proving Theorem \ref{thm_main}} \label{sec:discussion}
\subsection{Key ideas in proving Theorem \ref{thm_main} \eqref{it:tightness}-\eqref{it:duality}}
The open ASEP height function process converges to the open KPZ equation height function process under suitable weak asymmetry and triple point scaling. This result is basically contained in \cite{CS,Parekh} (see Section \ref{sec_intro_WASEP}). To apply it to the open ASEP height function stationary measure, we need to verify H\"older bounds on the exponential of $\hscale_{u,v}(X)$. These can be deduced from the following considerations.
When $u+v=0$, the open ASEP stationary measure is product Bernoulli. This implies tightness of the height function $\hscale_{u,-u}$ and H\"older bounds for it, and yields Theorem \ref{thm_main} \eqref{it:brownian}. To move to general $u,v$ we use the fact that here exists an attractive coupling between versions of open ASEP with different boundary rates (a finite $N$ version of Theorem \ref{thm_main} \eqref{it:coupling}). This coupling implies that the increments of the stationary height function are bounded above and below by random walk increments (by appealing to the $u+v=0$ result). This yields H\"older bounds for all $u$ and $v$. 
\subsection{Key ideas in proving Theorem \ref{thm_main} \eqref{it:explicit}}
The starting point for our work
is the {\em matrix product ansatz}, introduced by Derrida, Evans, Hakim and Pasquier \cite{Derrida_1993}, which describes the stationary measure in terms of certain non-commuting operator products.
Useful (infinite) matrix representations for these operators related
to Askey-Wilson polynomials \cite{AskeyWilson, KS2}.  Jacobi matrices were discovered by Uchiyama, Sasamoto and Wadati \cite{Uchiyama_2004}; based on a slightly more general matrix representation, Corteel and Williams \cite{corteel2011} developed a combinatorial description for the open ASEP stationary measure in terms of tableaux combinatorics. So far these formulas for the stationary measure have not been used for the type of asymptotics we need to perform in order to access the KPZ equation, though \cite{Uchiyama_2004} did perform other interesting asymptotics.

More recently, relying on the work of Uchiyama, Sasamoto and Wadati \cite{Uchiyama_2004}, Bryc and Weso{\l}owski \cite{BW1} discovered a
way to rewrite the Askey-Wilson Jacobi matrix solution to the matrix product ansatz in terms of the {\em Askey-Wilson processes}. These
Markov processes were introduced earlier in \cite{BW2} in relation to quadratic harnesses. Askey-Wilson polynomials are orthogonal martingale polynomials for these processes. The following remarkable identity is our starting point.

\begin{proposition}[Theorem 1 of \cite{BW1}]\label{BrycWangFluctuations} Let $ \big \langle \cdot  \big\rangle_N$
denote the expectation with respect to the stationary measure $\pi^{\mathrm{ASEP}}_N$ of open ASEP parameterized by $q$ and $(A,B,C,D)$ as in \eqref{eq_ABCD}.
Assume that $AC<1$. Then for $0 < t_1 \leq t_2 \leq\cdots\leq t_n$, the joint generating function of the stationary measure for open ASEP can be expressed as
\begin{equation}\label{eq:key_identity}
\bigg \langle\prod_{j=1}^N t_j^{\tau_j}  \bigg
\rangle _N=\dfrac{\E\left[\prod_{j=1}^N\left(1+t_j+2\sqrt t_j
      \Y_{t_j}\right) \right]}{2^N \E
  \left [\left (1+\Y_1 \right )^N\right ]},
\end{equation}
where $\{\Y_t\}_{t\geq 0}$ is the Askey-Wilson process with parameters $(A, B, C, D, q)$ defined in Section \ref{sec:AW}.
\end{proposition}

An immediate corollary \cite[Section 4.3]{BW} of this is the  multi-point Laplace transform formula for $H^{(N)}_{u,v}(X)$ (defined from $\tau$ by combining \eqref{eqn_hscale} and \eqref{eq:hNdef}) under the stationary measure  $\pi^{\mathrm{ASEP}}_N$.

\begin{corollary}\label{cor:ASEPgen}
As in \eqref{eq_ABCD}, let $ \big \langle \cdot  \big\rangle_N$ denote the expectation with respect to the stationary measure $\pi^{\mathrm{ASEP}}_N$ of open ASEP parameterized by $q$ and $(A,B,C,D)$. Assume that $AC<1$.
For any $d\in \Z_{\geq 1}$, let $\vec{X},\vec{c}$ and $\vec{s}$ be as in \eqref{eq:Xcs}, $\tilde{c}\in\R$. Then
\begin{align}\label{eq:pihn}
\phi^{(N)}(\vec{c},\tilde{c},\vec{X}):= \left\langle e^{-\sum_{k=1}^d c_k H^{(N)}_{u,v}(X^{(N)}_k) -\tilde{c}H^{(N)}_{u,v}(1)} \right\rangle_N \!\!\!\!= \dfrac{\E\left[\,\,\prod\limits_{k=1}^{d+1}\left(\!\!\cosh\left(\tilde{s}^{(N)}_k\right)+
      \Y_{e^{-2 \tilde{s}^{(N)}_k}}\!\!\right)^{n_k-n_{k-1}} \right]}{\E  \left [\left (1+\Y_1 \right )^N\right ]}\,\,\,\,\,\,
\end{align}
where $\tilde{s}^{(N)}_k = N^{-1/2} (s_k + \tilde{c})$ for $k\in
\llbracket 1,d+1\rrbracket$, $X_k^{(N)}:= N^{-1} \lfloor NX_k\rfloor$,
$n_k=\lfloor NX_k\rfloor$  for $k\in \llbracket 1,d\rrbracket$, and
$\Y_s$ is the Askey-Wilson process (Definition \ref{def:AWprocess})
with parameters $A,B,C,D,q$ matching those of the ASEP we are
considering, and with marginal distribution $\pi_s$ (see
\eqref{eq:def_pi}) at all times $s$. When $\tilde{c}=0$ we will write $\phi^{(N)}(\vec{c},\vec{X})$ instead of $\phi^{(N)}(\vec{c},0,\vec{X})$ and use $s_k$ instead of $\tilde{s}_k$ on the right-hand side of \eqref{eq:pihn}.
\end{corollary}
Note that the restriction that the $c_k$ are strictly positive comes from the increasing nature of the $t$'s from Proposition \ref{BrycWangFluctuations}. We do not know how to analytically continue to general $c_k$. However, for our purposes it is sufficient to work with the positive $c_k$ and also to assume $\tilde{c}=0$. In that case we when $\tilde{c}=0$ we write $\phi^{(N)}(\vec{c},\vec{X})$ instead of $\phi^{(N)}(\vec{c},0,\vec{X})$.

Using this corollary we see that the finite $N$ Laplace transform can be written as (see also \eqref{eq:phineq})
\begin{equation}\label{eq:introphin}
\phi^{(N)}_{u,v}(\vec{c},\vec{X}) = \frac{\tilde\phi^{(N)}_{u,v}(\vec{c},\vec{X})}{\tilde\phi^{(N)}_{u,v}(\vec{0},\vec{X})},\quad
\tilde\phi^{(N)}_{u,v}(\vec{c},\vec{X}):=\E \left[N^{u+v} \G^{(N)}\Big((\Yh^{(N)}_{s_1},\ldots ,\Yh^{(N)}_{s_{d+1}}); \vec{c};\vec{X}\Big)\right].
\end{equation}
Here we have set (see \eqref{eq:rescaling}) $\Yh_s^{(N)}:=2 N \left(1-\Y_{q^{s}}\right)$  and (see \eqref{eq:gN})
\begin{equation}
\G^{(N)}(\vec{r}; \vec{c};\vec{X}):= {\bf 1}_{\vec{r}\in \R^{d+1}_{\leq 4N}} \cdot 2^{-N}\prod_{k=1}^{d+1}\left(\cosh\left (\tfrac{s_k}{\sqrt N}\right)+1-\tfrac{r_k}{2N}\right)^{N(X^{(N)}_k-X^{(N)}_{k-1})}.
\end{equation}
It is  easy to check that as a  function in $\vec{r}$, $\G^{(N)}(\vec{r}; \vec{c};\vec{X})$ converges point-wise to (see \eqref{eq_G})
\begin{equation}
    \G(\vec{r}; \vec{c};\vec{X}) := \exp\left(\frac{1}{4}\sum_{k=1}^{d+1} (s_k^2-r_k)(X_k-X_{k-1})\right).
\end{equation}
The overwhelming majority of the work is thus left to show that in an appropriate strong sense
\begin{equation}\label{eq:introYHN}
N^{u+v} \mathrm{Law}\big(\Yh^{(N)}_{s_1},\ldots ,\Yh^{(N)}_{s_{d+1}})\Rightarrow \mathrm{Law}(\T_{s_1},\ldots ,\T_{s_{d+1}})
\end{equation}
where $\T_s$ is the continuous dual Hahn process started at $\T_{s_1}$ according to the infinite measure $\p_{s_1}$. In fact, what is really needed is the convergence of
\begin{equation}
\label{eq:introphinconverg}
\tilde\phi^{(N)}_{u,v}(\vec{c},\vec{X}) \rightarrow \tilde\phi_{u,v}(\vec{c},\vec{X}):=\E \left[ \G\Big((\T_{s_1},\ldots ,\T_{s_{d+1}}); \vec{c};\vec{X}\Big)\right]
\end{equation}
from which the convergence $\phi^{(N)}_{u,v}(\vec{c},\vec{X})\to \phi_{u,v}(\vec{c},\vec{X})$ readily follows.

The $\Yh^{(N)}$ process can be thought of as a secant process to the Askey-Wilson process. If one conditions on its value at $s_1$ then the distribution of its values at times $s_2,\ldots$ are given by the product of bona-fide transition probabilities. On the other hand, the marginal distribution that we have assumed of the Askey-Wilson process requires rescaling by $N^{u+v}$ to have a non-trivial limit as $N\to \infty$. As a result, in the limit $N\to \infty$, the marginal distribution of $\Yh^{(N)}$ ends up becoming an infinite measure. This means that in order to establish the limit \eqref{eq:introphinconverg} we must establish both point-wise convergence and rather strong bounds on the marginal and transitional measures used to define the $\Yh^{(N)}$ process, so as to be able to apply the dominated convergence theorem. Owing to the fact that these measures are written in terms of $q$-Pochhammer functions, this analysis ends up involving rather refined asymptotics of $\log(\pm q^z;q)_{\infty}$ for $q=e^{-\kappa}$ with both $\kappa$ and $z$ varying in certain potentially unbounded ranges.

Before delving into those asymptotics, let us explain one final wrinkle in the proof of Theorem \ref{thm_main} \eqref{it:explicit}. The argument described above ultimately shows (see Proposition \ref{prop_ASEP_gen_function}) that the Laplace transform $\phi^{(N)}_{u,v}(\vec{c},\vec{X})$ of the finite dimensional marginals of $\mu^{(N)}_{u,v}$ converge to a limit $\phi_{u,v}(\vec{c},\vec{X})$ as $N\to \infty$, provided the spectral variables $\vec{c}\in (0,\Cduv)^d$. However, this does not immediately means that $\mu^{(N)}_{u,v}$ converges to a limit itself. Indeed, if we knew independently that $\phi_{u,v}(\vec{c},\vec{X})$ was the Laplace transform (in the $\vec{c}$ variables) of some probability distribution $\mu_{u,v}$ then this would imply that $\mu^{(N)}_{u,v}\Rightarrow \mu_{u,v}$. This is due to a generalization (see \cite[Appendix A]{BW}) of an old result of \cite{Curtiss}. The work of \cite{BKWW} (subsequent to our current paper) established this property of  $\phi_{u,v}(\vec{c},\vec{X})$ provided $\min(u,v)>-1$ (in addition to the ongoing assumption here that $u+v>0$). That, however, does not cover the full range of $u$ and $v$.

In any case, prior to this inversion work we developed a rather different route to show that  $\phi_{u,v}(\vec{c},\vec{X})$ is the Laplace transform of some  probability distribution $\mu_{u,v}$ and hence that $\mu^{(N)}_{u,v}\Rightarrow \mu_{u,v}$.
Our approach uses some of the additional probabilistic information about the $\mu^{(N)}_{u,v}$ measures provided to us by the earlier parts of Theorem \ref{thm_main}. Namely, we use the tightness of $\mu^{(N)}_{u,v}$ (from Theorem \ref{thm_main} \eqref{it:tightness} and uniform
control over exponential moments $\mu^{(N)}_{u,v}$  (from Theorem \ref{thm_main} \eqref{it:coupling} and \eqref{it:brownian}) to show that
$\phi_{u,v}(\vec{c},\vec{X})$ coincides on an open set with the Laplace transform of some sub(sub)sequential weak limits of $\mu{(N)}_{u,v}$. This identifies uniquely the weak limits along all subsubsequences as being the same, and hence shows convergence of the original sequence of measures. This combination of integrable (exact asymptotic calculation) and probabilistic (the tightness and coupling arguments) methods is quite powerful and allowed us to proceed where each method on its own failed to produce results.

As noted above, the proof of the Laplace transform convergence $\phi^{(N)}_{u,v}(\vec{c},\vec{X})\to \phi_{u,v}(\vec{c},\vec{X})$ constitutes the most technically demanding part of this work. The starting point for this convergence is the fact that the Askey-Wilson process marginal distributions and transition probabilities are written explicitly in terms of the Askey-Wilson orthogonality measure, which in turn is written in terms of $q$-gamma functions (i.e., certain $q$-Pochhammer symbols).

Thus, one of the key technical challenges here is to develop an explicit asymptotic expansion of $q$-Pochhammer symbols (really the $q$-gamma function) as $q\to 1$ with precise error bounds which can be controlled uniformly over all arguments. Recall that for $a, q\in \C,$ $|q|<1$, we let $(a; q)_j:=(1-a)(1-a q)\dots\left(1-a q^{j-1}\right)$ and write $(a_1, \dots, a_n; q)_j=(a_1;q)_j\dots (a_j;q)_j$. We often drop the $q$ dependence.
Let us define for $z\in \C$ and $\kappa >0$ the following functions:
\begin{align}
\AqP^+[\kappa, z]&= -\dfrac{\pi^2}{6 \kappa}-\left(z-\dfrac{1}{2}\right)\log \kappa- \log \left[\dfrac{\Gamma (z)}{\sqrt{2\pi}}\right ],\label{eq:AqPp}\\
\AqP^-[\kappa, z]&= \dfrac{\pi^2}{12\kappa}-\left(z-\dfrac{1}{2}\right)\log 2.\label{eq:AqPm}
  \end{align}
\begin{proposition}\label{factorials}
For $\kappa\in (0,1)$ let $q=e^{-\kappa}$. For $z\in\C$ and $m\in\Z_{\geq 1}$
\begin{align}
\text{ }\log (q^z; q)_\infty&=\AqP^+[\kappa, z]-\sum\limits_{n=1}^{m-1}\dfrac{B_{n+1}(z)B_{n}}{n(n+1)!}\kappa^{n}+\Error_m^+[\kappa, z],\label{eq:factorialspos}\\
\log (-q^z; q)_\infty&=\AqP^-[\kappa, z]-\sum\limits_{n=1}^{m-1}\left(2^{n}-1\right)\dfrac{B_{n+1}(z)B_{n}}{n(n+1)!}\kappa^{n}+\Error_m^-[\kappa, z].\quad\label{eq:factorialsneg}
\end{align}
where $B_{k}(z)$ and $B_{k}$ denote Bernoulli polynomials and Bernoulli numbers (Section \ref{sec:bern}).
For any $\alpha\in (0,\pi)$, $\e\in (0,1/2)$ and $b\in (m-1,m)$ there exist $C,\kappa_0>0$ such that for all $\kappa\in (0,\kappa_0)$ and all $z\in\C$ with $|\textup{Im}(z)|<\tfrac{\alpha}{\kappa}$
\begin{equation}\label{eq:errorbounds}
\big \lvert \Error^{\pm}_m[\kappa, z]\big \rvert \leq  C\left(\kappa(1+|z|)^2 + \kappa^b(1+|z|)^{1+2b+\e}\right).
\end{equation}
The bound on $\Error^{+}_m[\kappa, z]$ further holds with the condition $|\textup{Im}(z)|<\tfrac{\alpha}{\kappa}$ replaced by $|\textup{Im}(z)|<\tfrac{2\alpha}{\kappa}$. Furthermore, for any $r>0$, $\e\in (0,1/2)$ and $b\in (m-1,m)$ there exist $C,\kappa_0>0$ such that for all $\kappa\in (0,\kappa_0)$ and all $z\in\C$ with $\textup{dist}(\textup{Re}(z),\mathbb{Z}_{\leq 0})>r$, \eqref{eq:errorbounds} continues to hold.
\end{proposition}

Observe that since $e^{\pi \iu}=-1$,
\begin{equation}\label{eq:transfo}
\log(q^z;q)_{\infty}\mapsto \log(-q^z;q)_{\infty}\quad \textrm{when}\quad z\mapsto z+\frac{\pi}{\kappa}\iu.
\end{equation}
Thus, the restriction that $|\textup{Im}(z)|<\tfrac{\alpha}{\kappa}$ is quite natural and not really a restriction since we can extract asymptotics for $\log(q^z;q)_{\infty}$ for general imaginary part using the above fact in conjunction that $\log(q^z;q)_{\infty}$ remains invariant under $z\mapsto z+\frac{2\pi}{\kappa}\iu$. This invariance easily implies the claims after equation \eqref{eq:errorbounds} as corollaries of that bound with the restriction $|\textup{Im}(z)|<\tfrac{\alpha}{\kappa}$ in place.

We will, in fact, only make use of the $m=1$ case of the proposition, though we leave the general result since it is not much harder to prove and may be of subsequent use to others.

The $q$-gamma function is closely related to $(q^z;q)_{\infty}$ and given by
$
\Gamma_q(z) = (1-q)^{1-z} \frac{(q;q)_{\infty}}{(q^z;q)_{\infty}},
$
thus our result can be seen as an asymptotic result for the $q$-gamma function as well. Asymptotics of $\Gamma_q$ have been studied in a number of contexts previously, e.g. \cite{Moak, Mcintosh, Daalhuis,Z}. In all of those works (and others) the error bounds are either for $z$ fixed as $\kappa$ goes to zero, or $\kappa$ fixed as $z$ goes to infinity in some direction. To our knowledge, there has been no analysis of how these two limits balance. This balance, however, is extremely important for us since we will deal with measures that are defined with respect to these $q$-Pochhammer symbols and certain key asymptotics that we perform in Section \ref{sec:propASEPproof} will involve probing $|z|$ of order $\kappa^{-1}$, with $\kappa$ going to zero.

The proof of Proposition \ref{factorials} (Section \ref{sec:factorials}) relies on complex analytic methods often used in analytic number theory \cite{MV, R} such as the Mellin transform and the use of gamma, zeta, Hurwitz zeta and Jacobi theta functions. The formula \eqref{eq:factorialspos} can already be found in \cite[Theorem 2]{Z}, though the error bound stated there involves fixed $z$ with $\kappa$ tending to zero. In fact, the proof of \cite[Theorem 2]{Z} relies on an incorrect result, \cite[Lemma 5]{Z}, which claims that $\zeta(s,z) = \mathcal{O}(|t|^{2N+1})$ as $|t|\to \infty$ uniformly for $z$ in any compact subset of $\textup{Re}(z)>0$. Here $\zeta(s,z)$ is the Hurwitz zeta function and $s=\sigma+\iu t$ with real $\sigma >-2N$ for $N\in \Z_{\geq 1}$ and $t\in\R$. Proposition \ref{hurwitz_bound} provides a correct bound on the Hurwitz zeta function with exponential growth $e^{|\textup{Arg(z)}\cdot t|}$ and important polynomial factors which ultimately translate into our error bound above. The analysis of $(-q^z;q)_{\infty}$ is similar, though it involves the Dirichlet eta function as well.

We close this discussion by comparing our proof to that of  \cite{BW} who studied the scaling limit of the ASEP stationary measure for $q$ (and the boundary parameters) fixed, as opposed to scaling with $N$. There they also utilized the Askey-Wilson processes, but there were two major simplifications. The first was that the $q$-Pochhammer symbols that come up there are easily controlled since $q$ is not varying. This renders the asymptotics of the Laplace transform considerably simpler. The second is that the limiting Laplace transform was recognizable as the Laplace transform of a probability distribution due to their work in \cite{BRYC201877}. This avoided the need for the additional twist described above. Another difference with \cite{BW} is that they were only concerned with taking a limit of the stationary measure --- they do not show that this limit measure is stationary for some limiting Markov process. In our case, additional probabilistic/stochastic analytic work is needed to show that the limiting measures are stationary measures for the open KPZ equation height function increment process. These remarks are not meant to diminish the work of \cite{BW} but rather indicate how it serves a key starting point for this current paper which has to confront a number of additional conceptual and technical challenges.

\section{Weak asymmetry limit to the open KPZ equation}\label{sec_intro_WASEP}

The open ASEP height function process (recall from Section \ref{sec:intrasep}) converges to the Hopf-Cole solution to the open KPZ equation under the following assumptions on parameters. Here we will not necessarily assume that ASEP is started from its stationary measure, but rather allow for a very general class of initial data that satisfy some H\"older bounds.

\begin{assumption}\label{assump_wasep}\mbox{}
\begin{enumerate}[leftmargin=*]
\item {\em Weak asymmetry scaling}:
$
q= \exp\left(-\frac{2}{\sqrt{N}}\right).
$
\item {\em Liggett's condition}:
$\alpha +\gamma/q=1$ and $\beta+\delta/q=1.
$

\item {\em Triple point scaling}: For some $u, v\in \R$, as $N\to \infty$
\begin{align}
&\alpha=\frac{1}{2}+\frac{u}{2} N^{-1/2} + o(N^{-1/2}), &\beta=\frac{1}{2}+\frac{v}{2} N^{-1/2} + o(N^{-1/2}),\\
&\gamma=\frac{1}{2}-\frac{u}{2} N^{-1/2} + o(N^{-1/2}), &\delta=\frac{1}{2}-\frac{v}{2} N^{-1/2} + o(N^{-1/2}).
\end{align}


\item {\em $4:2:1$ height function scaling}: For $T\geq 0$ and $X\in [0,1]$ define
\begin{align}\label{eq_hNtx_scaled}
\hscale_{u,v}(T,X) &:= N^{-1/2} h_N(\tfrac{1}{2} e^{N^{-1/2}} N^2 T,N X) + (\tfrac{1}{2} N +\tfrac{1}{24}) T,\\
\zscale_{u,v}(T,X) &:= e^{\hscale_{u,v}(T,X)},
\end{align}
where $h_N$ is the ASEP height function process.

\item {\em H\"older bounds on initial data}: The $N$-indexed sequence
  of open ASEP initial data $h_N(0,\cdot)$ satisfies that for all
  $\theta\in (0,\tfrac{1}{2}\big)$  and $n\in \Z_{\geq 1}$, there
  exist positive $C(n), C(\theta,n)$  such that for every $X,X'\in[0,1]$ and $N\in \Z_{\geq 1}$
\begin{equation}\label{eq_Holder}
\|\zscale_{u,v}(0,X)\|_n\leq C(n),\quad \textrm{and}\quad \|\zscale_{u,v}(0,X)-\zscale_{u,v}(0,X')\|_n\leq C(\theta,n)|X-X'|^{\theta}.
\end{equation}
Here $\|\cdot\|_n:= \big(\E\big[|\cdot|^n\big]\big)^{1/n}$ where $\E$ is expectation over $h_N(0,\cdot)$ (recall that we are not currently assuming that the law of this initial data is stationary).
\end{enumerate}
\end{assumption}

%



\begin{proposition}\label{thm_kpz_limit}
Consider any $N$-indexed sequence of open ASEPs with parameters and initial datum satisfying all Assumption \ref{assump_wasep}. Then the law of $\zscale_{u,v}(\cdot,\cdot)\in D\big([0,T_0],C([0,1])\big)$ (the Skorohod space) is tight as $T\to \infty$ for any fixed $T_0>0$ and all limit points are in $C\big([0,T_0],C([0,1])\big)$.
If there exists a (possibly random) non-negative-valued function $Z_0\in C([0,1])$ such that, as $N\to \infty$, $\zscale_{u,v}(0,X)\Longrightarrow Z_0(X)$ in the space of continuous processes of $X\in [0,1]$), then
$
\zscale_{u,v}(T,X) \Longrightarrow Z_{u,v}(T,X)
$
in $D\big([0,T_0],C([0,1])\big)$ for any $T_0>0$ as $N\to \infty$,
where $Z_{u,v}(T,X)$ in $C([0,T_0],C([0,1]))$ is the unique mild solution to the stochastic heat equation with boundary parameters $u$ and $v$, and initial data $Z_0(X)$ (recall the definition from the beginning of Section \ref{sec:intro}).
\end{proposition}

The Skorohod space is used above since ASEP takes discrete jumps in time.

\begin{proof}


This result is essentially contained in \cite{CS} for $u,v\geq 1/2$ and \cite{Parekh} for general $u,v\in \R$. The tightness is from \cite[Proposition 4.17]{CS} and \cite[Proposition 5.4]{Parekh} while the convergence result is from
\cite[Theorem 2.18]{CS} and \cite[Theorem 1.1]{Parekh}. The only difference from those works is that we have used a different parametrization. For the tightness, the boundary parameters play no role and hence our result follows immediately from that of \cite{CS,Parekh}. For the convergence result, our parametrization can relatively easily be matched to that used in \cite{CS,Parekh} and their parameters $A$ and $B$ correspond to $u-1/2 +o(1)$ and $v-1/2 + o(1)$ respectively. The $o(1)$ terms go to zero as $N$ go to infinity and thus do not affect the limiting equation (as can be justified either by a coupling argument or by tracing through the proof in \cite{CS,Parekh}).
\end{proof}

\section{Attractive coupling of different boundary parameters}\label{sec_couplingoverall}
This prepares us for the proof of Theorem \ref{thm_main} \eqref{it:tightness}-\eqref{it:duality} in Section \ref{sec_thm123}.

\subsection{Coupling via multi-species open ASEP}\label{sec_coupling}




We prove an {\em attractive} coupling of open ASEPs with different boundary conditions. This means that if the occupation variables start ordered between different ASEPs, then they will remain ordered. As is standard in proving attractive couplings (e.g. \cite{AV97}), we appeal to a multi-species version of the model. For $M=2$, part (1) below coincides with \cite[Lemma 2.1]{gantert2020mixing}.

\begin{lemma}\label{prop:threespeciescoupling}
Fix $q\geq 0$, any $M\geq 2$ and any non-negative real numbers $\{\alpha^i\big\}_{i=1}^{M}$,  $\{\beta^i\big\}_{i=1}^{M}$,  $\{\gamma^i\big\}_{i=1}^{M}$ and  $\{\delta^i\big\}_{i=1}^{M}$ such that for all $i<j$,
$$
\alpha^i\leq \alpha^j,\qquad \beta^i\geq \beta^j,\qquad \gamma^i\geq \gamma^j,\qquad \delta^i\leq \delta^j.
$$
For each $i\in \{1,\ldots, M\}$ fix any initial data $\tau^i = (\tau^i_x)_{x\in \llbracket 1,N\rrbracket}\in \{0,1\}^{N}$ such that for all $1\leq i<j\leq M$,  $\tau^i\leq \tau^j$ (i.e., $\tau^i_x\leq \tau^j_x$ for all $x\in \llbracket 1, N\rrbracket$). Let $\tau^i(\cdot)$ denote the $N$ site open ASEP with parameter $(q,\alpha^i,\beta^i,\gamma^i,\delta^i)$ started with $\tau^i(0)=\tau^i$. Then:
\begin{enumerate}[leftmargin=*]
\item There exists a single probability space supporting $M$ processes  $\tau^1(\cdot),\dots, \tau^M(\cdot)$ and has the property that for all $t\geq 0$ and  $1\leq i<j\leq M$,  $\tau^i(t)\leq \tau^j(t)$.
\item Let $\tilde{\tau}^i$ denote an occupation vector distributed according to the stationary measure for the $N$ site open ASEP with parameters $(q,\alpha^i,\beta^i,\gamma^i,\delta^i)$. Then there exists a coupling of all $M$ stationary measures such that for all  $1\leq i<j\leq M$,  $\tilde\tau^i\leq \tilde\tau^j$.
\end{enumerate}
\end{lemma}
\begin{proof}
The second claim follows immediately from the first by taking time to
infinity and using the uniqueness of the stationary measure. The first
claim can be shown by appealing to a multi-species open ASEP. Consider
an $M$ species version of open ASEP where sites can be occupied by a
single particle of species $1$ through $M$. This process has the
following transition rates (as a convention let
$\alpha^0=\beta^{M+1}=\gamma^{M+1}=\delta^0=0$)
\begin{enumerate}[leftmargin=*]
\item For $x\in\llbracket 1,N-1\rrbracket$, if sites $x$ and $x+1$ are
  occupied by $AB$ (i.e., there is a species $A$ particle  at site $x$ and a species $B$ particle at site $x+1$), then
  this becomes $BA$ with rate $1$ if $A<B$ and rate $q$ if $A>B$.
\item If site $1$ is occupied by $A$, then this becomes $B$ at rate
  $\alpha^B-\alpha^{B-1}$ if $A>B$ and at rate $\gamma^B-\gamma^{B+1}$
  of $A<B$.
\item If site $N$ is occupied by $A$, then this becomes $B$ at rate $\beta^B-\beta^{B+1}$ if $A<B$ and at rate $\delta^B-\delta^{B-1}$ if $A>B$.
\end{enumerate}
Denote the occupation variables for this process by $\eta^i_x(t)\in\{0,1\}$ where $i\in \llbracket 1, M\rrbracket$, $x\in \llbracket 1,N\rrbracket$ and $t\geq 0$.  In other words, $\eta^i_x(t)=1$ if there is a species $i$ particle at position $x$ at time $t$, and $0$ otherwise.
From these multi-species occupation variables we define
$\tau^i_x(t) = \sum_{j=1}^{i} \eta^j_x(t)$
From the $\tau^i_x$ in the statement of the lemma we define initial
data for the multi-species ASEP by $\eta^j_x(0) =
\tau^i_x-\tau^{i-1}_x$ (with the convention that $\tau^{0}_x=0$).
It is evident that $\tau^i_x(0)=\tau^i_x$ and that marginally, for each $i\in \llbracket 1,M\rrbracket$, $\tau^i(t)$ evolves as a process in $t$ precisely as open ASEP with parameters $(q,\alpha^i,\beta^i,\gamma^i,\delta^i)$. This implies the desired attractive coupling since for any $1\leq i<j\leq N$, the difference $\tau^j_x(t)-\tau^i_x(t) = \sum_{k=i+1}^{j} \eta^k_x(t)$ is positive, hence $\tau^i_x(t)\leq \tau^j_x(t)$ as desired.
\end{proof}

\subsection{Height function coupling and its implications}\label{sec_height_coupling}
Armed with the attractive coupling of Lemma  \ref{prop:threespeciescoupling} we may now prove the following results.

\begin{proposition}\label{lem:momentbound}
Fix $u,v\in \R$ and consider the stationary measure for $N$ site open ASEP parameterized as in \eqref{eq:assumerates} of Assumption \ref{assumption1} by $u$ and $v$. As in \eqref{eqn_hscale}, define the diffusive scaled stationary height function $\hscale_{u,v}(X)\in C([0,1])$, and then define its exponential transform $\zscale_{u,v}(X) := \exp(\hscale_{u,v}(X))$.
Then the following holds (in points \eqref{it:mb3} and \eqref{it:mb4} below we use the notation $\|\cdot\|_n:= \big(\langle|\cdot|^n\rangle_N\big)^{1/n}$ where $\langle \cdot\rangle_N$ is the expectation from \eqref{eq:langrang} with respect to the open ASEP stationary measure $ \pi^{\mathrm{ASEP}}_N$):
\begin{enumerate}[leftmargin=*]
\item\label{it:mb1} As processes in $C([0,1])$ as $N\to \infty$ $\hscale_{u,-u}(X)\Longrightarrow B_u(X)$ where $B_u$ is a standard Brownian motion with drift $u$ (i.e. $B_u(X) = B(X)+uX$ for $B$ a standard variance 1 Brownian motion and $X\in [0,1]$).

\item\label{it:mb12} For any $u,v\in \R$, $\hscale_{u,v}(\cdot)$ and $\hscale_{v,u}(1-X)-\hscale_{v,u}(1)$ have the same law as random functions in $C([0,1])$.

\item\label{it:mb2} For all $u$ and $v$ such that $u+v\geq 0$ there exists a coupling of $\hscale_{u,v}$  with $\hscale_{u,-u}$ and $\hscale_{-v,v}$ such that for all $X,X'\in [0,1]$ with $X\leq X'$,
$$
\hscale_{-v,v}(X')-\hscale_{-v,v}(X)\leq \hscale_{u,v}(X')-\hscale_{u,v}(X)\leq \hscale_{u,-u}(X')-\hscale_{u,-u}(X).
$$
For all $u$ and $v$ such that $u+v\leq 0$ there exists a coupling of $\hscale_{u,v}$  with $\hscale_{u,-u}$ and $\hscale_{-v,v}$ such that for all $X,X'\in [0,1]$ with $X\leq X'$,
$$
\hscale_{-v,v}(X')-\hscale_{-v,v}(X)\geq \hscale_{u,v}(X')-\hscale_{u,v}(X)\geq \hscale_{u,-u}(X')-\hscale_{u,-u}(X).
$$
\item\label{it:mb3} For all $u,v\in \R$ we have the following H\"older bound. For all $\theta\in (0,\tfrac{1}{2})$ and every $n\in \Z_{\geq 1}$, there exists a constant $C(\theta,n,u,v)>0$ such that for every $X,X'\in [0,1]$ and every $N\in \Z_{\geq 1}$,
    \begin{equation}\label{eq:Hholder}
    \big\| \hscale_{u,v}(X)-\hscale_{u,v}(X')\big\|_n \leq C(\theta,n,u,v)\left| X-X'\right|^{\theta}.
    \end{equation}
\item\label{it:mb4}  For all $u,v\in \R$ we have the following H\"older bounds. For all $n\in \Z$ there exists $C(n,u,v)>0$ such that for all $N\in\Z_{\geq 1}$ and all  $X\in [0,1]$
\begin{equation}\label{eq:LNbound}
\| \zscale_{u,v}(X)\|_n \leq C(n,u,v),
\end{equation}
and for all $\theta\in (0,\tfrac{1}{2})$ and every $n\in \Z_{\geq 1}$, there exists a constant $C(\theta,n,u,v)>0$ such that for every $X,X'\in [0,1]$ and every $N\in \Z_{\geq 1}$,
\begin{equation}\label{eq:LNbound2}
    \big\| \zscale_{u,v}(X)-\zscale_{u,v}(X')\big\|_n \leq C(\theta,n,u,v)\left|X-X'\right|^{\theta}.
\end{equation}
\end{enumerate}
\end{proposition}
\begin{proof}

Part \eqref{it:mb1} follows from that fact that when $v=-u$, $AC=1$ (recall $A$ and $C$ from \eqref{eq:assumerates})  which implies (see \cite[Remark 2.4]{BW}, \cite{ED}, or \cite{AHL_2020__3__87_0}) that the open ASEP stationary measure is product Bernoulli with particle density $\rho=(1+A)^{-1}$. As in \eqref{eq:hNdef}, for $x\in \llbracket 0,N\rrbracket$ define $h_{N;u,-u}(x)=\sum_{i=1}^{x} (2\tau_i-1)$ where $\tau$ has this stationary Bernoulli product measure. This means that $h_{N;u,-u}(\cdot)$ is a random walk with i.i.d. $\pm1$ increments, increasing by 1 with probability $\rho(u) = (1+q^u)^{-1}$ and decreasing by 1 with probability $1-\rho(u)$. Since $\rho=\rho(u) = \tfrac{1}{2}\big(1-u N^{-1/2} + O(N^{-1})\big)$ the mean of the jump distribution for $h_{N;u,-u}(x)$ is $uN^{-1/2} + O(N^{-1})$ and the variance is $1+O(N^{-1/2})$.  It then follows from Donsker's invariant theorem that under diffusive scaling $\hscale_{u,v}(X) = N^{-1/2} \hscale_{u,-u}(NX)$ converges as a process in $C([0,1])$ to  $B_u$, a standard Brownian motion with drift $u$, as claimed. Part \eqref{it:mb12} follows immediately from the particle/hole duality for open ASEP.
Part \eqref{it:mb2} follows immediately by applying Lemma \ref{prop:threespeciescoupling} with $M=3$ to sandwich the $u,v$ height function by the $u,-u$ and $-v,v$ one.

Part \eqref{it:mb3} makes use of results we have already demonstrated
in Parts \eqref{it:mb1} and \eqref{it:mb2} above. By Part
\eqref{it:mb1}, $\hscale_{-v,v}$ and $\hscale_{u,-u}$ are both
diffusively rescaled random walks and by Part \eqref{it:mb2}, the
increments of $\hscale_{u,v}$ are bounded below and above in our
coupled probability space by the increments of $\hscale_{-v,v}$ and
$\hscale_{u,-u}$ (depending on the sign of $u+v$, the ordering is
switched). Taking the $n$th power of these inequalities and then expectations yields
\begin{align}
&\big\| \hscale_{u,v}(X')-\hscale_{u,v}(X)\big\|^n_n\\
&\quad \leq 2 \max\left(\left\| \hscale_{-v,v}(X')-\hscale_{-v,v}(X)\right\|^n_n, \left\| \hscale_{u,-u}(X')-\hscale_{u,-u}(X)\right\|^n_n\right).
\end{align}
(Here we have used the following argument. Given two non-negative numbers $a$ and $b$, note that $\max(a,b) \leq a+b\leq 2\max(a,b)$. Now, consider two non-negative random variables $A$ and $B$. The first inequality gives  $\E\big[\max(A,B)\big] \leq \E[A]+\E[B]$ and using the second inequality we find that $\E[A]+\E[B]\leq 2\max\big(\E[A],\E[B]\big)$).
Thus, in order to prove the H\"older bound \eqref{eq:Hholder} for $\hscale_{u,v}$, it suffices to show it for both $\hscale_{-v,v}$ and $\hscale_{u,-u}$. These bounds, however follow readily since  $\hscale_{-v,v}$ and $\hscale_{u,-u}$ are diffusively rescaled simple random walks which converge to drifted Brownian motions. 

Turning to Part \eqref{it:mb4}, we first show \eqref{eq:LNbound}. As in the Part \eqref{it:mb3} proof we may use the coupling in Part \eqref{it:mb2} to show that
\begin{equation}\label{eq_zscalemax}
\big\|\zscale_{u,v}(X)\big\|_n^n \leq 2  \max\left(\big\|\zscale_{-v,v}(X)\big\|^n_n,\big\|\zscale_{u,-u}(X)\big\|^n_n\right).
\end{equation}
So, it is sufficient to show \eqref{eq:LNbound}  with $\zscale_{u,v}(X)$ replaced by both $\zscale_{-v,v}(X)$ and $\zscale_{u,-u}(X)$. Each of these expressions involves the exponential of a diffusively scaled i.i.d. simple random walks. As we did not include the analogous calculation in Part \eqref{it:mb3}, we include this.

Let $y_1,\ldots, y_N$ be i.i.d. random variables with $\P(y_1=1)=\rho:=(1+q^u)^{-1}$ and $\P(y_1=-1)=1-\rho$. Then we claim that for all $n\in \Z_{\geq 1}$ there exists $C(n,u)>0$ such that
\begin{equation}\label{eq:iidbound}
\E\Big[\Big(e^{N^{-1/2} (y_1+\cdots+y_M)}\Big)^n\Big] \leq C(n,u)
\end{equation}
for all $M\leq N\in\Z_{\geq 1}$.
It is clear that \eqref{eq:LNbound} immediately follows by combining this with \eqref{eq_zscalemax} and the random walk description of both $\zscale_{-v,v}(X)$ and $\zscale_{u,-u}(X)$. So, we now focus on proving \eqref{eq:iidbound}.

By independence of the $y_i$, we can rewrite the left-hand side of \eqref{eq:iidbound} as
\begin{equation}\label{eq:iidbound2}
\E\Big[\Big(e^{N^{-1/2} (y_1+\cdots+y_M)}\Big)^n\Big] = \left(\rho e^{nN^{-1/2}}+(1-\rho)e^{-nN^{-1/2}}\right)^M.
\end{equation}
Taylor expanding yields the asymptotics that $\rho e^{nN^{-1/2}}+(1-\rho)e^{-nN^{-1/2}}=1+(n^2/2+un)N^{-1} +o(N^{-1})$. Raising this to the $M^{th}$ power yields a bound on the right-hand side of \eqref{eq:iidbound2} like $e^{(n^2/2+un)X}$ where $X=M/N\in [0,1]$. The maximum occurs at either $X=0$ or $X=1$ depending on the sign of $n^2/2+un$, and thus letting $C(n,u) = \max(1, e^{n^2/2+un})$ seems to work.
To make the above argument rigorous we just need to control the Taylor expansion error. If $ \rho e^{nN^{-1/2}}+(1-\rho)e^{-nN^{-1/2}}\leq 1$ then the right-hand side in \eqref{eq:iidbound2} is maximized when $M=0$, otherwise it is maximized when $M=N$. In the first case when $M=0$, the right-hand side of \eqref{eq:iidbound2} is obviously bounded by $1$. So, it suffices to bound the right-hand side of \eqref{eq:iidbound2} when $M=N$.

We recall a few elementary inequalities which control the Taylor expansion and provide bounds which hold for all $n,N\in \Z_{\geq 1}$ and $u\in \R$. Below, $C$ will denote a positive constant. Recall that $q=e^{-2N^{-1/2}}$. We can bound
$
|\rho - (\frac{1}{2} - \frac{u}{2} N^{-1/2})| \leq C|u|^3N^{-3/2}.
$
With this we can show that
\begin{align}
&\left|\rho e^{nN^{-1/2}}+(1-\rho)e^{-nN^{-1/2}} - \left(\cosh(nN^{-1/2}) + \sinh(nN^{-1/2}) uN^{-1/2}\right) \right|\\
&\qquad\qquad\leq C\cosh(nN^{-1/2})|u|^3N^{-3/2}.
\end{align}
Bounding
$$
|\cosh(nN^{-1/2}) - (1 + \frac{n^2}{2} N^{-1})| \leq C n^4 N^{-2},
|\sinh(nN^{-1/2}) - nN^{-1/2} | \leq C n^3 N^{-3/2},$$
and $\cosh(nN^{-1/2})\leq \cosh(n)$ for $N\in \Z_{\geq 1}$, we deduce that
$$
\Big|\rho e^{nN^{-1/2}}+(1-\rho)e^{-nN^{-1/2}} - \Big(1 + \Big(\frac{n^2}{2}+un\Big) N^{-1}\Big)\Big|\leq C(u,n) N^{-3/2}.
$$
Here and below $C(u,n)>0$ depends on $u$ and $n$ (though may vary between lines).  From this it follows that
$
(\rho e^{nN^{-1/2}}+(1-\rho)e^{-nN^{-1/2}})^N \leq C(u,n)
$
as needed to show \eqref{eq:LNbound}.
%
%

We now turn to showing \eqref{eq:LNbound2}. This bound follows by combining the result of Part \eqref{it:mb3} with the already showed inequality \eqref{eq:LNbound} in Part \eqref{it:mb4}. First note that for all $a,b\in \R$,
$
| e^a-e^b| \leq \max(e^a,e^b) | a-b|.
$
Substituting $a= \hscale_{u,v}(X)$ and $b=\hscale_{u,v}(X')$, taking expectations and using Cauchy-Schwarz yields
$
\| \zscale_{u,v}(X)-\zscale_{u,v}(X')\|_n \leq \| \max(\zscale_{u,v}(X),\zscale_{u,v}(X'))\|_{2n} \cdot \|\hscale_{u,v}(X)-\hscale_{u,v}(X')\|_{2n}.
$
The result of Part \eqref{it:mb3} provides us with control over the second term on the right-hand side
thus it suffices to prove that there exists $C'(n,u,v)>0$ such that for all $X,X'\in [0,1]$
\begin{equation}\label{eqmaxsza}
 \left\| \max\left(\zscale_{u,v}(X),\zscale_{u,v}(x')\right)\right\|_{2n} \leq C'(n,u,v).
\end{equation}

We already know how to control the $n$-norms of the individual terms inside of the $\max$ in the left-hand side of \eqref{eqmaxsza} by the already proved bound \eqref{eq:LNbound}. We need to control to the $n$-norm of the $\max$. For positive random variables $A$ and $B$, using the binomial expansion and Cauchy-Schwarz we show that for any $m\in \Z_{\geq 1},$
$
\E[\max(A,B)^m] \leq 2^m \max_{k\in \llbracket 0,m\rrbracket } \E[A^{2k}]^{1/2} \E[B^{2(m-k)}]^{1/2}.
$
%
%
Substituting $m=2n$, $A=\zscale_{u,v}(X)$,  $B=\zscale_{u,v}(X')$ and using the bound proved in \eqref{eq:LNbound} it follows that $\E[A^{2k}]^{1/2}\leq C(2k,u,v)^k$ and likewise $\E[B^{2(m-k)}]^{1/2} \leq C(2(m-k),u,v)^{m-k}$. Thus,
$$
 \| \max(\zscale(x),\zscale(x'))\|_{2n} \leq 2 \max_{k\in \{0,\ldots,2n\}} C(2k,u,v)^{\frac{k}{2n}}C(2(2n-k),u,v)^{\frac{2n-k}{2n}}
$$
which we can take to be $C'(n,u,v)$. This proves \eqref{eqmaxsza} and completes the proof of \eqref{eq:LNbound2}.
\end{proof}

\section{Proof of Theorem \ref{thm_main} (\ref{it:tightness})-(\ref{it:duality})}\label{sec_thm123}

\noindent{\em Proof of Theorem \ref{thm_main}  \eqref{it:tightness}}: That the law $\mu^{(N)}_{u,v}$ of $\hscale_{u,v}(\cdot)\in C([0,1])$ is tight as $N\to \infty$ follows from Proposition \ref{lem:momentbound} \eqref{it:mb3} and the fact that $\hscale_{u,v}(0)=0$ by applying the Kolmogorov continuity theorem. That theorem further implies that all subsequential limits $\mu_{u,v}$ of  $\mu^{(N)}_{u,v}$ are supported on the space of H\"older $\alpha$ functions for all $\alpha<1/2$.

Now, consider any subsequence $\{N_k\}_{k=1}^{\infty}$ along which $\mu^{(N_k)}_{u,v}$ converges to a limit $\mu_{u,v}$. Let $H^{(N)}_{u,v}(\cdot)\in C([0,1])$ be distributed according to the law $\mu_{u,v}$. We claim that all of the conditions of Assumption \ref{assump_wasep} are satisfied. Weak asymmetry, Liggitt's condition and the triple point scaling assumptions all follow from the choice of parameters we have made in Assumption \ref{assumption1} and the H\"older bounds on the initial data are shown in Proposition \ref{lem:momentbound} \eqref{it:mb4}.
Finally, since we have assumed that the subsequence $\{N_k\}$ is such that $H^{(N_k)}_{u,v}(\cdot)\Rightarrow H_{u,v}(\cdot)$ as random functions in $C([0,1])$, it follows likewise that
$Z^{(N_k)}_{u,v}(X)\Rightarrow Z_{u,v}(X)=e^{H_{u,v}(X)}$. Thus Proposition \ref{thm_kpz_limit} implies that $Z^{(N_k)}_{u,v}(T,X)\Rightarrow Z_{u,v}(T,X)$ in $D\big([0,T_0],C([0,1])\big)$ for any $T_0>0$ where $Z_{u,v}(T,X)\in C([0,T_0],C([0,1]))$
the unique mild solution to the SHE with boundary parameters $u$ and $v$ and initial data $Z_0(X)= e^{H_{u,v}(X)}$.
Since the initial data for the open ASEP height process was chosen to be stationary (in terms of the height function increment process) it follows immediately that the law of $X\mapsto Z^{(N_k)}_{u,v}(T,X)/Z^{(N_k)}_{u,v}(T,0)\in C([0,1])$ is independent of $T$. By the convergence to the SHE, the same is true for  $X\mapsto Z_{u,v}(T,X)/Z_{u,v}(T,0)\in C([0,1])$ and taking the logarithm of this implies that the law of $X\mapsto H_{u,v}(T,X)-H_{u,v}(T,0)\in C([0,1])$ is likewise independent of $T$. This implies that the law $\mu_{u,v}$ of any limit point of $\mu^{(N)}_{u,v}$ will be a stationary measure for the open KPZ equation height function increment process.

\smallskip
\noindent{\em Proof of Theorem \ref{thm_main}  \eqref{it:coupling}}:  Proposition \ref{lem:momentbound} \eqref{it:mb2} implies that the coupling holds along any subsequence $\{N_k\}_{k=1}^{\infty}$ such that all of the $\{\mu^{(N_k)}_{u_i,v_i}\}_{i=1}^M$ converge to the limit points $\{\mu_{u_i,v_i}\}_{i=1}^{M}$; hence it passes to the limit.

\smallskip
\noindent{\em Proof of Theorem \ref{thm_main}  \eqref{it:brownian}}: This follows from Proposition \ref{lem:momentbound} \eqref{it:mb1} in conjunction with the result in Part \eqref{it:tightness} of this theorem.

\smallskip
\noindent{\em Proof of Theorem \ref{thm_main}  \eqref{it:duality}}: Proposition \ref{lem:momentbound} \eqref{it:mb12} implies that the desired equality holds along any subsequence $\{N_k\}_{k=1}^{\infty}$ such that $\mu^{(N_k)}_{u,v}$ and $\mu^{(N_k)}_{v,u}$ converges to their limit points $\mu_{u,v}$ and $\mu_{v,u}$; hence it passes to the limit.

\section{Askey-Wilson and continuous dual Hahn processes}\label{sec:AW}

\subsection{Definition of the Askey-Wilson process}\label{sec:AWprocessdefs}
Fix $q\in (-1, 1)$ and $a, b, c, d\in \C$ be such that
\begin{equation}\label{eq:assumptabcd}
ac,ad,bc,bd,qac,qad,qbc,qbd,abcd,qabcd \in \C\setminus [1, \infty),
  \end{equation}
and such that either all $a, b, c, d$ are real, or two of them are real and two form a complex conjugate pair (e.g. $a,b\in \R$ and $c=\bar{d}$), or they form two complex conjugate pairs (e.g. $a=\bar{c}$ and $b=\bar{d}$).
We will not need the Askey-Wilson polynomials, but rather just their orthogonality measure.
%

\subsubsection{Askey-Wilson probability measure}\label{sec:awd}
Fix parameters $(a, b, c, d, q).$ Under assumption \eqref{eq:assumptabcd} the corresponding Askey-Wilson polynomials
are orthogonal with respect to a unique compactly supported
probability measure. This {\it Askey-Wilson probability measure} on $\R$ is defined by the values it takes on Borel sets $V\subset \R$, which we write  as $AW(V;a, b, c, d, q)$. As shown in  \cite[Theorem A.1]{BW2} or \cite[Theorem 2.5]{AskeyWilson}, the Askey-Wilson probability measure can be decomposed into an absolutely continuous part and a discrete finitely supported atomic part.
The density, relative to the Lebesgue measure, of the absolutely continuous part will be denoted by $AW^c$ and the discrete part, which is a sum of weighted Dirac delta functions, will be denoted by $AW^d$. Thus,
\begin{equation}
AW(V; a, b, c, d, q)=\int_{V}AW^c(x; a,b,c,d,q)dx+AW^d(V;a,b,c,d,q).
\end{equation}
The density $AW^c$ is supported on $x\in \Supp^c:=[-1,1]$ and given by
    \begin{equation}\label{eq:AW}
AW^c(x;a,b,c,d,q)=\frac{(q,ab,ac,ad,bc,bd,cd)_\infty}{2\pi(abcd)_\infty \sqrt{1-x^2}}\bigg\lvert\frac{(e^{2\iu \theta_x})_\infty}{(a
  e^{\iu\theta_x}, be^{\iu\theta_x}, c e^{\iu\theta_x}, d e^{\iu\theta_x}
  )_\infty}\bigg\rvert^2,
\end{equation}
with $x=\cos\theta_x$. (We have dropped the $q$ in the $q$-Pochhammer symbols.)
The discrete part $AW^d$ is given by
\begin{equation}\label{eq:AWddef}
AW^d(V;a,b,c,d,q) = \sum\limits_{y\in V\cap \Supp^d(a,b,c,d,q)}AW^d(y;a,b,c,d,q),
\end{equation}
a sum of delta functions with masses $AW^d(y;a,b,c,d,q)$  at points $y\in $$\Supp^d(a,b,c,d,q)$
The set $\Supp^d(a,b,c,d,q)$ and masses $AW^d$ are given as follows. If
$|a|, |b|, |c|, |d|<1$ the set $\Supp^d(a,b,c,d,q)$ is empty. By
\eqref{eq:assumptabcd}, if $\chi \in\{a,b,c,d\}$ has $|\chi|>1$ then
it must be real. Each  $\chi \in\{a,b,c,d\}$ with $|\chi|>1$ generate
its own set of atoms, the union of which constitutes
$\Supp^d(a,b,c,d,q)$. By \eqref{eq:assumptabcd}, any element $\chi\in\{a,b,c,d\}$ with $|\chi|>1$ must be distinct from all other elements in that set. There are finitely many atoms generated by such $\chi$ and they are at locations
\begin{equation}\label{eq:atomchi}
y_j=\dfrac{1}{2}\left(\chi q^j+\dfrac{1}{\chi q^j}\right), \text{ for } j\in\Z_{\geq 0} \text{ such that }|\chi q^j|\geq 1.
\end{equation}
Each atom has a different mass. When $\chi=a$, these are given as (here $y_j$ are as above with $\chi=a$)
\begin{align}\label{eq:AWdmasses}
AW^d(y_0;a,b,c,d,q)&=\dfrac{\left(a^{-2}, b c, b d, c d\right)_\infty}{\left(b/a,c/a, d/a, a b c d\right)_\infty},\\
\frac{AW^d(y_j;a,b,c,d,q)}{AW^d(y^a_0;a,b,c,d,q)}&=\dfrac{\left(a^{2}, a b, a c, a d\right)_j\left(1-a^2 q^{2j}\right)}
{\left(q, q a/b, qa/c, q a/d \right)_j \left(1-a^2\right)} \left(\dfrac{q}{a b c d}\right)^j,
\end{align}
where, in the second line, we assume that $j\in\Z_{\geq 1}$ such that $|\chi q^j|\geq 1$.
For other values of $\chi$, the masses are as above except with $a$ and $\chi$ swapped.

\subsubsection{Askey-Wilson process}

Following \cite{BW2} we define the Askey-Wilson process, a
time\textendash inhomogeneous Markov process that depends on parameters $A,  B, C, D, q$.
We assume that $A,B,C,D$ correspond to the ASEP parameters $\alpha,\beta,\gamma,\delta$ via \eqref{eq_ABCD}, in which case $A,C>0$ and $B,D\in (-1,0]$. Additionally, we will assume that $AC<1$, as is necessary for the existence of the Askey-Wilson processes. This puts us in the {\em fan region} of the open ASEP phase diagram.

\begin{definition}\label{def:AWprocess}
Under the above assumptions on $A,B,C,D$, the Askey-Wilson process is
the time-inhomogeneous Markov process $\{\Y_s\}_{s\in[0, \infty)}$
whose time- inhomogeneous state-space and transitional probability distributions are given as follows.
Define the continuous and discrete atomic part of the time $s$ state-space as
\begin{equation}\label{eq:Fs}
\Supp^c_s=\Supp^c:=[-1,1], \qquad \Supp^d_s:=\Supp^d\left(A\sqrt{s}, B\sqrt{s}, \dfrac{C}{\sqrt{s}}, \dfrac{D}{\sqrt{s}},q\right)
\end{equation}
where $\Supp^d(a,b,c,d,q)$ is defined as in Section \ref{sec:awd}. Let $\Supp_s:= \Supp^c_s\cup \Supp^d_s$ represent the time $s$ state-space for the Askey-Wilson process.
For any $s<t$ the transitional probability distribution $\pi_{s,t}$ from $x\in \Supp_s$ to any Borel $V\subset\R$ is
\begin{equation}\label{eq:def_p}
\pi_{s,t}( x, V):=AW \left (V; A\sqrt{t}, B\sqrt{t},\sqrt{ \dfrac{s}{t} }\left(x+\sqrt{x^2-1}\right), \sqrt{ \dfrac{s}{t} }\left(x-\sqrt{x^2-1}\right)\right),
\end{equation}
where, for $x\in\Supp^c$, we define $x\pm\sqrt{x^2-1}=e^{\pm \iu \theta_x}$ with $\theta_x$ defined through $x=\cos \theta_x$. From the definitions of the Askey-Wilson probability measure, for $x\in \Supp_s$, the support of $\pi_{s,t}( x,\cdot)$ is $\Supp_t$. This defines the Askey-Wilson process.

We will also make use of a family of probability distributions $\pi_s$ with support $\Supp_s$ defined such that for any Borel $V\subset\R$,
\begin{equation}\label{eq:def_pi}
\pi_s(V):=AW\left(V; A\sqrt{s}, B\sqrt{s}, \dfrac{C}{\sqrt{s}}, \dfrac{D}{\sqrt{s}}\right).
\end{equation}
As explained below, the Askey-Wilson process started with distribution $\pi_s$ at time $s$ will have the property that it marginally has the distribution $\pi_t$ at any later $t>s$.
\end{definition}

Both $\pi_s$ and $\pi_{s,t}$ have absolutely continuous and discrete atomic parts. For $x\in \Supp^c$ we denote the density of $\pi_s$  by $\pi^c_s(x)$ and for $x\in \Supp^d_s$ we denote the mass $\pi_s(\cdot)$ assigns to $x$ by $\pi^d_s(x)$. Likewise for the transitional probability distribution if $x\in \Supp^c$ and $y\in \Supp^c$ then we write $\pi^{c,c}_{s,t}(x,y)$ for the density  in $y$; if $x\in \Supp^c$ and $y\in \Supp^d_t$ then we write $\pi^{c,d}_{s,t}(x,y)$ for the mass assigned to $y$; if $x\in \Supp^d_s$  and $y\in \Supp^c$ then we write $\pi^{d,c}_{s,t}(x,y)$ for the density in $y$; and if $x\in \Supp^d_s$ and $y\in \Supp^d_t$ then we write $\pi^{d,d}_{s,t}(x,y)$ for the mass assigned to $y$. For all other values of $x$ and $y$, we declare that these functions are zero.

The existence and uniqueness of the Askey-Wilson process defined above is shown in \cite[Section 3]{BW2} as is the property that it preserves the $\pi_s$ marginals. In particular, \cite[Proposition 3.4]{BW2} shows that for all Borel sets $V\subset \R$ (and in the second equation, for all $x\in \Supp_s$)
\begin{align}
\int_{\R}\pi_s(dx)\pi_{s,t}(x,V) = \pi_t(V),\quad \textrm{and} \quad
\int_{\R}\pi_{s,t}(x,dy)\pi_{t,u}(y,V) = \pi_{s,u}(x,V).
\end{align}
The first identity implies that the Askey-Wilson process started under $\pi_s$ at time $s$ has marginal distribution $\pi_t$ at time $t$, while the second is the Chapman-Kolmogorov identity necessary to define the Markov process.



\subsection{The continuous dual Hahn process}\label{sec:tangentAWP}

The open KPZ stationary measures that we construct in Theorem \ref{thm_main} \eqref{it:explicit} is characterized by a duality with the {\em continuous dual Hahn process} that we define here (the special case of this when $u,v>0$ was already defined in Section \ref{sec:introCDH}). This time-inhomoeneous Markov process is a certain limit of the Askey-Wilson processes \cite{BW1} (see Section \ref{sec:AWprocessdefs}).
We will use the standard notation for the Pochhamer symbol: For $j\in \Z_{\geq 0}$ and $x\in \R$, define $[x]_j:=(x)(x+1)\cdots(x+j-1)$ with the convention that $[x]_0:=1$. For multiple arguments $x_1,\ldots x_n\in \R$, define $[x_1, \ldots, x_n]_j:=[x_1]_j\cdots [x_n]_j$. Likewise, define $\Gamma(x_1,\ldots, x_n) = \Gamma(x_1)\cdots \Gamma(x_n)$.

We will define the time-inhomogeneous state-spaces  $\Suppp_s$ and transitional probability distributions $\p_{s,t}$ for the continuous dual Hahn process, and show their consistency. We will also define a family $\{\p_s\}_{s}$ of infinite measures supported on the state-spaces for this process that are preserved by the transitional probability distributions. All of these distributions are all related to generalized beta integrals \cite{Askey,doi:10.1137/0511064}.  We will only define the continuous dual Hahn process for times in $s\in [0,\Cuv)$ where we recall from \eqref{eq:Cuv} that $[0,\Cuv)$ always has a non-empty interior. \cite{BrycCDH} has subsequently extended this definition to all times in $\R$, though we do not need to rely on that.

We start by defining the family  $\{\p_s\}_{s}$ of infinite measures along with their supports $\Suppp_s$ that will also serve as the support for the continuous dual Hahn process. We emphasize that the $\p_s$ are not probability distributions but rather positive distributions of infinite mass. For the remainder of the definitions below, we will assume that $u+v>0$. We will also adopt the following notation. For $u,v,s\in \R$ and $j\in \Z$ define
\begin{equation}\label{eq:xuvs}
\xu_j(s):=-4(u+j-s/2)^2 \qquad \textrm{and}  \qquad  \xv_j(s):=-4(v+j+s/2)^2.
\end{equation}
Similarly, if $u-s/2<0$ or  $v+s/2<0$  we define (respectively)
\begin{equation}
\SupppU_s:=\bigcup\limits_{j=0}^{\lfloor -u+s/2\rfloor} \{\xu_j(s)\}, \qquad\qquad \SupppV_s:=\bigcup\limits_{j=0}^{\lfloor -v-s/2\rfloor}\{\xv_j(s)\}.
\end{equation}
Here  $d$ indicates that this will be the support for a discrete atomic measure.

\begin{definition}\label{def_marginal}
Assume that $u, v\in \R$ with $u+v>0$ and $s\in [0,\Cuv)$.
For any Borel $V\subset \R$, define the infinite measure
\begin{equation}
\p_s(V):=\int\limits_{V\cap \Suppp^c_s} \p^c_s (r)dr+\sum_{r\in \Suppp^d_s\cap V} \p^d_s (r)
\end{equation}
where the support of the absolutely continuous part is $\Suppp^c_s:=(0,\infty)$ and its density is defined as
\begin{equation}
\p^c_s(r):=\dfrac{(v+u)(v+u+1)}{8\pi}\cdot
      \frac{\Big|\Gamma\Big(\frac{s}{2}+v+\iu\frac{\sqrt{r}}{2},-\frac{s}{2}+u+\iu\frac{\sqrt{r}}{2}\Big)\Big|^2}{\sqrt{r}
        \cdot \big|\Gamma(\iu \sqrt{r})\big|^2}{\bf 1}_{r>0},
\end{equation}
and the support $\Suppp^d_s$ and masses $\p^d_s (r)$ of the discrete part are as follows:
If $u-s/2<0$  then  $\Suppp^d_s=\SupppU_s$ and for $j\in \llbracket 0,\lfloor -u+s/2\rfloor\rrbracket$, the masses at the points of the support are given by
\begin{equation}
\p^{d}_{s}\left(\xu_j(s)\right)=\dfrac{\Gamma\left(v-u+s ,v+u+2\right)}{\Gamma\left(-2 u+s \right)}\cdot \dfrac{(u+j-s/2)\cdot\left [ 2 u-s,v+u\right ]_j } {(u-s/2)\cdot \left[1, 1-v+u-s\right ]_j}.
\end{equation}
If $v+s/2<0$ then $\Suppp^d_s=\SupppV_s$ and for $j\in \llbracket 0,\lfloor -v-s/2\rfloor\rrbracket$, the masses at the points of the support are given by
\begin{equation}
\p^{d}_{s}\left(\xv_j(s)\right)=\dfrac{\Gamma\left(u-v-s ,2+v+u\right)}{\Gamma\left(-2 v-s \right)}\cdot \dfrac{(v+j+s/2)\cdot\left [ 2 v+s,v+u\right ]_j } {(v+s/2)\cdot \left[1, 1-u+v+s\right ]_j}.
\end{equation}
If neither of these conditions hold, then $\Suppp^d_s=\varnothing$ and there is no discrete part. Since $u+v>0$, it is not possible that both conditions hold. Define the total support of $\p_s$ to be
\begin{equation}\label{eq:cdhss}
\Suppp_s = \Suppp^c_s \cup \Suppp^d_s.
\end{equation}
This will also serve as the support for the continuous dual Hahn process. Since $\Suppp^c_s \cap \Suppp^d_s=\varnothing$, in the proof of Lemma \ref{lem:consistancy} we will find it convenient to overload notation and write $\p_s(x)$ for the density function $\p_s(x)$ when $x\in \Suppp^c_s$ and for the mass function  $\p_s(x)$ when $x\in \Suppp^d_s$.
\end{definition}

To define the transition probability distributions we first define the orthogonality measure for the continuous dual Hahn orthogonal polynomials.

\begin{definition}\label{def:CDHtransition}
Assume that $a, b, c\in \R$ with $a+b, a+c>0 $ or  $a\in \R$ and $b=\bar c\in
 \C\setminus \R$ with $\textup{Re}(b)=\textup{Re}(c)>0$. For such $a,b,c$, define
  \begin{equation}
\CDHtransition^c (x; a, b, c):=\dfrac{1}{8\pi}\,\cdot\,\dfrac{\bigg\rvert \Gamma\left(a+\iu  \dfrac{\sqrt x}{2},b+\iu
  \dfrac{\sqrt x}{2},c+\iu  \dfrac{\sqrt x}{2}\right)\bigg \rvert^2}{\Gamma(a+b,a+c,b+c)\cdot \sqrt x\cdot\left\rvert \Gamma\left(\iu \sqrt x\right)\right\rvert^2}{\bf 1}_{x> 0}.
  \end{equation}
\end{definition}

\begin{definition}\label{def:CDHd}
Assume that $a<0$ and that $\SuppCDH$ is a finite subset of $\R$
with size $|\SuppCDH| = \lfloor -a\rfloor+1$ and elements $x_0<
x_1<\cdots<x_{\lfloor -a\rfloor}$. Assume that $b,c$ and $j$ satisfy one of the three conditions:
\begin{itemize}[leftmargin=*]
  \item $b=\bar c\in \C\setminus \R$ with
    $\textup{Re}(b),\textup{Re}(c)>0$, and $j\in \llbracket 0,\lfloor -a\rfloor\rrbracket$;
  \item $b, c\in \R$ with
   $a+b, a+c>0 $, and $j\in \llbracket 0,\lfloor -a\rfloor\rrbracket$;
\item $b, c\in \R$ and $b, b+c, c-a>0,$ with $a+b=-k\in \Z_{\leq 0}$, and $j\in\llbracket 0,k\rrbracket$.
\end{itemize}
For $a,b,c$,  $\SuppCDH$, and $j$ as above, define

  \begin{equation}
\CDHtransition^d (x_j; a, b, c;\SuppCDH):=\dfrac{[2 a,a+b,a+c]_j\cdot(a+ j)\cdot
  \Gamma(b-a,c-a)}{[1,a-b+1,a-c+1]_j \cdot a \cdot\Gamma(-2a,b+c)}\,\cdot\,(-1)^j.
\end{equation}
For all other first arguments, define $\CDHtransition^d (\cdot; a, b, c;\SuppCDH)=0$.
  \end{definition}

\begin{definition}\label{def:CDH}
We will define the measure $\CDHtransition(V;a,b,c;\SuppCDH)$ (for Borel subsets $V$ of $\R$) under three possible sets of conditions on parameters:
\begin{itemize}
\item[{\bf P:}] For $a\geq 0$; $\SuppCDH=\varnothing$; and either $b=\bar c\in \C\setminus \R$ with
  $\textup{Re}(b),\textup{Re}(c)>0$ or $b, c\in \R_{>0}$, define
\begin{equation}
\CDHtransition(V; a,b,c;\SuppCDH):=\int_V \CDHtransition^c (x; a, b, c)dx.
\end{equation}
\item[{\bf N1:}] For $a<0$; $\SuppCDH$ a finite subset of $\R$
with size $|\SuppCDH| = \lfloor -a\rfloor+1$ and elements $x_0<x_1<\cdots<x_{\lfloor -a\rfloor}$; and either $b=\bar c\in \C\setminus \R$ with
   $\textup{Re}(b),\textup{Re}(c)>0$  or $b, c\in \R$ with $a+b, a+c>0$, define
\begin{equation}
\CDHtransition(V;a,b,c;\SuppCDH):=\int_V \CDHtransition^c (x; a,
   b, c)dx+\sum\limits_{x\in \SuppCDH\cap V}\CDHtransition^d(x; a,b,c;\SuppCDH)
\end{equation}
\item[{\bf N2:}] For $a<0$;  $\SuppCDH$ a finite subset of $\R$
with size $|\SuppCDH| = \lfloor -a\rfloor+1$ and elements $x_0<x_1<\cdots<x_{\lfloor -a\rfloor}$; and $b, c\in \R$ with $b, b+c, c-a>0$ and $a+b=-k\in \Z_{\leq 0}$, define
\begin{equation}
\CDHtransition(V;a,b,c;\SuppCDH):=\sum\limits_{x\in \SuppCDH\cap V}\CDHtransition^d(x; a,b,c;\SuppCDH).
\end{equation}
\end{itemize}
%
%
%
%
\end{definition}

\begin{lemma}\label{lem:CDH}
In all three cases  of Definition \ref{def:CDH}, $\CDHtransition(\cdot; a,b,c;\SuppCDH)$ is a probability measure on $\R$.
\end{lemma}
\begin{proof}
This follows from (3.1), (3.3) and (3.4) in
\cite{doi:10.1137/0511064} as a limit when Wilson polynomials
degenerate to dual continuous Hahn polynomials. In particular, this limit corresponds to taking one of the parameters of the Wilson polynomials to infinity. The case when $a=0$ does not seem to be covered therein, but can be recovered by taking the limit as $a\to 0$.
\end{proof}


We now define what will be the transition probability distribution of the continuous dual Hahn process.

\begin{definition}\label{def_transition}
Assume that $u,v\in \R$ with $u+v>0$ and $s,t\in [0,\Cuv)$ with $s<t$.
For any Borel set $V\subset \R$ we define the transition probability $\p_{s, t}$ as follows. For $m\in \Suppp^c_s=(0,\infty)$
\begin{equation}\label{eq:pstmV}
\p_{s, t}(m, V):=\CDHtransition\left(V; u-\dfrac{t}{2}, \dfrac{t-s}{2}+\iu\dfrac{\sqrt m}{2}, \dfrac{t-s}{2}-\iu\dfrac{\sqrt m}{2}; \SupppU_t\right).
\end{equation}
If $u-s/2<0$, so that $\Suppp^d_s=\SupppU_s$, then for $j\in \llbracket 0,\lfloor -u+s/2\rfloor\rrbracket$
\begin{equation}\label{eq:pstvV}
\p_{s, t}(\xu_j(s),V):=\CDHtransition\left(V; u-\dfrac{t}{2}, -u+\dfrac{t}{2}-j, u+\dfrac{t}{2}-s+j; \SupppU_t\right).
\end{equation}
If $v+s/2<0$, so that $\Suppp^d_s=\SupppV_s$, then for $j\in \llbracket 0,\lfloor -v-s/2\rfloor\rrbracket$
\begin{equation}\label{eq:pstuV}
\p_{s, t}(\xv_j(s),V):=\CDHtransition\left(V; v+j+\dfrac{t}{2},\dfrac{t}{2}-s-v-j, u- \dfrac{t}{2}; \SupppV_t\right).
\end{equation}
For all other first arguments besides those described above, we define $\p_{s, t}=0$.

As in Definition \ref{def_marginal} we introduce the following notation: If $x\in \Suppp^{c}_s$ and $y\in \Suppp^{c}_t$ then we write $\p_{s,t}^{c,c}(x,y)$ for the density in $y$ of the absolutely continuous part of the measure $\p_{s,t}(x,\cdot)$; If $x\in \Suppp^{c}_s$ and $y\in \Suppp^{d}_t$ then we write $\p_{s,t}^{c,d}(x,y)$ for the mass assigned to $y$ of the discrete atomic part of the measure $\p_{s,t}(x,\cdot)$; If $x\in \Suppp^{d}_s$ and $y\in \Suppp^{c}_t$ then we write $\p_{s,t}^{d,c}(x,y)$ for the density in $y$ of the absolutely continuous part of the measure $\p_{s,t}(x,\cdot)$; If $x\in \Suppp^{d}_s$ and $y\in \Suppp^{d}_t$ then we write $\p_{s,t}^{d,d}(x,y)$ for the mass assigned to $y$ of the discrete atomic part of the measure $\p_{s,t}(x,\cdot)$.

In the proof of Lemma \ref{lem:consistancy} we will find it convenient to overload notation and write $\p_{s,t}(x,y)$ to denote the corresponding density or mass function dictated by whether $x$ and $y$ are in their discrete or continuous supports (for example, when $x\in \Suppp^c_s$ and $y\in \Suppp^d_t$, $\p_{s,t}(x,y)=\p_{s,t}^{c,d}(x,y)$).
%
\end{definition}

The following lemma verifies that $\p_{s,t}$ is, indeed, a probability distribution and provides conditions under which the density or mass function is non-zero.

\begin{lemma}\label{lem:CDHwelldef}
Assume that $u,v\in \R$ with $u+v>0$ and $s,t\in [0,\Cuv)$ with $s<t$. For any $x\in \Suppp_s$, $\p_{s,t}(x,\cdot)$ in Definition \ref{def_transition} defines a probability distribution whose support is contained in $\Suppp_t$. For all $x\in \Suppp_s$ and $y\in \Suppp_t$, $\p_{s,t}(x,y)\geq 0$ (we are using the overloaded notation from the end of Definition \ref{def_transition}) and the only choices of $x$ and $y$ for which $\p_{s,t}(x,y)=0$ are:
(1) $x\in \Suppp^c_s$ and $y\in \Suppp^d_t=\SupppV_t$;
(2) $x\in \Suppp^d_s= \SupppV_s$ and  $y\in \Suppp^d_t= \SupppU_t$;
(3) $x\in \Suppp^d_s= \SupppU_s$ and $y\in \Suppp^c_t$;
(4) $x\in \Suppp^d_s= \SupppU_s$, $y\in \Suppp^d_t= \SupppV_t$;
(5) $x = \xu_j(s) \in \Suppp^d_s= \SupppU_s$, $y=\xu_k(s)\in \Suppp^d_t= \SupppU_t$ with $k>j$.
\end{lemma}

\begin{proof}
There are three cases to consider: \eqref{eq:pstmV},  \eqref{eq:pstuV} and  \eqref{eq:pstvV}.

\medskip\noindent {\bf In \eqref{eq:pstmV}:} If $u-t/2\geq 0$, then case {\bf P} of Definition \ref{def:CDH} applies with
$$
a=u-t/2,\quad \SuppCDH = \SupppU_t,\quad b= \dfrac{t-s}{2}+\iu\dfrac{\sqrt m}{2}, \quad c= \dfrac{t-s}{2}-\iu\dfrac{\sqrt m}{2}
$$
since $a\geq 0$, $\SuppCDH=\varnothing$ and $b=\bar{c}$ with $\textup{Re}(b)= \dfrac{t-s}{2}>0$; if $u-t/2<0$, then case {\bf N1} of Definition \ref{def:CDH} applies with the same choices of parameters since $a<0$, $|\SuppCDH| = \lfloor -a\rfloor +1$ and $b=\bar{c}$ with $\textup{Re}(b)= \dfrac{t-s}{2}>0$. We see from above that for $x\in \Suppp^c_s$, the measure $\p_{s,t}(x,\cdot)$ is supported and everywhere non-zero (in terms of its density or mass function) on $\Suppp^c_t\cup \SupppU_t$. In particular, when $u-t/2\geq 0$ and $v+t/2<0$, the mass function $\p_{s,t}(x,y)$ is zero on the discrete set $y\in\Suppp^d_t = \SupppV_t$.

\medskip\noindent {\bf In \eqref{eq:pstvV}:} Since $u-s/2<0$, so does $u-t/2< 0$. In that case {\bf N2} of Definition \ref{def:CDH} applies with
$$
a=u-t/2,\quad \SuppCDH = \SupppU_t,\quad b= -u+\dfrac{t}{2}-j,\quad c=u+\dfrac{t}{2}-s+j
$$
since $a<0$, $|\SuppCDH| = \lfloor -a\rfloor +1$,
$b\geq  \dfrac{t-s}{2}>0$, $b+c=t-s>0$, $c-a = t-s+j>0$ and $a+b = -j$ for $j\in \Z_{\geq 0}$.
We see from this that for $x = \xu_j(s)\in \SupppU_s$, the measure $\p_{s,t}(x,y)$ is non-zero only when $y= \xu_k(t)\in  \SupppU_t$ with $k\in \llbracket 0,j\rrbracket$.

\medskip\noindent {\bf In \eqref{eq:pstuV}:} If $v+j+t/2\geq 0$, then case {\bf P} of Definition \ref{def:CDH} applies with
$$
a=v+j+t/2,\quad \SuppCDH = \SupppV_t,\quad b= \dfrac{t}{2}-s-v-j,\quad c=u-\dfrac{t}{2}
$$
since $a\geq 0$, $\SuppCDH=\varnothing$, $b\geq  \dfrac{t-s}{2}>0$ (since $j\in \llbracket 0,\lfloor -v-s/2\rfloor\rrbracket$) and $c>0$ (since when $v+s/2<0$ it follows that $u>0$ and hence $t\in [0,\Cuv)$ implies that $t<2u$); if $v+j+t/2<0$, then case {\bf N1} of Definition \ref{def:CDH} applies with the same choices of parameters since $a<0$, $|\SuppCDH| = \lfloor -a\rfloor +1$, $a+b=t-s>0$ and $a+c= v+u+j>0$.  We see from above that for $x\in \SupppV_s$, the measure $\p_{s,t}(x,\cdot)$ is supported and everywhere non-zero (in terms of its density or mass function) on $\Suppp^c_t\cup \SupppV_t$. In particular if $u-t/2<0$, the mass function $\p_{s,t}(x,y)$ is zero on the discrete set $y\in\Suppp^d_t = \SupppU_t$.
\end{proof}

%
%

The next lemma is key to defining the continuous dual Hahn process and to showing that it preserves the class of marginal measures $\p_s$.

\begin{lemma} \label{lem:consistancy}
Let $u, v\in \R$ with $u+v>0$ and $0\leq s<t<w<\Cuv$. For any Borel $V\subset \R$,
\begin{equation}
\int\limits_{\R}\p_s(dm)\p_{s, t}(m, V) =\p_t(V), \qquad\textrm{and}\qquad
\int\limits_{\R}\p_{s,t}(m, dr)\p_{t, w}(r, V) =\p_{s, w}(m, V).
    \end{equation}
  \end{lemma}
We prove this after defining the continuous dual Hahn process.

\begin{definition}[Continuous dual Hahn process]\label{def:CDHprocess}
Let $u, v\in \R$ with $u+v>0$.
The {\em continuous dual Hahn process} is the Markov process  $\{\T_s\}_{s\in [0,\Cuv)}$ with time-inhomogeneous state space $\Suppp_s$ from \eqref{eq:cdhss} and transition probability given by $\p_{s,t}$. This process is well-defined since Lemma \ref{lem:consistancy} shows that  Chapman-Kolmogorov is satisfied. Lemma \ref{lem:consistancy} also proves that if the continuous dual Hahn process is started at time 0 according to the infinite measure $\p_0$ then at time $s\in [0,\Cuv)$ the marginal infinite measure for the process will be given by $\p_s$.
\end{definition}
The preservation of the family $\p_s$ should be thought of as similar to the fact that Brownian motion preserves Lebesgue measure.


The rest of this section is devoted to the proof of Lemma \ref{lem:consistancy}. We start by recalling
the orthogonality probability measure for the Wilson orthogonal polynomials from \cite{doi:10.1137/0511064}.

\begin{definition}\label{def:f}
Let $a,b,c,d\in \C$ either form two conjugate pairs such that
$a=\bar b$, $c=\bar d$, one conjugate pair (either $a=\bar b$ or
$c=\bar d$) and one pair of real numbers, or four real numbers; in all
cases assume that $\textup{Re}(b),\textup{Re}(c),\textup{Re}(d)>0$. For such $a,b,c,d$, define
%
%
 \begin{align}
   \CDWtransition^c (x; a, b, c, d)&:=\frac{\Gamma(a+b+c+d)\Big\rvert \Gamma\left(a+\iu  \frac{\sqrt x}{2},b+\iu \frac{\sqrt x}{2},c+\iu  \frac{\sqrt x}{2},d+\iu  \frac{\sqrt x}{2}\right)\Big\rvert^2}{8\pi \cdot\Gamma(a+b,a+c,b+c,a+d,b+d,c+d)\sqrt x\, \cdot\, \left\rvert \Gamma\left(\iu \sqrt x\right)\right\rvert^2}.
 \end{align}
\end{definition}

\begin{definition}\label{def:CDWd}
Assume that $a<0$ and that $\SuppCDH$ is a finite subset of $\R$ with size $|\SuppCDH|=\lfloor -a\rfloor+1$ with elements $x_0<
x_1<\cdots<x_{\lfloor -a\rfloor}$. Assume that $b,c,d$ and $j$ satisfy one of the three conditions:
\begin{itemize}[leftmargin=*]
\item $b\in \R_{>0}$, $c=\bar d\in \C\setminus \R$ with
    $\textup{Re}(c),\textup{Re}(d)>0$, and $j\in \llbracket 0, \lfloor - a\rfloor\rrbracket$;
  \item $b,c,d\in \R_{>0}$, with $a+d, a+c>0$, $j\in \llbracket 0, \lfloor - a\rfloor\rrbracket$;
\item $b, c, d\in \R$ with $b, b+c, c-a>0,$ $a+b=-k$ for $k\in \Z_{\geq 0}$, and $j\in \llbracket 0, \lfloor - a\rfloor\rrbracket$.
\end{itemize}
For $a,b,c,d$,  $\SuppCDH$, and $j$ as above, define
\begin{align}
    \CDWtransition^d(x_j; a, b, c, d;\SuppCDH):=\dfrac{\left[2a,a+b,a+c,a+d\right]_j\cdot (a+j)}{\left[1,a-b+1,a-c+1,a-d+1\right]_j
    \cdot a}
    \\\times \frac{\Gamma(a+b+c+d,b-a,c-a,d-a)}{\Gamma(-2 a,b+c,c+d,b+d)}.
\end{align}
For all other first arguments, define $\CDWtransition^d(x_j; a, b, c, d;\SuppCDH)=0$.
  \end{definition}


\begin{definition}\label{def:CDW}
We will define the measure $\CDWtransition(V; a,b,c,d;\SuppCDH)$ (for Borel subsets $V$ of $\R$) under four possible sets of conditions on parameters:
\begin{itemize}

\item[{\bf P1:}] For $a\geq 0$; $\SuppCDH=\varnothing$; $b\in \R_{>0}$; and either $c=\bar d\in \C\setminus \R$ with
  $\textup{Re}(c),\textup{Re}(d)>0$ or $c, d\in \R_{>0}$, define
\begin{equation}
\CDWtransition(V; a,b,c,d;\SuppCDH):=\int_V \CDWtransition^c (x; a, b, c,d)dx.
\end{equation}
\item[{\bf P2:}] For $a=\bar b\in \C\setminus \R$ with $\textup{Re}(a),\textup{Re}(b)>0$; $\SuppCDH=\varnothing$; and either $c=\bar d\in \C\setminus \R$ with
  $\textup{Re}(c),\textup{Re}(d)>0$ or $c, d\in \R_{>0}$, define
\begin{equation}
\CDWtransition(V; a,b,c,d;\SuppCDH):=\int_V \CDWtransition^c (x; a, b, c,d)dx.
\end{equation}
\item[{\bf N1:}] For $a<0$; $\SuppCDH$ a finite subset of $\R$
with size $|\SuppCDH| = \lfloor -a\rfloor+1$ and elements $x_0<x_1<\cdots<x_{\lfloor -a\rfloor}$; $b\in \R_{>0}$; and either $c=\bar d\in \C\setminus \R$ with $\textup{Re}(c),\textup{Re}(d)>0$  or $c, d\in \R$ with $a+c, a+d>0$, define
\begin{equation}
\CDWtransition(V;a,b,c,d;\SuppCDH):=\int_V \CDWtransition^c (x; a,
   b, c,d)dx+\sum\limits_{x\in \SuppCDH\cap V}\CDWtransition^d(x; a,b,c,d;\SuppCDH)
\end{equation}
\item[{\bf N2:}] For $a<0$;  $\SuppCDH$ a finite subset of $\R$
with size $|\SuppCDH| = \lfloor -a\rfloor+1$ and elements $x_0<x_1<\cdots<x_{\lfloor -a\rfloor}$; $b\in\R_{>0}$ with  $a+b=-k$ for $k\in \Z_{\geq 0}$; and $c, d\in \R$ with $b+c,b+d,c-a,d-a>0$, define
\begin{equation}
\CDWtransition(V;a,b,c,d;\SuppCDH):=\sum\limits_{x\in \SuppCDH\cap V}\CDWtransition^d(x; a,b,c,d;\SuppCDH).
\end{equation}
\end{itemize}
%
%
\end{definition}
\begin{lemma}\label{lem:CDW}
In all cases  of Definition \ref{def:CDW}, $\CDWtransition(\cdot; a,b,c,d;\SuppCDH)$ is a probability measure on $\R$.
\end{lemma}
\begin{proof}
This follows directly from the identities (3.1), (3.3) and (3.4) in \cite{doi:10.1137/0511064}. The case when $a=0$ does not seem to be covered therein, but can be recovered by taking the limit as $a\to 0$.
\end{proof}

We can now give the proof of Lemma \ref{lem:consistancy}.

\begin{proof}[Proof of Lemma \ref{lem:consistancy}]
The idea  is to rewrite the relations in Lemma \ref{lem:consistancy} in terms of the $\CDHtransition(\cdot; a,b,c;\SuppCDH)$ and $\CDWtransition(\cdot; a,b,c,d;\SuppCDH)$ measures, and then use the fact that they integrate to 1 to demonstrate the desired identities.

\medskip \noindent{\bf Proving the first identity in Lemma \ref{lem:consistancy}.} It suffices to prove that
\begin{equation}\label{eq:id11}
\int_{\R} \p_s(dm) \p_{s,t}(m,r) = \p_t(r)
\end{equation}
for all $r\in \R$. Here we have overloaded the $\p_t$ and $\p_{s,t}$ notation (as density and mass functions) as explained at the end of Definitions \ref{def_marginal} and  \ref{def_transition}. For $r\notin \Suppp_t$, $\p_t(r)=0$. Likewise, for such $r$ it is easy to see from the five cases in Lemma \ref{lem:CDHwelldef} that the left-hand side in \eqref{eq:id11} is also zero. Thus, we assume now that $r\in \Suppp_t$.
We can now divide and rewrite \eqref{eq:id11} as
\begin{equation}\label{eq:id12}
\int_{\R} \frac{\p_s(dm) \p_{s,t}(m,r)}{\p_t(r)} = 1.
\end{equation}
To show this identity, we identify the integrand above with the continuous dual Hahn probability measure (hence its integral is 1). There are three cases which we address below.

\medskip\noindent{\bf Case 1.}
For $r\in \Suppp^c_t$,
\begin{equation}\label{eq:id12r}
 \frac{\p_s(dm) \p_{s,t}(m,r)}{\p_t(r)} = \CDHtransition\left(dm;
  v+\dfrac{s}{2}, \dfrac{t-s}{2}+\iu \dfrac{\sqrt r}{2},
  \dfrac{t-s}{2}-\iu \dfrac{\sqrt r}{2}; \SupppV_s\right).
\end{equation}
To prove this, observe that for $m\in \Suppp^c_s$ we may rewrite
\begin{equation}\label{eq:id12r1}
\dfrac{\p_s(m) \p_{s,t}(m, r)}{\p_t(r)} = \CDHtransition^c\left(m ;
  v+\dfrac{s}{2}; \dfrac{t-s}{2}+\iu \dfrac{\sqrt r}{2},
  \dfrac{t-s}{2}-\iu \dfrac{\sqrt r}{2}\right).
\end{equation}
The computation here is obtained by regrouping the terms and
  using the idenity $\Gamma(x+1)=x\Gamma(x).$
Similarly, for $m=\xv_j(s)\in \Suppp^d_s = \SupppV_s$,
\begin{equation}\label{eq:id12r2}
\dfrac{\p_s(m)\p_{s, t}(m,r)}{\p_t(r)}
  = \CDHtransition^d\left(m;v+\dfrac{s}{2}, \dfrac{t-s}{2}+\iu \dfrac{\sqrt r}{2},
  \dfrac{t-s}{2}-\iu \dfrac{\sqrt r}{2}; \SupppV_s\right).
\end{equation}
From Lemma \ref{lem:CDHwelldef} we have that for $m=\xu_j(s)\in \Suppp^{d}_s=\SupppU_s$, $\p_{s, t}(m,r)=0$.
Depending on the value of $v$, we see that the parameters in \eqref{eq:id12r1} and \eqref{eq:id12r2} either satisfy {\bf P} or {\bf N1} in Definition \ref{def:CDH} and either way we arrive at \eqref{eq:id12r} and verify that the right-hand side is a probability measure.

\medskip\noindent{\bf Case 2.}
For $r=\xu_k(t)\in \Suppp^d_t = \SupppU_t$,
\begin{equation}\label{eq:id12v}
  \frac{\p_s(dm) \p_{s,t}(m,r)}{\p_t(r)} =  \CDHtransition\left(dm; u+k-\dfrac{s}{2},
  v+\dfrac{s}{2}, -u-k+t-\dfrac{s}{2}; \Suppp^{d,u}_s\right).
\end{equation}
To prove this, observe that for $m\in \Suppp^c_s$ we may rewrite
\begin{equation}\label{eq:id12v1}
  \frac{\p_s(dm) \p_{s,t}(m,r)}{\p_t(r)}
  = \CDHtransition^c\left(m; u+k-\dfrac{s}{2}, v+\dfrac{s}{2}, -u-k+t-\dfrac{s}{2}\right).
\end{equation}
Similarly, for $m=\xu_j(t)\in \Suppp^d_s=\SupppV_s$,
\begin{equation}\label{eq:id12v2}
  \frac{\p_s(dm) \p_{s,t}(m,r)}{\p_t(r)} =
  \CDHtransition^d\left(m; u+k-\dfrac{s}{2}, v+\dfrac{s}{2}, -u-k+t-\dfrac{s}{2}; \SupppU_s\right).
\end{equation}
From Lemma \ref{lem:CDHwelldef}  we have that for $m=\xv_j(s)\in \Suppp^{d}_s= \SupppV_s$, $\p_{s, t}(m,r)=0$.
Depending on the value of $u$, the parameters in \eqref{eq:id12v1} and \eqref{eq:id12v2} either satisfy  {\bf P} or {\bf N1} in  Definition \ref{def:CDH} and either way we arrive at \eqref{eq:id12v}  and verify that the right-hand side is a probability measure.

\medskip\noindent{\bf Case 3.}
For $r=\xv_k(t)\in \Suppp^d_t = \SupppV_t$,
\begin{equation}\label{eq:id12u}
  \frac{\p_s(dm) \p_{s,t}(m,r)}{\p_t(r)} = \CDHtransition\left(dm; v\dfrac{s}{2},
  -v-\dfrac{s}{2}-k, v+t-\dfrac{s}{2}+k; \SupppV_s\right).
\end{equation}
To prove this, observe that for $m=\xv_j(s)\in \Suppp^d_s=\SupppV_s$  we may rewrite
\begin{equation}\label{eq:id12u1}
\dfrac{\p_s(m)\p_{s, t}(m,r)}{\p_t(r)}= \CDHtransition^d\left(m;v+\dfrac{s}{2},
  -v-\dfrac{s}{2}-k, v+t-\dfrac{s}{2}+k; \SupppV_s\right).
\end{equation}
From Lemma \ref{lem:CDHwelldef} we have that for $m=\xu_j(s)\in \Suppp^{d}_s=\SupppU_s$, $\p_{s, t}(m,r)=0$ and likewise for $m\in \Suppp^c_s$,  $\p_{s, t}(m, r)=0$.
The parameters in \eqref{eq:id12u1} satisfy {\bf N2} in Definition \ref{def:CDH} and thus we arrive at \eqref{eq:id12u}  and verify that the right-hand side is a probability measure.

\medskip \noindent{\bf Proving the second identity in Lemma \ref{lem:consistancy}.} It suffices to prove that
\begin{equation}\label{eq:id2}
\int\limits_{\R}\p_{s,t}(m, dr)\p_{t, w}(r, x) =\p_{s, w}(m, x)
\end{equation}
for all $m,x\in \R$. As in the proof of the first identity, we are
overloading the $\p_{s,t}$ notation (as density and mass functions) as
explained at the end of Definition \ref{def_transition}. If $m\notin
\Suppp_s$ or $x\notin \Suppp_t$ then $\p_{s, w}(m, x)=0$. It is
likewise easy to see that in this case, the right-hand side of
\eqref{eq:id2} is also zero. The five cases in Lemma
\ref{lem:CDHwelldef} identify the choices of $m\in \Suppp_s$ and $x\in
\Suppp_t$ for which we still have $\p_{s, w}(m, x)=0$. It is easy to
check from that list that for these choices of $m$ and $x$, the right-hand side of \eqref{eq:id2} is also zero. Thus, we may now assume that $m$ and $x$ are such that $\p_{s, w}(m, x)>0$.

In that case, we may divide and rewrite \eqref{eq:id2} as
$
\int_{\R}\frac{\p_{s,t}(m, dr)\p_{t, w}(r, x)}{\p_{s, w}(m, x)}=1.
$
To show this identity, we identify the integrand above with the Wilson probability measure (hence its integral is 1). There are five cases which we address below. In light of the five cases in Lemma \ref{lem:CDHwelldef}, we see that these are the only cases in which $\p_{s, w}(m, x)>0$. Since the proofs hear are similar to those used to show the first identity in Lemma \ref{lem:consistancy}, we simply record the five cases. To shorten notation, write $(\star) := \frac{\p_{s,t}(m, dr)\p_{t, w}(r, x)}{\p_{s, w}(m, x)}$. Then

\medskip\noindent{\bf Case 1.} For $m\in \Suppp^c_s$ and $x\in \Suppp^c_w$,
\begin{equation}\label{eq:id21case1}
(\star) =  \CDWtransition\left(dr; \dfrac{w-t}{2}+\iu
    \dfrac{\sqrt{x}}{2}, \dfrac{w-t}{2}-\iu
    \dfrac{\sqrt{x}}{2},\dfrac{t-s}{2}+\iu\dfrac{\sqrt m}{2},
    \dfrac{t-s}{2}-\iu\dfrac{\sqrt m}{2}\right).
\end{equation}
{\bf P2} in Definition \ref{def:CDW} applies.


\medskip\noindent{\bf Case 2.} For $m\in \Suppp^c_s$ and $x=\xu_k(w)\in \Suppp^d_w=\SupppU_w$,
\begin{equation}\label{eq:id21case2}
(\star) = \CDWtransition\left(dr; u+k-\dfrac{t}{2},
  w-\dfrac{t}{2}-u-k, \dfrac{t-s}{2}+\iu\dfrac{\sqrt m}{2}; \dfrac{t-s}{2}+\iu\dfrac{\sqrt m}{2};
  \SupppU_t\right).
\end{equation}
Depending on the value of $u$, either {\bf P1} or  {\bf N1} in Definition \ref{def:CDW} applies.


\medskip\noindent{\bf Case 3.} For $m=\xv_j(s)\in \Suppp^d_s=\SupppV_s$ and $x\in \Suppp^c_w$,
\begin{equation}\label{eq:id21case3}
(\star) =
 \CDWtransition\left(dr; v+j+\dfrac{t}{2}, -v+\dfrac{t}{2}-s-j, \dfrac{w-t}{2}+\iu\dfrac{\sqrt x}{2}, \dfrac{w-t}{2}-\iu\dfrac{\sqrt x}{2}; \SupppV_t\right).
 \end{equation}
 Depending on the value of $v$, either
 {\bf P1} or  {\bf N1} in Definition \ref{def:CDW} applies.


\medskip\noindent{\bf Case 4.}  For $m=\xv_j(s)\in \Suppp^d_s=\SupppV_s$ and $x=\xv_k(s)\in \Suppp^d_w=\SupppV_w$,
\begin{equation}\label{eq:id21case4}
(\star)=
\CDWtransition\left(dr; v+k+\dfrac{t}{2},-v-\dfrac{t}{2}-j, -v+\dfrac{t}{2}-k-s,v-\dfrac{t}{2}+j+w; \SupppV_t\right).
\end{equation}
{\bf N2} in Definition \ref{def:CDW} applies.

%

\medskip\noindent{\bf Case 5.} For $m=\xu_j(s)\in \Suppp^d_s=\SupppU_s$ and $x=\xu_k(w)\in \Suppp^d_w=\SupppU_w$ with $k\in \llbracket 0,j\rrbracket$,
\begin{equation}\label{eq:id21case5}
(\star)= \CDWtransition\left(dr; u-\dfrac{t}{2}+k, -u+\dfrac{t}{2}-j, -u-\dfrac{t}{2}-k+w, u+\dfrac{t}{2}+j-s; \SupppU_t\right).
\end{equation}
{\bf N2} in Definition \ref{def:CDW} applies.
\end{proof}

\section{Asymptotics of the ASEP generating function and proof of Theorem \ref{thm_main} (\ref{it:explicit})}\label{sec:proof5}
The main technical ingredient to proving Theorem \ref{thm_main} \eqref{it:explicit} is provided by Proposition \ref{prop_ASEP_gen_function} which is stated below and proved in Section \ref{sec:propASEPproof}. 
Before stating the proposition, we rewrite the function $\phi_{u,v}(\vec{c},\vec{X})$  in \eqref{eq_mainthmformphi} as in Section \ref{sec:discussion} as
\begin{equation}\label{eq:phipdef}
\phi_{u,v}(\vec{c},\vec{X}) = \frac{\tilde\phi_{u,v}(\vec{c},\vec{X})}{\tilde\phi_{u,v}(\vec{0},\vec{X})}, \quad \tilde\phi_{u,v}(\vec{c},\vec{X}):=\E \left[ \G\Big((\T_{s_1},\ldots ,\T_{s_{d+1}}); \vec{c};\vec{X}\Big)\right],
\end{equation}
where $\T$ is the continuous dual Hahn process started with $\T_0$ distributed according to the infinite measure $\p_0$. Explicitly, this means that
$$
\tilde\phi_{u,v}(\vec{c},\vec{X}) = \int\G(\vec{r}; \vec{c};\vec{X}) \prod_{i=1}^{d}\p_{s_{i+1},s_{i}}(r_{i+1},dr_{i})\cdot  \p_{s_{d+1}}(dr_{d+1})
$$
where $\p_s$ and $\p_{s,t}$ (recall Definitions \ref{def_marginal} and \ref{def_transition}) are the marginal and transition measures for the continuous dual Hahn process with this initial distribution (recall Definition \ref{def:CDHprocess}), and the function $\G$ is defined as
\begin{equation}\label{eq_G}
    \G(\vec{r}; \vec{c};\vec{X}) := \exp\left(\frac{1}{4}\sum_{k=1}^{d+1} (s_k^2-r_k)(X_k-X_{k-1})\right).
\end{equation}
Recall that $\vec{s} = (s_1>\cdots >s_{d+1})$ is related to $\vec{c}=(c_1,\ldots,c_d)$ as in \eqref{eq:Xcs} by $s_k = c_k+\cdots +c_d$ for $k\in \llbracket 1,d\rrbracket$ and $s_{d+1}=0$. In \eqref{eq:phipdef} we are considering transition from $r_{i+1}$ to $r_{i}$ between times $s_{i+1}$ and $s_{i}$ (recall $s_{i+1}<s_{i}$). This slightly odd labeling of time (and hence of the $r$ variables) comes from a time reversal in the Askey-Wilson process which produces the continuous dual Hahn process.
For $d\in \Z_{\geq 1}$ define $\Cduv:=\frac{1}{d} \Cuv$ where $\Cuv$ is defined in \eqref{eq:Cuv}.
The following result is proved in Section \ref{sec:propASEPproof}.

\begin{proposition}\label{prop_ASEP_gen_function}
Assume that $q,A,B,C$ and $D$ satisfy Assumption \ref{assumption1} and
are parameterized by $N$ and $u,v\in \R$ with $u+v>0$; let
$H^{(N)}_{u,v}$ be the random function in $C([0,1])$ defined in \eqref{eqn_hscale} whose law $\mu^{(N)}_{u,v}$ is that of the diffusively scaled open ASEP height function stationary measure with the above specified parameters. For $d\in \Z_{\geq 1}$, let $\vec{X},\vec{c}$ and $\vec{s}$ be as in \eqref{eq:Xcs} and let $\phi^{(N)}_{u,v}(\vec{c},\vec{X})$ denote the Laplace transform of $H^{(N)}_{u,v}$ defined in \eqref{eq:pihn} with $\tilde{c}=0$.
Then, for all $\vec{c}\in (0,\Cduv)^d$, we have the point-wise convergence of the open ASEP Laplace transform
\begin{equation}\label{eq:phi_negative}
  \lim_{N\rightarrow \infty}\phi^{(N)}_{u,v}(\vec{c},\vec{X}) = \phi_{u,v}(\vec{c},\vec{X}).
\end{equation}
%
%
\end{proposition}

\begin{remark}\label{remark:ASEPgen}
The point-wise convergence in \eqref{eq:phi_negative} is only stated for $\vec{c}\in (0,\Cduv)^d$. Here we explain why. The Askey-Wilson process marginal and transition probability distribution involves a mixture of an absolutely continuous and discrete atomic part. The nature of the atomic part depends on the number of parameters whose norm exceeds $1$. By limiting the range of the $c_k$ we limit which atoms can arise. If we permitted the $c_k$ to be larger, we would need to keep track of additional groups of atoms as well as transition probabilities between them which would increase the complexity of notation and require additional care. For our purposes it is sufficient that we have convergence on some open interval.

There is an alternative approach to minimize the contribution of atoms. (This possible approach came out in discussions with Yizao Wang, after communicating a draft of our paper to him.) We could utilize the more general formula in Corollary \ref{cor:ASEPgen} for $\phi^{(N)}(\vec{c},\tilde{c},\vec{X})$ and choose $\vec{c}$ and $\tilde{c}$ so that $-2v<s_k+\tilde{c}<2u$ (the interval $(-2v,2u)$ is non-empty since $u+v>0$). This avoids atoms coming from the $a$ and $c$ terms in \eqref{eq:abcdqvals} but may introduce atoms coming from the $b$ and $d$ terms in that equation. These atoms are located near $-1$ and it seems that their contribution will disappear in our scaling limit. We do not pursue it here.

Let us briefly compare our proof below to the style of proof used in \cite{BW1}. Therein, the authors used the fact that their limiting
Laplace transform formula could be identified as the Laplace transform for a bona-fide probability measure (this identification is made in
\cite{BRYC201877}). On account of this, the authors are able to apply a result which generalizes \cite{Curtiss} and shows that convergence
of the Laplace transform on any open set to a Laplace transform of some other probability measure implies weak convergence of the underlying probability measures to that limiting measure. If we wanted to apply this exact approach in our current situation we would need to know a priori that $\phi_{u,v}(\vec{c},\vec{X})$ is the Laplace transform for some probability measure. When we first posted this paper, such an identification was an open problem, hence we came up with another approach. Since first posting this paper, this fact has been established in \cite{BKWW} when $\min(u,v)>-1$ (see also \cite{BLD}). Since the restriction on $\min(u,v)>-1$ does not cover the full range of $u+v>0$, we still provide our approach which does not rely on the identification of the limit as the Laplace transform of a probability measure. Our approach uses some probabilistic information about the WASEP process, namely tightness and uniform control over exponential moments (both of which follow from a nice coupling of the stationary measure with random walks) to show that $\phi_{u,v}(\vec{c},\vec{X})$ coincides on an open set with the Laplace transform of some sub(sub)sequential weak limits. This identifies uniquely the weak limits along all subsubsequences as being the same, and hence shows convergence of the original sequence of measures.
\end{remark}

Using Proposition \ref{prop_ASEP_gen_function} we may now give the proof of Theorem \ref{thm_main} (\ref{it:explicit}).

\begin{proof}[Proof  of Theorem \ref{thm_main} \eqref{it:explicit}]

There are two things to show here. The first is that when $u+v>0$, the tight sequence of measures $\mu^{(N)}_{u,v}$ (i.e. the laws of $H^{(N)}_{u,v}(\cdot)\in C([0,1])$) from Theorem \ref{thm_main}  \eqref{it:tightness} has a unique limit point. To show this it suffices to show that the finite dimensional distributions of $H^{(N)}_{u,v}(\cdot)$ converge weakly. The second is to show is that the Laplace transform of the limiting finite dimensional distributions are given by $\phi_{u,v}(\vec{c},\vec{X})$ as claimed in \eqref{eq_mainthmformphi}.

Fix any $d\in \Z_{\geq 1}$ and $0<X_1<\cdots<X_d\leq 1$, and let $X^{(N)} = N^{-1} \lfloor N X\rfloor$. We will consider the sequence of random vectors $H^{(N)}_{u,v}(\vec{X}) := \big(H^{(N)}_{u,v}(X^{(N)}_1),\ldots ,H^{(N)}_{u,v}(X^{(N)}_d)\big)$ and use $\P^{(N)}$ to denote the law of the corresponding random vector. For $\vec{c}\in \C^d$ we let $L^{(N)}(\vec{c}):= \int_{\R^d} e^{\vec{c}\cdot \vec{x}} \P^{(N)}(d\vec{x})$ denote the Laplace transform of $\P^{(N)}$. Note that $L^{(N)}(\vec{c}) = \phi^{(N)}_{u,v}(\vec{c},\vec{X})=\langle e^{-\sum_{k=1}^d c_k H^{(N)}_{u,v}(X^{(N)}_k)} \rangle_N$.

Theorem \ref{thm_main}  \eqref{it:tightness} shows that $\{\P^{(N)}\}_{N}$ is a tight sequence as $N\to \infty$. In particular, for any
subsequence $N_k$ this implies that there exists a further subsubsequence $N_{k_j}$ along which the $\P^{(N_{k_j})}$ converge weakly to a limit which we will denote by $\P^{(\infty)}$. A priori, $\P^{(\infty)}$ may depend on the choice of subsequence and subsubsequence. We will show that it does not. This will imply that the original sequence $\P^{(N)}$ converges weakly to $\P^{(\infty)}$ as well.

Proposition \ref{prop_ASEP_gen_function} shows that there exists an open interval $I\subset \R$ (e.g. $I=(0,\Cuv)$ works) so that $\lim_{N\rightarrow \infty}\phi^{(N)}_{u,v}(\vec{c},\vec{X}) =\lim_{N\rightarrow \infty}L^{(N)}(\vec{c})= \phi_{u,v}(\vec{c},\vec{X})$ for all $\vec{c}\in I^d$. This convergence, of course, extends to the subsubsequence $N_{k_j}$. We claim that the Laplace transform $L^{(\infty)}(\vec{c}):=\int_{\R^d} e^{\vec{c}\cdot \vec{x}} \P^{(\infty)}(d\vec{x})$
of $\P^{(\infty)}$ is finite for all $\vec{c}\in \C^d$ and that for $\vec{c}\in I^d$ it agrees with $\phi$ so that $L^{(\infty)}(\vec{c}) = \phi_{u,v}(\vec{c},\vec{X})$. Let us assume this claim for the moment. Then by analyticity of $L^{(\infty)}(\vec{c})$ we see that the knowledge of $\phi_{u,v}(\vec{c},\vec{X})$ for $\vec{c}\in I^d$ uniquely (by uniqueness of analytic continuations) characterizes the Laplace transform elsewhere, including on the imaginary axis. On account of this and the Cramer-Wold device, we can uniquely characterize the law of $\P^{(\infty)}$ from $\phi_{u,v}(\vec{c},\vec{X})$. Since the same $\phi$ arises for any choice of subsubsequence $N_{k_j}$, this implies that $\P^{(\infty)}$ does not depend on the choice of subsubsequence and hence that $\P^{(N)}$ converges weakly to $\P^{(\infty)}$ which has Laplace transform $L^{(\infty)}$ which coincides with $\phi_{u,v}(\vec{c},\vec{X})$ for $\vec{c}\in I^d$. Finally, note that we have been dealing above with convergence of $\big(H^{(N)}_{u,v}(X^{(N)}_1),\ldots ,H^{(N)}_{u,v}(X^{(N)}_d)\big)$. This also implies that $\big(H^{(N)}_{u,v}(X_1),\ldots ,H^{(N)}_{u,v}(X_d)\big)$ converges weakly to the same limit since $|H^{(N)}_{u,v}(X^{(N)})-H^{(N)}_{u,v}(X)|<N^{-1/2}$.

What remains from above is to prove the claim that $L^{(\infty)}(\vec{c})$ is finite for all $\vec{c}\in \C^d$ and that $L^{(\infty)}(\vec{c}) = \phi_{u,v}(\vec{c},\vec{X})$  for $\vec{c}\in I^d$. To prove the first part of this claim we appeal to \eqref{eq:browniancomparison} (which follows from Theorem \ref{thm_main} \eqref{it:coupling} and \eqref{it:brownian}) which implies that if $\big(H_{u,v}(X_1),\ldots,H_{u,v}(X_d)\big)$ has distribution $\P^{(\infty)}$ then there exists a coupling of that random vector w with $B_{-v}$, a standard Brownian motions of drift $-v$, and with $B_u$, a standard Brownian motion of drift $u$, such that $B_{-v}(X_k) \leq H_{u,v}(X_k) \leq B_u(X_k)$ for all $k\in \llbracket 1,d\rrbracket$. Owing to this and the Gaussian tails of Brownian motion, it follows easily that the Laplace transform $L^{(\infty)}(\vec{c})$ is finite for all $\vec{c}\in \C^d$.

Finally, we claim  that $\lim_{N\to\infty}L^{(N)}(\vec{c})=L^{(\infty)}(\vec{c})$ for all $\vec c\in \C^d$. From this it will immediately follow that $L^{(\infty)}(\vec{c}) = \phi_{u,v}(\vec{c},\vec{X})$  for $\vec{c}\in I^d$. To show the claim it suffices to show the following: For all $\e>0$ there exists $M>0$ such that for all $N\in \Z_{\geq 1}\cup \{\infty\}$,
\begin{equation}\label{eq:laplaceerrorbound}
\int_{\R^d} \mathbf{1}\{||\vec{x}||>M\} \, e^{\vec{c}\cdot \vec{x}}\, \P^{(N)}(d\vec{x}) <\e,
\end{equation}
where $||\vec{x}|| = \max(|x_1|,\ldots, |x_d|)$.
By convergence of $\P^{(N)}$ to $\P^{(\infty)}$,
$$
\lim_{N\to \infty} \int_{\R^d} \mathbf{1}\{||\vec{x}||\leq M\}\, e^{\vec{c}\cdot \vec{x}}\, \P^{(N)}(d\vec{x}) = \int_{\R^d} \mathbf{1}\{||\vec{x}||\leq M\}\, e^{\vec{c}\cdot \vec{x}} \,\P^{(\infty)}(d\vec{x})
$$
for all $M>0$.
Combining this with the error bound claimed in \eqref{eq:laplaceerrorbound} proves that $\lim_{N\to\infty} L^{(N)}(\vec{c})=L^{(\infty)}(\vec{c})$.

To  prove \eqref{eq:laplaceerrorbound} we use Cauchy-Schwarz to show that
$$
\int_{\R^d} \mathbf{1}\{||\vec{x}||>M\}\, e^{\vec{c}\cdot \vec{x}}\, \P^{(N)}(d\vec{x}) \leq \sqrt{\P^{(N)}(||\vec{x}||>M) \int_{\R^d} e^{2\vec{c}\cdot \vec{x}} \,\P^{(N)}(d\vec{x})}.
$$
By tightness we know that for any $\e>0$ there is some $M>0$ so that for all $N\in \Z_{\geq 1}\cup \{\infty\}$, $\P^{(N)}(||\vec{x}||>M)<\e$. So, it suffices to show that the other term on the right-hand side stays uniformly bounded in $N$. Notice that by repeated use of H\"older's inequality we can bound
$$
\int_{\R^d} e^{2\vec{c}\cdot \vec{x}}\, \P^{(N)}(d\vec{x}) \leq \prod_{k=1}^{d} \left(\int_{\R} e^{2d c_k x_k}\, \P^{(N)}(dx_k)\right)^{1/d},
$$
where we are writing $\P^{(N)}(dx_k)$ for the marginal of $\P^{(N)}$ in the $x_k$ coordinate.
The integrals on the right-hand side above can be rewritten in terms of the notation of Proposition \ref{lem:momentbound} as
$\E[ Z^{(N)}_{u,v}(X_k)^{2dc_k}]$. For $c_k>0$, we can bound $\E[ Z^{(N)}_{u,v}(X_k)^{2dc_k}] \leq \E[ Z^{(N)}_{u,v}(X_k)^{n}]+1$ where $n$ is any integer which is larger than $2dc_k$; for $c_k<0$, we can similarly bound $\E[ Z^{(N)}_{u,v}(X_k)^{2dc_k}] \leq \E[ Z^{(N)}_{u,v}(X_k)^{n}]+1$ where $n$ is any integer which is smaller than $2dc_k$. In either case, we can uniformly bound $\E[ Z^{(N)}_{u,v}(X_k)^{n}]$ via the bound \eqref{eq:LNbound} in Proposition \ref{lem:momentbound} \eqref{it:mb4}. This proves \eqref{eq:laplaceerrorbound}.
\end{proof}

We close this section by recording one of the results proved above that generalizes \eqref{eq:phi_negative} to all $\vec{c}\in \C^d$.
\begin{lemma}
For all  $\vec c\in \C^d$, $\lim_{N\to\infty}\phi^{(N)}_{u,v}(\vec{c},\vec{X})=\phi_{u,v}(\vec{c},\vec{X})$.
\end{lemma}

\section{Proof of Proposition \ref{prop_ASEP_gen_function}}\label{sec:propASEPproof}
We start, in Section \ref{sec:hupos}, with a heuristic explanation for the convergence in Proposition \ref{prop_ASEP_gen_function}. In Section \ref{sec:rewritingupos} we introduce scalings of our Askey-Wilson process formulas in a manner fitting for asymptotics. Section \ref{sec:asymptoticupos} contains precise bounds and asymptotic results involving these scaled Askey-Wilson process formulas (these are proved later in Section \ref{sec:applications_lemmas}). Section \ref{sec:upos} puts these bounds and asymptotics together to prove the convergence in \eqref{eq:phi_negative}---thus proving Proposition \ref{prop_ASEP_gen_function}. The key technical input to the asymptotics performed in this section are the $q$-Pochhammer asymptotics from Proposition \ref{factorials} (which are proved in Section \ref{sec:factorials}).

\subsection{Heuristic for the convergence \eqref{eq:phi_negative}}\label{sec:hupos}
Corollary \ref{cor:ASEPgen} provides a formula, \eqref{eq:pihn}, for $\phi^{(N)}_{u,v}(\vec{c},\vec{X})$ in terms of a ratio of expectations over the Askey-Wilson process. In the numerator of this ratio, there is a product over $d+1$ terms which take the form (assume $X_k\in \Z/N$ for the moment)
\begin{equation}\label{eq:coshexample}
\left(\cosh\left(s_k/\sqrt{N}\right)+ \Y_{e^{-2 s_k/\sqrt{N}}}\right)^{N(X_k-X_{k-1})}.
\end{equation}
As $N\to\infty$, we are taking the $s_k$ to be fixed and positive, and
likewise for the difference $X_k-X_{k-1}$. As $N\to\infty$, we have that $\cosh(s_k/\sqrt{N})\approx 1+
  \tfrac{s_k^2}{2N}$. The question is how does $\Y_{e^{-2 s_k/\sqrt{N}}}$ behave. Recall that from Section \ref{sec:AWprocessdefs} there are two parts to the support of the Askey-Wilson process $\Y_s$, an absolutely continuous part of support $\Supp^c_s=[-1,1]$ and a discrete atomic part $\Supp^d_s$ support above $1$. In our scaling, the atomic part lives in a $N^{-1}$ window above $1$. Thus, writing $\Y_{e^{-2 s/\sqrt{N}}} = 1 - \Yh_s^{(N)}/(2N)$ and assuming that $\Yh_s^{(N)}$ is of order one, \eqref{eq:coshexample} behaves (for $N$ large enough) like $2^{N(X_k-X_{k-1})}$ times
$e^{\frac{1}{4}(s^2-\Yh_s^{(N)})(X_k-X_{k-1})}$. This is the origin of the $\G$ function in \eqref{eq:phipdef}.

There are a few issues complicating the above heuristic. Recall that in \eqref{eq:pihn} we are considering the Askey-Wilson process $\Y_s$ with marginal distribution $\pi_s$. Under our scalings, while the discrete part of  $\Y_{e^{-2 s_k/\sqrt{N}}}$ does converge to a limit in a $N^{-1}$ window above $1$, the absolutely continuous part does not stay in that window. In fact, it remains of full support in $\Supp^c=[-1,1]$ even though the window is of order $N^{-1}$ around $1$. However, the $\G$ function has strong decay as the $\Y_{e^{-2 s_k/\sqrt{N}}}$ variable drops below a $N^{-1}$ window of $1$. Thus, we need to justify that the contribution  to the expectation coming from $\Y_{e^{-2 s_k/\sqrt{N}}}$ below this window is negligible in the large $N$ limit. Furthermore, we need to determine what happens to $\Y_{e^{-2 s_k/\sqrt{N}}}$ when we only consider it in this window. This leads to the continuous dual Hahn process that we have introduced in Section \ref{sec:tangentAWP}. (The continuous dual Hahn process can be thought of as a {\em tangent} process to the Askey-Wilson process.)
We will see that the marginal distribution in this $N^{-1}$ window has
a limit when compensated by a suitable power of $N$. The limit is no longer a probability measure, but rather of infinite mass. However, the transition probabilities of the Askey-Wilson process converge to bona-fide transition probabilities.

\subsection{Rewriting formulas to take asymptotics}\label{sec:rewritingupos}
Recall that $q, A,B,C$ and $D$ satisfy Assumption \ref{assumption1} and are parameterized by $N$ (through $q=e^{-2/\sqrt{N}}$) and $u,v\in \R$ with $u+v>0$. In what follows we will assume that our Askey-Wilson processes $\Y$ depend $N$, $u$ and $v$ through these parameters $q,A,B,C,D$. As $N$ changes, the law of the process changes. Though this dependence will be implicit at times, it should not be forgotten.

We will assume here and below that the Askey-Wilson process $\Y_s$ is always taken with marginal distribution $\pi_s$ for all $s$.
Define the centered and scaled Askey-Wilson process
\begin{equation}\label{eq:rescaling}
\Yh_s^{(N)}:=2 N \left(1-\Y_{q^{s}}\right).
\end{equation}
Due to the factor $q^s = e^{-2s/\sqrt N}$, the process $\Yh_s^{(N)}$ involves a time reversal of $\Y$. Thus, its transition probabilities
involve a conjugation by the marginal distribution. In writing down the marginal distribution and transition probabilities of $\Yh^{(N)}$
we distinguish the absolutely continuous and discrete atomic part of the support and measure. This is important since there is a Jacobian
factor which is present when the measure is absolutely continuous, though not when it is discrete.

\begin{remark} This time reversal, which was also used in \cite{BW}, is convenient conceptually since it allows us to write out limiting formulas in terms of a process that moves forward in time. It is not strictly necessary, though. We have opted to include it since it more closely matches \cite{BW}.
  \end{remark}

For any Borel $V\subset \R$, denote the marginal probability that $\Yh_s^{(N)}\in V$ by $\hat\pi^{(N)}_s(V)$. This probability measure can be written as the sum of an absolutely continuous part and a discrete atomic part. We denote the density of the absolutely continuous part by $\hat\pi^{(N), c}_s(y)$ and the probability mass of the discrete atomic part by $\hat\pi^{(N), d}_s(y)$. The support of $\hat\pi^{(N), c}_s$ is
$\Supph^{(N),c}_s=\Supph^{(N),c}:= [0,4N]$
and does not depend on $s$. The support of the discrete atomic part is
$
\Supph^{(N),d}_s := \left\{y\in \R: 1-\tfrac{y}{2N}\in \Supp^d_{q^{s}}\right\}
$
where $\Supp^d_{q^{s}}$ is defined via \eqref{eq:Fs} with $A,B,C,D$ and $q$ scaled dependent on $N$ and $u$ and $v$ as in the statement of Proposition \ref{prop_ASEP_gen_function}. We will use $\pi^{(N),c}_s$ and $\pi^{(N),d}_s$ to denote $\pi^{c}_s$ and $\pi^{d}_s$ from \eqref{eq:def_pi} where  $A,B,C,D$ and $q$ are scaled dependent on $N$ and $u$ and $v$ as noted above. Similarly, we introduce a superscript $(N)$ for the transition probabilities defined in \eqref{eq:def_p}.

For any $s$ the marginal distribution of $\Yh_s^{(N)}$ is specified by
($y\in \Supph^{(N),c}$ in the first formula and $y\in
\Supph^{(N),d}_t$ in the second)

\begin{equation}\label{eq:hatp}
\hat\pi^{(N), c}_t(y)
=\tfrac{1}{2N}  \pi^{(N), c}_{q^{t}}\left(1-\tfrac{y}{2N}\right), \qquad \hat\pi^{(N), d}_t(y) = \pi^{(N), d}_{q^{t}}\left(1-\tfrac{y}{2N}\right).
\end{equation}


Using the same convention as described below \eqref{eq:def_p}, for
$x\in \Supph^{(N),c}$ we write $\hat \pi^{(N), c, c}_{s,t}(x, y)$ for
the transition probability density supported on $y\in \Supph^{(N),c}$
while $\hat \pi^{(N), c, d}_{s,t}(x, y)$ is the mass function for
$y\in \Supph^{(N),d}_t$. Similarly, for  $x\in \Supph^{(N),d}_s$ we
write $\hat \pi^{(N), d, c}_{s,t}(x, y)$ for the transition
probability density supported on $y\in \Supph^{(N),c}$ while $\hat
\pi^{(N), d, d}_{s,t}(x, y)$ is the mass function for $y\in
\Supph^{(N),d}_t$. For all other values of $x$ or $y$, we declare
these functions to be zero. With this notation we have

\begin{align}
\hat \pi^{(N), c,c}_{s,t}(x, y)
&=\tfrac{1}{2N}  \pi^{(N), c, c}_{q^{t}, q^{s} }\left (1-\tfrac{y}{2N}, 1-\tfrac{x}{2N}\right)\cdot \dfrac{\pi^{(N), c}_{q^{t}}\left(1-\frac{y}{2N}\right)}{\pi^{(N), c}_{q^{s}}\left(1-\frac{x}{2N}\right)}
\label{eq:hatppcc}\\
\hat \pi^{(N), d, c}_{s,t}(x, y)
&=\tfrac{1}{2 N} \pi^{(N), c, d}_{q^{t}, q^{s} }\left (1-\tfrac{y}{2N}, 1-\tfrac{x}{2N}\right)\cdot \dfrac{\pi^{(N), c}_{q^{t}}\left(1-\frac{y}{2N}\right)}{\pi^{(N),d}_{q^{s}}\left(1-\frac{x}{2N}\right)}
\label{eq:hatppdc}\\
\hat \pi^{(N), c,d}_{s,t}(x, y)
&=  \pi^{(N), d, c}_{q^{t}, q^{s} }\left(1-\tfrac{y}{2N}, 1-\tfrac{x}{2N}\right)\cdot \dfrac{\pi^{(N), d}_{q^{t}}\left(1-\frac{y}{2N}\right)}{\pi^{(N), c}_{q^s}\left(1-\frac{x}{2N}\right)}
\label{eq:hatppcd}\\
\hat \pi^{(N), d,d}_{s,t}(x, y)&=  \pi^{(N), d, d}_{q^{t},q^{s}}\left(1-\tfrac{y}{2N} , 1-\tfrac{x}{2N}\right)\cdot \dfrac{\pi^{(N), d}_{q^{t}}\left(1-\frac{y}{2N}\right)}{\pi^{(N), d}_{q^{s}}\left(1-\frac{x}{2N}\right)}.
\label{eq:hatppdd}
\end{align}

We need one last piece of notation. For $\vec{X},\vec{c}$ and
$\vec{s}$ as in \eqref{eq:Xcs} and $\vec{r} = (r_1,\ldots ,r_{d+1})$
define

\begin{equation}\label{eq:gN}
\G^{(N)}(\vec{r}; \vec{c};\vec{X}):= {\bf 1}_{\vec{r}\in \R^{d+1}_{\leq 4N}} \cdot 2^{-N}\prod_{k=1}^{d+1}\left(\cosh\left (\tfrac{s_k}{\sqrt N}\right)+1-\tfrac{r_k}{2N}\right)^{N(X^{(N)}_k-X^{(N)}_{k-1})}.
\end{equation}
Recall that $X^{(N)} = N^{-1} \lfloor N X\rfloor$.
Now, we can rewrite \eqref{eq:pihn} as in \eqref{eq:phipdef}:
\begin{equation}\label{eq:phineq}
\phi^{(N)}_{u,v}(\vec{c},\vec{X}) = \frac{\tilde\phi^{(N)}_{u,v}(\vec{c},\vec{X})}{\tilde\phi^{(N)}_{u,v}(\vec{0},\vec{X})},\quad
\tilde\phi^{(N)}_{u,v}(\vec{c},\vec{X}):=\E \left[ N^{u+v}\G^{(N)}\Big((\Yh^{(N)}_{s_1},\ldots ,\Yh^{(N)}_{s_{d+1}}); \vec{c};\vec{X}\Big)\right],
\end{equation}
or more explicitly,
\begin{equation}\label{eq:phiNpreasymptot}
\tilde\phi^{(N)}_{u,v}(\vec{c},\vec{X})=
\int\G^{(N)}(\vec{r}; \vec{c};\vec{X}) \prod_{i=1}^{d}\hat\pi^{(N)}_{s_{i+1},s_{i}}(r_{i+1},dr_{i})\cdot N^{u+v}\hat\pi^{(N)}_{s_{d+1}}(dr_{d+1}).
\end{equation}
The inclusion of the factor $N^{u+v}$ will be seen below as necessary to have $\tilde\phi^{(N)}_{u,v}(\vec{c},\vec{X})$ converge to a limit as $N\to \infty$ (in particular, it is required in Lemma \ref{lem:p_zero} for the convergence of the marginal distribution to a non-trivial limit). Since $\phi^{(N)}_{u,v}(\vec{c},\vec{X})$ is a ratio of such terms, its inclusion in both the numerator and denominator does not change the ratio.

\subsection{Lemmas for asymptotics and bounds}\label{sec:asymptoticupos}
We provide the key technical results necessary to prove the point-wise convergence in \eqref{eq:phi_negative}. The proofs of these lemmas are postponed until Section \ref{sec:applications_lemmas}. We assume the scalings in Proposition \ref{prop_ASEP_gen_function} and the notation introduced in Section \ref{sec:rewritingupos}.

Equation \eqref{eq:phiNpreasymptot} is our starting point for asymptotics. In order to take $N\to\infty$ therein, we must control the convergence of $\G^{(N)}$ (in terms of a point-wise limit and dominating function) and the convergence of the $\Yh$ process. We start with the $\G^{(N)}$ function.

\begin{lemma} \label{lem:G_bound}
For every compact interval $I\subset \R$ there exists a constant $C>0$ such that for all $\vec{c}\in I^{d}$ and
 all $\vec{r}\in \R^{d+1}$,
\begin{equation}\label{eq:Gbound}
\G^{(N)}(\vec{r}; \vec{c};\vec{X})\leq C \prod\limits_{k=1}^{d+1} e^{-(X_k-X_{k-1})\frac{r_k}{4} + \frac{|r_k|}{2N}}
\end{equation}
where $\vec{X}$ and $\vec{c}$ are as in \eqref{eq:Xcs}.
For all  $\vec{r}\,^{(N)}\in \R^{d+1}$ if $\lim_{N\to \infty}\vec{r}\,^{(N)}=\vec{r}$,
\begin{equation}\label{eq:GNconv}
 \lim_{N\rightarrow \infty}  \G^{(N)}(\vec{r}\,^{(N)}; \vec{c};\vec{X})=\G(\vec{r}; \vec{c};\vec{X}).
\end{equation}
\end{lemma}

The next four lemmas provide limits with quantified error bounds for the marginal and transition measures of $\Yh$. These limits are written in terms of the continuous dual Hahn process transition measures and family of marginal distributions from Definitions \ref{def_transition} and \ref{def_marginal}. The first lemma deals with the continuous part $\hat \pi^{(N), c}_{t}$ of the distribution of $\Yh_t$, and the second with the discrete part.

\begin{lemma} \label{lem:p_zero}
For all $t\in \R$ and all $\eta>1$, there exists $N_0\in \Z_{\geq 1}$ and  $C,\chi\in \R_{>0}$ such that for all $N>N_0$ and $r\in [0,4N]$,
\begin{equation}\label{Lem83eq}
N^{u+v}\sqrt{1-\frac{r}{4N}}\, \hat \pi^{(N), c}_{t}(r)
=\p^c_{t}(r)\,\cdot\, e^{\ErrLem^{(N), c}_{t}(r)},
\end{equation}
where the error term satisfy
\begin{equation}\label{eq:lemma82ineq}
\Big|\ErrLem^{(N), c}_{t}(r)\Big| \leq CN^{-\chi}(1+\sqrt{r})^{\eta}.
\end{equation}
For each fixed $r\in \R_{\geq 0}$ we have the point-wise convergence
\begin{equation}\label{eq:ptlem73}
\lim_{N\to \infty} N^{u+v} \, \hat \pi^{(N), c}_{t}(r)=\p^c_{t}(r).
\end{equation}
  \end{lemma}

The next lemma deals with the atomic part $\hat \pi^{(N), d}_{t}$ of $\Yh_t$. We show that the locations and masses of the finitely many atoms have limits as $N\to \infty$ and that those match with $\p^d_t$ from Definition \ref{def_marginal}.

Recall from Section \ref{sec:rewritingupos} that we write $\pi^{(N),d}_t$ to denote the atomic measure $\pi^{d}_t$ from \eqref{eq:def_pi}.
In light of \eqref{eq:def_pi} and \eqref{eq:AWddef}, for any Borel set $V\subset \R$
\begin{equation}\label{eq:pindtv}
\pi^{(N),d}_t(V)= \sum\limits_{y\in V\cap \Supp^d(a,b,c,d,q)}AW^d(y;a,b,c,d,q)
\end{equation}
where $q=e^{-2/\sqrt{N}}$ and
\begin{equation}\label{eq:abcdqvals}
 a= q^{v+t/2},\quad b= -q^{1+t/2},\quad c= q^{u-t/2},\quad d= -q^{1-t/2}.
\end{equation}
If $t\in (-2,2)$ then  $|b|,|d|<1$. Since $ac = q^{u+v}<1$, it follows that at most one of $|a|$ or $|c|$ can exceed $1$.

When $u-t/2<0$, $|c|>1$ and the set of atoms in $\pi^{(N),d}_t$ are given by
\begin{equation}\label{eq:ynvkt}
\xuNunscaled_j(t):=\dfrac{1}{2}\left(q^{u+j-t/2} +q^{-(u+j-t/2)}\right),\quad j\in \llbracket0,\ldots \lfloor -u+t/2\rfloor\rrbracket.
\end{equation}

Similarly, when
$v+t/2 <0$, $|a|>1$ and thus based on the discussion prior to \eqref{eq:atomchi}, we conclude that the set of atoms in $\pi^{(N),d}_t$ are given by
\begin{equation}\label{eq:ynujt}
\xvNunscaled_j(t):=\dfrac{1}{2}\left(q^{v+j+t/2} +q^{-(v+j+t/2)}\right),\quad \textrm{for } j\in\llbracket 0,\lfloor -v-t/2\rfloor\rrbracket.
\end{equation}

These account for all of the atoms in $\Supp^d(a,b,c,d,q)$.
Finally, let us denote
\begin{equation}\label{eq:yhatscale}
\xuN_j(t) := -2N\left(\xuNunscaled_j(t)-1\right),\quad \xvN_j(t) := -2N\left(\xvNunscaled_j(t)-1\right).
\end{equation}
The support, $\Supph^{(N),d}_t$, of $\hat\pi^{(N), d}_t$ is the union of these atoms. At most one type of atoms, either from $u$ or $v$, will appear for a given $t$. As $q$ varies with $N$, the number and type of atoms remains fixed.

\begin{lemma}\label{marginal_atoms}
Assume $t\in (-2,2)$. The location and masses of the (finitely many) atoms of $\hat\pi^{(N), d}_t$ converge to those of $\p^d_t$ (recall \eqref{eq:xuvs}). Explicitly, when  $v+t/2 <0$ the atoms of $\hat\pi^{(N), d}_t$ are at $\xvN_j(t)$ for $j\in \llbracket0,\ldots \lfloor -v-t/2\rfloor\rrbracket$ and
\begin{equation}\label{eq:lem83a}
\lim_{N\to \infty} \xvN_j(t) = \xv_j(t):= -4(v+j+t/2)^2.
\end{equation}
Similarly, when $u-t/2<0$ the atoms of $\hat\pi^{(N), d}_t$ are at $\xuN_j(t)$ for $j\in \llbracket0,\ldots \lfloor u-t/2\rfloor\rrbracket$ and
\begin{equation}\label{eq:lem83aprime}
\lim_{N\to \infty} \xuN_j(t) = \xu_j(t):= -4(u+j+t/2)^2.
\end{equation}

There exists $C,\chi\in \R_{>0}$ such that for all $N\in \Z_{\geq 1}$ and all atoms
\begin{equation}
\begin{split}\label{marginal_atoms_convergence}
N^{u+v} \hat \pi^{(N), d}_t(\xvN_j(t)) &=\p_t^{d}(\xv_j(t))\,\cdot\, e^{\ErrLem^{(N),d}_t(\xv_j(t))},\\
N^{u+v} \hat \pi^{(N), d}_t(\xuN_j(t)) &=\p_t^{d}(\xu_j(t))\,\cdot\, e^{\ErrLem^{(N),d}_t(\xu_j(t))}
\end{split}
\end{equation}
where the error terms satisfy
\begin{equation}\label{eq:lem83ab}
\Big|\ErrLem^{(N),d}_t(\xv_j(t))\Big|,\Big|\ErrLem^{(N),d}_t(\xu_j(t))\Big| \leq CN^{-\chi}.
\end{equation}
In particular,
\begin{equation}
\begin{split}\label{eq:lem83b}
\lim_{N\to \infty} N^{u+v} \hat \pi^{(N), d}_t(\xvN_j(t)) &=\p_t^{d}(\xv_j(t)),\\
\lim_{N\to \infty} N^{u+v} \hat \pi^{(N), d}_t(\xuN_j(t)) &=\p_t^{d}(\xu_j(t)).
\end{split}
\end{equation}
%
%
%
%
\end{lemma}

The next two lemmas deal with the transition probabilities for $\Yh$. The first of these lemmas deals with the continuous part of the transition probability while the second deals with the discrete part.

\begin{lemma}\label{lem:p} For all real $s<t$ and $\eta>1$, there exists $N_0\in \Z_{\geq 1}$ and  $C,\chi\in \R_{>0}$ such that for all $N>N_0$  the following bounds hold:
\begin{enumerate}[leftmargin=*]
\item For all $m, r\in \Supph^{(N),c}=[0,4N]$
\begin{equation}
\sqrt{1-\frac{r}{4N}}\,\hat \pi^{(N),c,c}_{s, t}(m,r) =\p^{c,c}_{s, t}(m, r)\,\cdot\, e^{\ErrLem^{(N),c,c}_{s, t}(m,r)},
\end{equation}
where the error term satisfies
\begin{equation}\label{eq:hatpNmrbound}
\Big|\ErrLem^{(N),c,c}_{s, t}(m,r)\Big| \leq CN^{-\chi}(1+\sqrt{m})^{\eta}+CN^{-\chi}(1+\sqrt{r})^{\eta}.
\end{equation}
For each fixed $m,r\in \R$ we have the point-wise convergence
\begin{equation}\label{eq:hatpNmrlim}
\lim_{N\to \infty} \hat \pi^{(N), c, c}_{s, t}(m, r)=\p^{c, c}_{t}(m,r).
\end{equation}
\item If $s,t\in (-2,2)$ and $v+s/2<0$, so that $\Supph^{(N),d}_s$ is entirely composed of $v$ atoms $\xvN_j(s)$ for $j\in\llbracket 0,\lfloor -v-s/2\rfloor\rrbracket$, see \eqref{eq:yhatscale}, then for each $j\in\llbracket 0,\lfloor -v-s/2\rfloor\rrbracket$, and all $r\in \Supph^{(N),c}=[0,4N]$
\begin{equation}
\hat \pi^{(N),d, c}_{s, t}\left(\xvN_j(s), r\right) =\p^{d,c}_{s,t}\left(\xv_k(s), r\right)\,\cdot\, e^{\ErrLem^{(N),d, c}_{s, t}(\xv_k(s),r)}
\end{equation}
where the error term satisfies
\begin{equation}\label{eq:ucasebound}
\Big|\ErrLem^{(N),d, c}_{s, t}(\xv_k(s),r)\Big|\leq CN^{-\chi}(1+\sqrt{r})^{\eta}.
\end{equation}
In particular, for each fixed $j\in\llbracket 0,\lfloor -v-s/2\rfloor\rrbracket$ and $r\in \R_{\geq 0}$,
\begin{equation}\label{eq:ucaselimit}
\lim_{N\to\infty}  \hat \pi^{(N),d, c}_{s, t}\left(\xvN_j(s), r\right) = \p^{d,c}_{s,t}\left(\xv_k(s), r\right).
\end{equation}
\item If $s,t\in (-2,2)$ and  $u-s/2<0$, so that $\Supph^{(N),d}_s$ is entirely composed of $u$ atoms $\xuN_j(s)$ for $j\in\llbracket 0,\lfloor -u+s/2\rfloor\rrbracket$, see \eqref{eq:yhatscale}, then for each $j\in\llbracket 0,\lfloor -u+s/2\rfloor\rrbracket$ and all $r\in \Supph^{(N),c}=[0,4N]$
\begin{equation}
\hat \pi^{(N),d, c}_{s, t}\left(\xuN_j(s), r\right)=\p^{d, c}_{s, t}\left(\xu_j(s), r\right)=0.
\end{equation}
\end{enumerate}
%
%
%
%
%
%
%
%
%
%
%
\end{lemma}

\begin{lemma}\label{transition_atoms}
The locations and masses of the (finitely many) atoms of the transition probability distributions $\hat\pi^{(N),d, d}_{s, t}$ and $\hat\pi^{(N),c, d}_{s, t}$ converge to those of $\p^{d, d}_{s, t}$ and $\pi^{c, d}_{s, t}$. For all $s,t\in [0,\Cuv)$ with $s<t$
\begin{enumerate}[leftmargin=*]
\item If $v+s/2 <0$, then for all $j\in\llbracket 0,\lfloor
  -v-s/2\rfloor\rrbracket$, the discrete support of \\$\hat\pi^{(N),d, d}_{s, t}\left(\xvN_j(s), \cdot\right)$ coincides with the set of points $\xvN_k(t)$\\ for $k\in\llbracket0,\lfloor -v-t/2\rfloor\rrbracket$ (if $v+t/2>0$ then there are no atoms) and the masses satisfy the following bound: There exists $C,\chi\in \R_{>0}$ such that for all $N\in \Z_{\geq 1}$, all $j\in\llbracket 0,\lfloor -v-s/2\rfloor\rrbracket$, and all $k\in\llbracket0,\lfloor -v-t/2\rfloor\rrbracket$
\begin{equation}\label{eq:pinnexp}
\hat\pi^{(N),d, d}_{s, t}\left(\xvN_j(s), \xvN_k(t)\right) = \p^{d, d}_{s, t}\left(\xv_j(s), \xv_k(t)\right) \,\cdot\, e^{\ErrLem^{(N),d, d}_{s, t}}
\end{equation}
where the error term satisfies
\begin{equation}\label{eq:pinnexpbound}
\left|\ErrLem^{(N),d, d}_{s, t}\right| \leq CN^{-\chi}.
\end{equation}
In particular,
\begin{equation}
\lim_{N\to \infty} \hat\pi^{(N),d, d}_{s, t}\left(\xvN_j(s), \xvN_k(t)\right)= \p^{d, d}_{s, t}\left(\xv_j(s), \xv_k(t)\right).
\end{equation}
\item If $u-s/2 <0$, then for all $j\in\llbracket 0,\lfloor -u+s/2\rfloor\rrbracket$, the discrete support of \\$\hat\pi^{(N),d, d}_{s, t}\left(\xuN_j(s), \cdot\right)$ is the set of points $\xuN_k(t)$ for $k\in\llbracket0,j\rrbracket$ and the masses satisfy the following bound: There exists $C,\chi\in \R_{>0}$ such that for all $N\in \Z_{\geq 1}$, all $j\in\llbracket 0,\lfloor -u-s/2\rfloor\rrbracket$, and all $k\in\llbracket0, j\rrbracket$
\begin{equation}\label{eq:pinnexp2}
\hat\pi^{(N),d, d}_{s, t}\left(\xuN_j(s), \xuN_k(t)\right) = \p^{d, d}_{s, t}\left(\xu_j(s), \xu_k(t)\right) \,\cdot\, e^{\ErrLem^{(N),d, d}_{s, t}}
\end{equation}
where the error term satisfies
\begin{equation}\label{eq:pinnexpbound2}
\left|\ErrLem^{(N),d, d}_{s, t}\right| \leq CN^{-\chi}.
\end{equation}
In particular,
\begin{equation}
\lim_{N\to \infty} \hat\pi^{(N),d, d}_{s, t}\left(\xuN_j(s), \xuN_k(t)\right)= \p^{d, d}_{s, t}\left(\xu_j(s), \xu_k(t)\right).
\end{equation}
\item If $v+s/2<0$ and $m\in\Supph^{(N),d}=[0,4N]$, then the measure $\hat\pi^{(N),c, d}_{s, t}\left(m,\cdot\right)$ has no discrete atomic part. If $u-s/2<0$ and $m\in\Supph^{(N),d}=[0,4N]$, then the measure  $\hat\pi^{(N),c, d}_{s, t}\left(m,\cdot\right)$ has discrete support equal to the set of points $\xuN_k(t)$ for $k\in\llbracket0,\lfloor -u+s/2\rfloor\rrbracket$ and the masses satisfy the following bound:
    There exists $C,\chi\in \R_{>0}$ such that for all $N\in \Z_{\geq 1}$, all $m\in\Supph^{(N),d}=[0,4N]$ and all $k\in\llbracket 0,\lfloor -u+s/2\rfloor\rrbracket$,
\begin{equation}
\hat\pi^{(N),c, d}_{s, t}\left(m, \xuN_k(t)\right) =  \p^{c, d}_{s, t}\left(m, \xu_k(t)\right) \,\cdot\, e^{\ErrLem^{c, d}_{s, t}\left(m, \xu_k(t)\right)}
\end{equation}
where the error term satisfies
\begin{equation}
\left|\ErrLem^{c, d}_{s, t}\left(m, \xu_k(t)\right)\right| \leq CN^{-\chi}.
\end{equation}
\end{enumerate}

\end{lemma}

\begin{remark}
Lemmas  \ref{lem:p} and  \ref{transition_atoms} prove convergence of the transition probabilities for $\Yh^{(N)}$ to those of $\T$, and Lemmas \ref{lem:p_zero} and \ref{marginal_atoms} prove the convergence of the their state spaces (in addition to the convergence of a family of marginal distributions). This implies that as Markov processes on the time interval $[0,\Cuv)$, $\Yh^{(N)}_s$ converges in finite-dimensional distributions to $\T_s$. Of course, our results prove precise error bounds on the convergence of the transition probabilities too.
\end{remark}

\subsection{Proof of Proposition \ref{prop_ASEP_gen_function}}\label{sec:upos}
%

Recall $\tilde\phi_{u,v}(\vec{c},\vec{X})$ from \eqref{eq:phipdef}. 
\begin{lemma}\label{lem_strict_positive}
For $\vec{c}\in (0,\Cduv)^d$, $\tilde\phi_{u,v}(\vec{c},\vec{X}),\tilde\phi_{u,v}(\vec{0},\vec{X})\in (0,\infty)$.
\end{lemma}

\begin{proof}
We show that  $\tilde\phi_{u,v}(\vec{c},\vec{X}\edit{)}\in (0,\infty)$ for $\vec{c}\in (0,\Cduv)^d$ since $\tilde\phi_{u,v}(\vec{0},\vec{X}\edit{)}\in (0,\infty)$ follows similarly.
For $\vec{c}\in  (0,\Cduv)^d$, $s_k:= c_k+\cdots +c_d\in[0,\Cuv)$ for all $k\in \llbracket 1, d+1\rrbracket$.
The support of the integrand defining $\tilde\phi_{u,v}(\vec{c},\vec{X})$ is $\Suppp_{s_1}\times \cdots\times \Suppp_{s_{d+1}}$ (see Definition \ref{def_marginal}). For fixed $u$ and $v$, there exists a constant $L<0$ such that for all $s\in [0,\Cuv)$ the support $\Suppp_s$ is lower bounded by $L$. This follows immediately from the definition of $\Suppp_s$ as the union of $\Suppp^c_s=(0,\infty)$ with $\Suppp^d_s$, a finite number of negative discrete atoms, whose locations vary continuously with $s$. Owing to this lower bound and the ordering of the $X_k$ variables, there exists a constant $C>0$ such that for all  $\vec{c}\in (0,\Cduv)^d$
\begin{equation}\label{eq:phirdd}
\tilde\phi_{u,v}(\vec{c},\vec{X})  \leq C  \int_{\R} e^{-\frac{(1-X_{d})}{4} r} \p_0(dr).
\end{equation}
We used $\G(\vec{r}; \vec{c};\vec{X})\leq C e^{-\frac{(1-X_{d})}{4} r_{d+1}}$ for $\vec{r}\in \Suppp_{s_1}\times \cdots\times \Suppp_{s_{d+1}}$, and then Lemma \ref{lem:consistancy} to integrate out the variables $r_1,\ldots, r_{d}$.

We claim that the integral on the right-hand side of \eqref{eq:phirdd} is finite.
%
%
%
The atoms in $\p_0$ have a finite contribution to the integral on the right-hand side of \eqref{eq:phirdd}, so it remains to control the integral on $(0,\infty)$. In that case, the measure $\p_0(dr)$ can be written as $\p^c_0(r)dr$ where the density function $\p^c_0(r)$ is given in Definition \ref{def_marginal}
Using the asymptotic behavior of the gamma function for small \eqref{eq:gamma_pole} and large \eqref{eq:asympt_gamma} imaginary parts, we see that for any fixed $u,v$ with $u+v>0$, there exists a constant $C>0$ such that
\begin{equation}\label{eqpcbound}
 C^{-1} f(r) \leq \p^c_0(r)\leq C f(r) \qquad \mathrm{where} \quad f(r):=\begin{cases}r^{1/2}&r\geq 1\\ r^{u+v-1}&r\in (0,1)\end{cases}.
\end{equation}
Substituting this into the right-hand side of \eqref{eq:phirdd} and using the integrability of $r^{-1/2}$ at 0, and the decay coming from $e^{-\frac{(1-X_{d})}{4} r} $ at infinity, we find that the right-hand side of \eqref{eq:phirdd} is finite.

Turning to the positivity of  $\tilde\phi_{u,v}(\vec{c},\vec{X})$ for $\vec{c}\in (0,\Cduv)^d$,
\begin{equation}
\label{eq:lowerboundstrict}
\tilde\phi_{u,v}(\vec{c},\vec{X})
\geq \G(\vec{2}; \vec{c};\vec{X})\int\limits_{[1,2]^{d+1}} \prod_{i=1}^{d}\p^{c,c}_{s_{i+1},s_{i}}(r_{i+1},r_{i})dr_i\cdot \p^c_{s_{d+1}}(r_{d+1})dr_{d+1}.
\end{equation}
The inequality follows from the definition of $\tilde\phi_{u,v}(\vec{c},\vec{X})$ and the positivity of the integrand therein  in conjunction with the lower bound $\G(\vec{r}; \vec{c};\vec{X})\geq \G(\vec{2}; \vec{c};\vec{X})$  for $\vec{r}\in[1,2]^{d+1}$ and the fact that on $[1,2]$ the transition and marginal distributions are absolutely continuous. There exists $C>0$ so that $\G(\vec{2}; \vec{c};\vec{X})\geq C$ for $\vec{c}\in (0,\Cduv)^d$.

It remains to show that the integral in the final line of \eqref{eq:lowerboundstrict} over $[1,2]^{d+1}$ is strictly positive for any $\vec{c}\in (0,\Cduv)^{d}$. Our assumption on $\vec{c}\in (0,\Cduv)^d$ implies that $s_{k}-s_{k+1}=c_k>0$. Using this and the explicit formulas for the marginal and transition density functions (see Definitions \ref{def_marginal} and \ref{def_transition}), we see that for all $\vec{c}\in (0,\Cduv)^d$ there exists a constant $C>0$ such that $\p^{c,c}_{s_i,s_{i+1}}(r_i,r_{i+1}),\p^c_{s_{d+1}}(r_{d+1})\geq C$ for all $\vec{r}\in [1,2]^{d+1}$.
\end{proof}

Turning to the proof of Proposition \ref{prop_ASEP_gen_function}, by Lemma \ref{lem_strict_positive} we see that in order to prove $\phi^{(N)}_{u,v}(\vec{c},\vec{X}) \to \phi_{u,v}(\vec{c},\vec{X})$, it suffices to show that
 $\tilde\phi^{(N)}_{u,v}(\vec{c},\vec{X}) \to \tilde\phi_{u,v}(\vec{c},\vec{X})$ for $\vec{c}\in (0,\Cduv)^{d}$ for all $k\in\llbracket 1,d\rrbracket$, and for $\vec{c}=\vec{0}$. We will just deal with the first case, since the second case where $\vec{c}=\vec{0}$ follows similarly. The idea in the proof is to use the convergence lemmas in Section \ref{sec:asymptoticupos} to show point-wise and dominated convergence of the integrand in $\phi^{(N)}_{u,v}(\vec{c},\vec{X})$ to that of $\phi_{u,v}(\vec{c},\vec{X})$. There is a bit of bookkeeping  since the measures we consider have mixed discrete and absolutely continuous support.

 In the definition of $\tilde\phi^{(N)}_{u,v}(\vec{c},\vec{X})$ we can insert a multiplicative factor $1$ as $\prod_{i=1}^{d+1} (\mathbf{1}_{r_i\geq 0}+\mathbf{1}_{r_i<0})$. Expanding leads to $2^{d+1}$ terms, each one corresponding to a choice of whether we integrate over the continuous part of $\hat\pi^{(N)}$ when $r_i\geq 0$ or the atomic part when $r_i<0$. Explicitly,
\begin{align}\label{eq:decomtildephi}
\tilde\phi^{(N)}_{u,v}(\vec{c},\vec{X}) = \!\!\!\!\sum_{I\subset \llbracket 1,d+1\rrbracket} \sum_{r_I} \int_{r_J} \G^{(N)}(\vec{r}; \vec{c};\vec{X}) &\prod_{i=1}^{d}\hat\pi^{(N),a_{i+1},a_{i}}_{s_{i+1},s_{i}}(r_{i+1},r_{i})\qquad\\
&\times N^{u+v} \hat\pi^{(N),a_{d+1}}_{s_{d+1}}(r_{d+1}).
\end{align}
Here $J$ is the complement of $I$ in $\llbracket 1,d+1\rrbracket$; the sum over $r_I$ is an $|I|$-fold summation over $r_{i}\in \Supph^{(N),d}_{s_{i}}$ for $i\in I$; the integral over $r_J$ is really a $|J|$-fold integral as the $r$ variables with indices in $J$ vary in $\Supph^{(N),c}=[0,4N]$ (note, we have suppressed the $\prod_{j\in J} dr_{j}$ symbols); and the variables $a_{i}$ take values in the set of symbols $\{c,d\}$ with $a_i = d$ when $i\in I$ and $a_i=c$ when $i\in J$.

As $N$ varies, the number of atoms in the atomic parts of $\hat\pi^{(N)}$ does not change (i.e. $\Supph^{(N),d}_{s}$ is independent of $N$). Therefore, the form of the decomposition \eqref{eq:decomtildephi} remains stable with $N$. Moreover, $\tilde\phi_{u,v}(\vec{c},\vec{X})$ admits the same form of decomposition. Thus, in order to show the convergence of $\tilde\phi^{(N)}$ to $\tilde\phi$, it suffices to show that for any choice of $I$, the corresponding sum over the $r_I$ and integral in the remaining $r_J$ variables in \eqref{eq:decomtildephi} converges to its proposed limit. The sum over $r_I$ can be indexed in terms of the labels of the elements chosen from each $\Supph^{(N),d}_{s_{i}}$. By labels we mean that for each $i\in I$, we may identify $r_i\in \Supph^{(N),d}_{s_{i}}$ by a label $\ell_i$ such that $r_i = x^u_{\ell_i}(s_{i})$ or $r_i = x^v_{\ell_i}(s_{i})$ (depending on whether we are dealing with $u$ or $v$ atoms). For each choice of $I$ there are a finite number of choices of labels $\{\ell_i\}_{i\in I}$. If for every choice of $I$ and every choice of labels, we can prove convergence of the remaining integral in the $r_J$ variables, then we will have achieved our goal of proving the point-wise limit in \eqref{eq:phi_negative}.

Lemmas \ref{marginal_atoms} and \ref{transition_atoms} show that the locations and masses of the atoms of the $\hat\pi^{(N)}$ measures converge as $N\to \infty$ to those of the $\p$ measures and Lemmas \ref{lem:p_zero} and \ref{lem:p} show the point-wise convergence of the densities of the absolute continuous parts of the $\hat\pi^{(N)}$ measures to those of the $\p$ measures.  Equation \eqref{eq:GNconv} in Lemma \ref{lem:G_bound} shows how $\G^{(N)}$ converges point-wise to $\G$. In light of these results it follows that for each choice of $I$ and labels $\{\ell_i\}_{i\in I}$, the integrand in \eqref{eq:decomtildephi} converges point-wise to its proposed limit. To show convergence of the integral itself, it suffices to demonstrate a dominating function and then use the Lebesgue dominated convergence theorem.

We will assume below that $X_d<1$. When $X_d=1$, the functions $\G^{(N)}$ and $\G$ do not depend on $r_{d+1}$ and thus we can integrate out the $r_{d+1}$ variable. Since $\int  \hat\pi^{(N)}_{s_{d+1},s_{d+1}}(r_{d},dr_{d})  \hat\pi^{(N)}_{s_{d+1}}(dr_{d+1}) = \hat\pi^{(N)}_{s_{d}}(dr_{d})$, the formula for $\tilde\phi$ reduces to a similar one but with one fewer variable, which can be bounded in the same manner as below.

Since $\Supph^{(N),d}_{s}$ has a uniform lower bound as $s$ varies, and $\Supph^{(N),c}_{s}=[0,4N]$, we can use \eqref{eq:Gbound} of Lemma \ref{lem:G_bound} to show that there exist $c,C>0$ such that
\begin{equation}\label{eq:Gnexpbound}
\G^{(N)}(\vec{r}; \vec{c};\vec{X})\leq C e^{-c(r_1+\cdots r_{d+1})}
\end{equation}
as the $r_k$ vary over the $\Supph^{(N)}_{s_{k}}$ for $k\in \llbracket 1,d+1\rrbracket$. Using this bound along with Lemmas \ref{lem:p_zero}, \ref{marginal_atoms}, \ref{lem:p} and \ref{transition_atoms} we arrive at the following bound: Fix $\eta=3/2$, then there exists $N_0\in \Z_{\geq 1}$ and $c,c',C,C',\chi\in \R_{>0}$ such that for all $N>N_0$, all $r_{i}\in \Supph^{(N),d}_{s_{i}}$ with $i\in I$ and all $r_j\in \Supph^{(N),c}=[0,4N]$ for $j\in J$, we have
\begin{align}\label{eq:Gnlongbound}
&\left|\G^{(N)}(\vec{r}; \vec{c};\vec{X}) \prod_{i=1}^{d}\hat\pi^{(N),a_{i+1},a_{i}}_{s_{i+1},s_{i}}(r_{i+1},r_{i})\cdot N^{u+v}\hat\pi^{(N),a_{d+1}}_{s_{d+1}}(r_{d+1}) \right|\\
&\quad \leq C e^{-c(r_1+\cdots r_{d+1})}
\prod_{j\in J} e^{CN^{-\chi}(1+\sqrt{r_j})^{\eta}}
\prod_{i=1}^{d}\p^{a_{i+1},a_{i}}_{s_{i+1},s_{i}}(r_{i+1},r_{i})\cdot \p^{a_{d+1}}_{s_{d+1}}(r_{d+1})\\
\nonumber &\quad \leq C' e^{-c'(r_1+\cdots r_{d+1})}
\prod_{i=1}^{d}\p^{a_{i+1},a_{i}}_{s_{i+1},s_{i}}(r_{i+1},r_{i})\cdot \p^{a_{d+1}}_{s_{d+1}}(r_{d+1})
\end{align}
 For the first inequality we used \eqref{eq:Gnexpbound} along with the bounds from Lemmas \ref{lem:p_zero} and  \ref{lem:p}; the second inequality uses $e^{-cr}  e^{CN^{-\chi}(1+\sqrt{r})^{\eta}} \leq C' e^{-c' r}$ for a large enough $C'>0$ and a small enough $c'>0$.

The point of \eqref{eq:Gnlongbound} is that it now provides us with an $N$-independent dominating function. If we can show that for each $I\subset \llbracket 1,d+1\rrbracket$,
\begin{equation}\label{eq:lasttoprove}
 \sum_{r_I} \int_{r_J} C' e^{-c'(r_1+\cdots r_{d+1})}
\prod_{i=1}^{d}\p^{a_{i+1},a_{i}}_{s_{i+1},s_{i}}(r_{i+1},r_{i})\cdot \p^{a_{d+1}}_{s_{d+1}}(r_{d+1}) <\infty
\end{equation}
we will be done owing to the point-wise convergence we have already shown.

Let us first consider \eqref{eq:lasttoprove} with $I$ such that $d+1\in I$. In that case, the $r_{d+1}$ variable is summed over the finite number of atoms in $\Suppp^{d}_0$, each of which has a finite mass. Thus, for such terms in  \eqref{eq:lasttoprove} we can bound $\p^{a_{I(d+1)}}_{s_{d+1}}(r_{d+1})\leq C$. All of the other terms $\p^{a_{i+1},a_{i}}_{s_{i+1},s_{i}}(r_{i+1},r_{i})$ are either densities or masses of probability measures. Owing to this and the fact that $e^{-c'r}$ is upper bounded by a constant for $r\in \Suppp_s$ it immediately follows that the sum over $r_I$ and integral over $r_J$ is likewise bounded by a constant.

For $I$ such that $d+1\notin I$, the $r_{d+1}$ variable is integrated over $\Suppp^{c} = (0,\infty)$. The term $\p^{a_{d+1}}_{s_{d+1}}(r_{d+1})$ now represents the density of that infinite measure. As in the previous paragraph we may integrate/sum out all of the other variables $r_1,\ldots, r_d$ at the cost of a constant factor. Thus, we are left to bound
$\int_0^{\infty} e^{-c' r}  \p^{c}_{0}(r) dr<\infty$, which is done precisely as in the proof of Lemma \ref{lem_strict_positive}.
This shows that the right-hand side of \eqref{eq:Gnlongbound} is a dominating function, completing the proof of Proposition \ref{prop_ASEP_gen_function}.

\subsection{Proof of lemmas in Section \ref{sec:asymptoticupos}}\label{sec:applications_lemmas}

\subsubsection{Proof of Lemma \ref{lem:G_bound}}
To prove \eqref{eq:Gbound} we show that
\begin{align}
\G^{(N)}(\vec{r}; \vec{c};\vec{X})
&= {\bf 1}_{\vec{r}\in \R^{d+1}_{\leq 4N}} \cdot \prod_{k=1}^{d+1}\left(1+ \sinh^2\left (\tfrac{s_k}{2\sqrt N}\right)-\tfrac{r_k}{4N}\right)^{N(X^{(N)}_k-X^{(N)}_{k-1})}\\
&\leq \prod_{k=1}^{d+1} \exp\left(\tfrac{N(X^{(N)}_k-X^{(N)}_{k-1})}{4N}\left(4N\sinh^2\left(\tfrac{s_k}{2\sqrt N}\right)-r_k\right)\right)\\
&\leq C \prod_{k=1}^{d+1} e^{-(X^{(N)}_k-X^{(N)}_{k-1})\frac{r_k}{4}} \leq C \prod_{k=1}^{d+1}e^{-(X_k-X_{k-1})\frac{r_k}{4} + \frac{|r_k|}{2N}}.
\end{align}
The first equality is by the definition of $\G^{(N)}(\vec{r};
\vec{c};\vec{X})$ and the hyperbolic trigonometric identity
$\frac{\cosh(x)+1}{2} = 1+ \sinh^2(\tfrac{x}{2})$. The next inequality
uses that $(1+x)^a\leq e^{ax}$ for $a\in \Z_{\geq 0}$ and $x\in
\R_{\geq -1}$. In particular, we take $a= N(X^{(N)}_k-X^{(N)}_{k-1})$
and $x= \sinh^2\left (\tfrac{s_k}{2\sqrt
    N}\right)-\tfrac{r_k}{4N}$. The $\sinh^2$ is always non-negative
and due to the indicator function ${\bf 1}_{\vec{r}\in \R^{d+1}_{\leq
    4N}}$ we may assume that $\tfrac{r_k}{4N}\leq 1$. After applying
the inequality we drop the indicator function. The next inequality
relies on the fact that $N\sinh^2\left(\tfrac{s_k}{2\sqrt N}\right)$
can be bounded above by a constant provided that the $s_k$ vary in a
compact set (which follows from the assumption on the $\vec{c}\in
I^d$). The constant $C$ will depend on the set $I$. The final
inequality uses the fact that $|X^{(N)}_k - X_k|\leq N^{-1},$ which means that we can replace the $X^{(N)}_k$ by their limiting values $X_k$ at the cost of introducing the factor $|r_k|/2N$ in the exponential.

Owing to the triangle inequality, to prove \eqref{eq:GNconv} it suffices to show that
\begin{equation}\label{eq:convp1}
 \lim_{N\rightarrow \infty}  \G^{(N)}(\vec{r}; \vec{c};\vec{X})=\G(\vec{r}; \vec{c};\vec{X}),\quad
 \lim_{N\rightarrow \infty} \left|\G^{(N)}(\vec{r}^{(N)}; \vec{c};\vec{X})-\G^{(N)}(\vec{r}; \vec{c};\vec{X})\right|=0.
\end{equation}
The first limit in \eqref{eq:convp1} follows immediately from Taylor expansion of the $\sinh$ function and the convergence of $(1+x/N)^N$ to $e^{x}$  and $X^{(N)}_k$ to $X_k$.

The second limit in \eqref{eq:convp1} will make use of two elementary inequalities. The first is that for all $a_1,\ldots, a_{d+1},b_{1},\ldots,b_{d+1}\in\R$ bounded in absolute value by $M\in \R_{>0}$,
$\left|\prod_{k=1}^{d+1} a_k -\prod_{k=1}^{d+1} b_k\right| \leq M^d \sum_{k=1}^{d+1} |a_k-b_k|.
$
We apply this inequality with
$
a_k^{(N)}= \Big(1+ \sinh^2\left (\tfrac{s_k}{2\sqrt N}\right)-\tfrac{r^{(N)}_k}{4N}\Big)^{N(X^{(N)}_k-X^{(N)}_{k-1})}$
and
$b_k^{(N)} = \Big(1+ \sinh^2\left (\tfrac{s_k}{2\sqrt N}\right)-\tfrac{r_k}{4N}\Big)^{N(X^{(N)}_k-X^{(N)}_{k-1})}.
$
For $\vec{X},\vec{s}$ fixed, it is easy to see that we can find some $M$ large enough so that $|a^{(N)}_k|,|b^{(N)}_k|\leq M$ for all $k\in \llbracket 1,d+1\rrbracket$ and all $N\in\Z_{\geq 1}$.
Thus, it suffices to show that $\lim_{N\to \infty}|a^{(N)}_k-b^{(N)}_k|=0$. Notice that $a^{(N)}_k$ can be written in the form $(1+\tilde{a}^{(N)}_k/L)^L$ where $L= N(X_k-X_{k-1})$. There exists some compact interval $I$ such that $\tilde{a}^{(N)}_k\in I$ for all  $N\in\Z_{\geq 1}$. Likewise $b^{(N)}_k$ can be written in the same form in terms of $\tilde{b}^{(N)}_k$ and we can find some compact interval $I$ so that $\tilde{b}^{(N)}_k\in I$ for all  $N\in\Z_{\geq 1}$ as well. The convergence $\lim_{N\to \infty} |r^{(N)}-r|=0$ implies that $\lim_{N\to \infty}|\tilde{a}^{(N)}_k-\tilde{b}^{(N)}_k|=0$. To finish the proof, we use the following elementary inequality: For any compact interval $I\subset \R$ there exists a constant $C>0$ such that for all $L$ large enough and $\tilde{a},\tilde{b}\in I$, $|(1+\tilde{a}/L)^L -(1+\tilde{b}/L)^L| \leq C |\tilde{a}-\tilde{b}|$. This implies the second limit in \eqref{eq:convp1} and completes the proof.

\subsubsection{Notation for asymptotics}\label{sec:notforasym}
Recall $\AqP^{\pm}[\kappa;z]$ from \eqref{eq:AqPp} and \eqref{eq:AqPm}. For $k,N\in \Z_{\geq 1}$ and $z_1,\ldots, z_k\in \C$, define
\begin{equation}\label{eq:AzN}
\AqP^{\pm}_N[z_1,\cdots, z_k]=\sum_{i=1}^k \AqP^{\pm}[\tfrac{2}{\sqrt N},z_i].
\end{equation}
Here we have fixed that $\kappa=\tfrac{2}{\sqrt N}$, in which case $q=e^{-\kappa}$.

Setting $m=1$ in Proposition \ref{factorials} shows that for $q=e^{-\kappa}$ and $z\in\C$,
\begin{equation}\label{eq:qqerrorref}
\log(\pm q^z;q)_{\infty} = \AqP^{\pm}[\kappa;z] +  \Error^{\pm}_1[\kappa;z]
\end{equation}
where $\Error^{\pm}_1[\kappa;z]$ satisfies \eqref{eq:errorbounds}.
For $k,N\in \Z_{\geq 1}$ and $z_1,\ldots, z_k\in \C$, define
\begin{equation}\label{eq:errorN}
\Err^\pm_N[z_1,\cdots, z_k]=\sum_{i=1}^k \Error^{\pm}_1[\tfrac{2}{\sqrt N},z_i].
\end{equation}

\subsubsection{Proof of Lemma \ref{lem:p_zero}}

For $r\in \R_{\geq 0}$, $\lim_{N\to \infty} \ErrLem^{(N), c}_{t}(r)=0$, so \eqref{eq:ptlem73} follows from \eqref{eq:lemma82ineq}.
It remains to prove \eqref{eq:lemma82ineq}.
We proceed in three steps. In step 1, we write down $\hat \pi^{(N), c}_{t}(r)$. In step 2, we further rewrite $\hat \pi^{(N), c}_{t}(r)$ in terms of a limiting term and error terms as in Section \ref{sec:notforasym}. In step 3, we control the error terms using the bounds in Proposition \ref{factorials}.

\smallskip
\noindent {\bf Step 1.} For $r\in \Supph^{(N),c}=[0,4N]$, in light of \eqref{eq:AW}, \eqref{eq:def_pi} and \eqref{eq:hatp} we have
\begin{equation}\label{eq:pirelations}
\hat \pi^{(N), c}_{t}(r) = \hat \pi^{(N), c}_{t;v,1,u,1}(r)
\end{equation}
where we define
\begin{align}\label{eq:firstexpand}
  &\qquad\hat \pi^{(N), c}_{t;v,\tilde{v},u,\tilde{u}}(r)  := \frac{1}{8\pi N  \sqrt{\frac{r}{4N}} \sqrt{1-\frac{r}{4 N}}}\times  \\
  & \dfrac{\big(q, -q^{v+\tilde{v}+t}, q^{v+u}, -q^{v+\tilde{u}}, -q^{u+\tilde{v}}, q^{\tilde{v}+\tilde{u}}, -q^{u+\tilde{u}-t}\big)_{\infty} \Big\lvert\big(q^{\iu  \sqrt N \theta_r }\big)_{\infty}\Big\rvert^2}
  {\big(q^{v+\tilde{v}+u+\tilde{u}}\big)_{\infty}\Big\lvert\big(q^{v+t/2+\iu  \sqrt N\frac{\theta_r}{2}},-q^{\tilde{v}+t/2+\iu \sqrt N \frac{\theta_r}{2}} , q^{u-t/2+\iu  \sqrt N \frac{\theta_r}{2}}, -q^{\tilde{u}-t/2+\iu  \sqrt N \frac{\theta_r}{2}},\big)_{\infty}\Big\rvert^2}
\end{align}
and where we have used the notation (suppressing the $N$ dependence above)
\begin{equation}\label{eq:thetardef}
\theta_r  = \theta^{(N)}_r:= \arccos\left(1-\frac{r}{2N}\right).
\end{equation}

Observe that if we send $r\mapsto 4N-r$, then $\theta_r\mapsto \pi -\theta_r$ and thus also $q^{\iu  \sqrt N \frac{\theta_r}{2}}\mapsto -q^{\iu  \sqrt N \frac{\theta_r}{2}}$ and $q^{\iu  \sqrt N \theta_r}\mapsto q^{\iu  \sqrt N \theta_r}$. These transformations imply that $\hat \pi^{(N), c}_{t;v,\tilde{v},u,\tilde{u}}(r)$ transforms under this change of variables by swapping the tilde and non-tilde variables:
\begin{equation}\label{eq:piinvar}
\hat \pi^{(N), c}_{t;v,\tilde{v},u,\tilde{u}}(4N-r) = \hat \pi^{(N), c}_{t;\tilde{v},v,\tilde{u},u}(r).
\end{equation}

In light of this transformation we will now consider the asymptotics behavior of $\hat \pi^{(N), c}_{t;v,\tilde{v},u,\tilde{u}}(r)$ for $r\in [0,2N]$. We will show (generalizing \eqref{Lem83eq}) that, provided $u+v>0$ and $\tilde{u}+\tilde{v}>0$, $r\in [0,2N]$
\begin{equation}
N^{u+v}\sqrt{1-\frac{r}{4N}}\, \hat \pi^{(N), c}_{t;v,\tilde{v},u,\tilde{u}} =\p^c_{t;v,u}(r)\,\cdot\, e^{\ErrLem^{(N), c}_{t;v,\tilde{v},u,\tilde{u}}(r)}
\end{equation}
where the error term satisfies the bound
\begin{equation}\label{eq:lemma82ineqprime}
\Big|\ErrLem^{(N), c}_{t;v,\tilde{v},u,\tilde{u}}(r)\Big| \leq CN^{-\chi}(1+\sqrt{r})^{\eta}.
\end{equation}
as in  \eqref{eq:lemma82ineq} and where $\p^c_{t;v,u}=\p^c_{t}$ is given in Definition \ref{def_marginal} (and does not depend on $\tilde{v}$ or $\tilde{u}$).
By combining this bound, the transformation \eqref{eq:piinvar} and the growth bound \eqref{eqpcbound} on $\p^c_{t}$ we can easily deduce that equation \eqref{Lem83eq} also holds for $r\in [2N,4N]$ with the claimed error bound \eqref{eq:lemma82ineq}. Thus, we focus the rest of this proof on demonstrating \eqref{eq:lemma82ineqprime} under the restriction $r\in [0,2N]$.


\smallskip
\noindent {\bf Step 2.}
By Taylor expanding around $r=0$ we can write
\begin{equation}\label{eq:thetataylor}
\sqrt{N} \theta^{(N)}_r = \sqrt{r} + \Err^{\theta}_N(r),
\end{equation}
where $\Err^{\theta}_N(r)$ is the remainder.
With this and \eqref{eq:qqerrorref}, we rewrite \eqref{eq:firstexpand}  as
\begin{align}\label{eq:firstexpand2}
  \qquad\log\left[\sqrt{1-\frac{r}{4N}}\, \hat \pi^{(N), c}_{t;v,\tilde{v},u,\tilde{u}}(r)\right] = - \log\left[8\pi N  \sqrt{\dfrac{r}{4N}}\,\right] + \AqP_N(\sqrt{r})+ \Err^{\AqP}_N(r)+\Err^{\mathcal{E}}_N(r).
\end{align}
In the above formula, for $a\in \R_{\geq 0}$, we have defined
\begin{align}
\AqP_N(a) :=& \AqP^+_N\Big[1, v+u, \tilde{v}+\tilde{u}, \iu  a, -\iu a\Big]+\AqP^-_N\Big[v+\tilde{v}+t, v+\tilde{u}, u+\tilde{v}, u+\tilde{u}-t \Big]\\
&-\AqP^+_N\Big[v+u+\tilde{v}+\tilde{u}, v+\dfrac{t}{2}+\iu \dfrac{a}{2}, v+\dfrac{t}{2}-\iu \dfrac{a}{2}, u-\dfrac{t}{2}+\iu \dfrac{a}{2}, u-\dfrac{t}{2}-\iu \dfrac{a}{2}\Big]\\
&-\AqP^-_N\Big[\tilde{v}+\dfrac{t}{2}+\iu \dfrac{a}{2}, \tilde{v}+\dfrac{t}{2}-\iu \dfrac{a}{2}, \tilde{u}-\dfrac{t}{2}+\iu \dfrac{a}{2}, \tilde{u}-\dfrac{t}{2}-\iu \dfrac{a}{2}\Big].\label{eq:AqPlong}
\end{align}

Here $\Err^{\AqP}_N(r)$ comes from the Taylor expansion \eqref{eq:thetataylor} and is given by
\begin{equation}\label{eq:Err1def}
\Err^{\AqP}_N(r) := \AqP_N\left(\sqrt{r} + \Err^{\theta}_N(r)\right)- \AqP_N\big(\sqrt{r}\big).
\end{equation}
Define the function $\mathcal{E}_N(a)$ exactly as in \eqref{eq:AqPlong}, except with the $\AqP$ symbol replaced by $\Err_N$, as in \eqref{eq:errorN}. $\Err^{\mathcal{E}}_N(r)$ comes from Proposition  \ref{factorials} and is
\begin{equation}\label{eq:error2def}
\Err^{\mathcal{E}}_N(r) = \mathcal{E}_{N}\left(\sqrt{r} + \Err^{\theta}_N(r)\right) =  \mathcal{E}_{N}\left(\sqrt{N} \theta^{(N)}_r \right) .
\end{equation}

The first two terms on the right-hand side of \eqref{eq:firstexpand2} can be simplified
considerably using the explicit expressions for  $\AqP^{\pm}[\kappa;z]$ from  \eqref{eq:AqPp} and \eqref{eq:AqPm}. In particular,  when $\AqP_N(\sqrt{r})$ is expanded, all of the terms that have a prefactor $\pi^2$ coming from  \eqref{eq:AqPp} and \eqref{eq:AqPm} end up canceling out. Combining the remaining terms we find that
\begin{align}
 &- \log\left[8\pi N  \sqrt{\dfrac{r}{4N}} \,\right] + \AqP_N(\sqrt{r}) =\\
 & \log\left[\frac{\Gamma(v+u+\tilde{v}+\tilde{u})}{\Gamma(v+u)\Gamma(\tilde{v}+\tilde{u})}\frac{\left|\Gamma\left(u-\frac{t}{2}+\iu \frac{\sqrt{r}}{2},v+\frac{t}{2}+\iu \frac{\sqrt{r}}{2}\right)\right|^2}{8 \pi \sqrt{r}\,\cdot\,  \left|\Gamma\left(\iu \sqrt{r}\right)\right|^2}   \right]
 -(v+u)\log(N).
\end{align}
Recalling \eqref{eq:pirelations}, it follows by taking $\tilde{v}=\tilde{u}=1$ in the above formula and using $\tfrac{\Gamma(v+u+2)}{\Gamma(v+u)} = (v+u)(v+u+1)$ that (recall also the formula in Definition \ref{def_marginal} for $\p^c_{t}$)
\begin{equation}\label{eq:logNabove}
\log\left[N^{u+v} \sqrt{1-\frac{r}{4N}}\,\hat \pi^{(N), c}_{t}(r)\right] = \log\left[\p^c_{t}(r) \right] + \Err^{\AqP}_N(r)+\Err^{\mathcal{E}}_N(r).
\end{equation}
Similarly, using \eqref{eq:piinvar} we see that
\begin{equation}\label{eq:logNabove2}
\log\left[N^{2} \sqrt{1-\frac{r}{4N}}\,\hat \pi^{(N), c}_{t}(4N-r)\right] = \log\left[\p^c_{t}(r) \right] + \Err^{\prime\AqP}_N(r)+\Err^{\prime\mathcal{E}}_N(r)
\end{equation}
where $\Err^{\prime\AqP}_N(r)$ and $\Err^{\prime\mathcal{E}}_N(r)$ are defined exactly as $\Err^{\AqP}_N(r)$ and $\Err^{\mathcal{E}}_N(r)$ are above but with $(v,u)$ and $(\tilde{v},\tilde{u})$ swapped.

\smallskip
\noindent {\bf Step 3.}
It remains to bound for $r\in [0,2N]$ the two error terms $\Err^{\AqP}_N(r)$ and $\Err^{\mathcal{E}}_N(r)$ in \eqref{eq:logNabove} and $\Err^{\prime\AqP}_N(r)$ and $\Err^{\prime\mathcal{E}}_N(r)$ in \eqref{eq:logNabove2}. The analysis is exactly the same in both cases so we will just focus on the first set of error terms. To be precise, in this step we will show that for all $\eta>1$, there exists $N_0\in \Z_{\geq 1}$ and  $C,\chi>0$ such that for all $N>N_0$ and $r\in [0,2N]$
\begin{equation}\label{eq:ErrAEbounds}
\left|\Err^{\AqP}_N(r)\right|,\left|\Err^{\mathcal{E}}_N(r)\right| \leq C N^{-\chi}(1+\sqrt{r})^{\eta}.
\end{equation}
From this and \eqref{eq:logNabove}, the desired bound in \eqref{eq:lemma82ineq} immediately follows.  Thus we will show \eqref{eq:ErrAEbounds}. In demonstrating those bounds we will also need to control $\Err^{\theta}_N(r)$, so we start with that.

\smallskip
\noindent {\bf Bounding $|\Err^{\theta}_N(r)|$.}
For $r\in [0,2N]$,
\begin{equation}\label{eq:Nsqrt5}
\sqrt{N} \theta^{(N)}_r\leq \sqrt{\frac{\pi^2}{8}r}\qquad\mathrm{and}\qquad \Err^{\theta}_N(r)\geq 0.
\end{equation}
The first inequality in \eqref{eq:Nsqrt5} is equivalent to the inequality
\begin{equation}\label{eq:arccosineq}
\arccos(1-x/2) \leq \sqrt{\frac{\pi^2}{8}x} \qquad\textrm{ for }x\in [0,2]
\end{equation} which can easily be shown by matching the values at $x=0$ and $x=2$ of both sides and then showing that the derivative of the difference strictly decreases (hence the difference is strictly concave). 
The second inequality in \eqref{eq:Nsqrt5} is equivalent to $\arccos(1-x/2)-\sqrt{x}\geq 0$ for $x\in [0,2]$ and is also shown by taking derivatives.

Now we claim that for all $\eta>1$ there exists $N_0\in \Z_{\geq 1}$ and $C,\chi>0$ such that for all $N>N_0$ and $r\in \Supph^{(N),c}=[0,2N]$,
\begin{equation}\label{eq:Err2bound}
\Err^{\theta}_N(r) \log(N) \leq C  N^{-\chi}(1+\sqrt{r})^{\eta}.
\end{equation}
Note that by \eqref{eq:Nsqrt5}, $\Err^{\theta}_N(r)\geq 0$. Changing variables $x = r/N$, \eqref{eq:Err2bound} reduces this to
\begin{equation}\label{eq:Err2bound2}
\arccos\left(1-\frac{x}{2}\right)-\sqrt{x} \leq \frac{C (1+\sqrt{Nx})^{\eta}}{N^{\chi +1/2} \log(N)}\qquad\mathrm{for} \quad x\in[0,2].
\end{equation}
We split the demonstration of  \eqref{eq:Err2bound2} into two cases.
Let $x_N$ be such that $\tfrac{(1+\sqrt{Nx_N})^{\eta-1}}{N^{\chi}\log(N)}=1$. Then, since $\tfrac{(1+\sqrt{Nx})^{\eta-1}}{N^{\chi}\log(N)}\geq 1$ for $x\in [x_N,2]$, we have on that interval that
\begin{equation}
\arccos\left(1-\frac{x}{2}\right)-\sqrt{x} \leq C \sqrt{x}
\leq C\frac{1+\sqrt{Nx}}{\sqrt{N}}  \frac{(1+\sqrt{Nx})^{\eta-1} }{N^{\chi}\log(N)}
\leq \frac{C (1+\sqrt{Nx})^{\eta}}{N^{\chi+1/2} \log(N)}
\end{equation}
where in the first inequality we can take $C=\sqrt{\pi^2/8}-1$ owing to \eqref{eq:arccosineq}. We do not require anything on the value of $\chi$. This shows \eqref{eq:Err2bound2} for $x\in [x_N,2]$.

Tuning to the case of $x\in [0,x_N]$, observe that $x_N$  goes to zero as $N$ grows and $x_N\leq N^{\frac{2\chi}{\eta-1} -1} \log(N)^{\frac{2}{\eta-1}}$. Provided that $\chi < \tfrac{\eta-1}{2}$ (so $x_N$ goes to zero as a power law in $N$), for $x\in [0,x_N]$ we can use Taylor expansion with remainder to show that there exists $N_0\in \Z_{\geq 1}$ and $C\in \R_{>0}$ such that for all $N>N_0$ and $x\in [0,x_N]$,
$
\arccos\left(1-\frac{x}{2}\right)-\sqrt{x}\leq C x^{3/2}.
$
From this we see that all that is left is to show that on the interval $x\in [0,x_N]$,
$x^{3/2} \leq C \frac{(1+\sqrt{Nx})^{\eta}}{N^{\chi+1/2} \log(N)}$.
Clearly this is true when $x=0$ (where both sides are zero). If we can show that the derivatives are likewise ordered on $x\in [0,x_N]$ then the above inequality will immediately follow. Calculating those derivatives we find that showing their ordering reduces to showing that
$
x\leq \frac{(1+\sqrt{Nx})^{\eta-1}}{N^{\chi}\log(N)}
$
which easily follows from Taylor expansion with remainder on $(1+\sqrt{Nx})^{\eta-1}$. Thus, we have shown \eqref{eq:Err2bound}.

\smallskip
\noindent {\bf Bounding $|\Err^{\mathcal{E}}_N(r)|$.} Recall that $\mathcal{E}_N(r)$ is defined in \eqref{eq:error2def} (the function $\mathcal{E}_N(a)$ is defined exactly as in \eqref{eq:AqPlong}, except with the $\AqP$ symbol replaced by $\Err_N$, as in \eqref{eq:errorN}) as a sum of many terms of the form $\Error^{\pm}_1[\kappa,z]$ for $\kappa= \tfrac{2}{\sqrt{N}}$ and for various choices of the variable $z$. Some of the assignments of the variable $z$ depend on $r$ while others do not (though may still depend on other variables like $u,v,t$). We refer to the former terms as Type (1) and the later as Type (2). For Type (2) terms, we see from \eqref{eq:errorbounds} in Proposition \ref{factorials} that for any $z$ fixed and $b\in (0,1)$, we can find a $\kappa_0>0$ and a $C>0$ so that for all $\kappa<\kappa_0$ (or equivalently we can find $N_0\in \Z_{\geq 1}$ so that for all $N>N_0$),
\begin{equation}
|\Error^{\pm}_{1}[\kappa,z]|\leq C(\kappa+\kappa^b) \leq C' N^{-b/2}.
\end{equation}
From this we see that the contribution of Type (2) terms to $|\Err^{\mathcal{E}}_N(r)|$ satisfies the bound in \eqref{eq:ErrAEbounds}.

Now, let us consider how Type (1) terms contributes to $|\Err^{\mathcal{E}}_N(r)|$. These terms take the form $\Error^{\pm}_{1}[\kappa,z]$ for either $z=z_c(r) = c \pm \iu \sqrt{N} \tfrac{\theta^{(N)}_r }{2}$ with some fixed real $c$ (e.g. $c=0$ or $c=u-\tfrac{t}{2}$) or $z=\pm \iu \sqrt{N} \theta^{(N)}_r$. Call the first type of term Type (1a) and the second Type (1b). Observe that for $r\in [0,2N]$, $\theta^{(N)}_r\in [0,\pi/2]$. This means that $\sqrt{N} \theta^{(N)}_r\in [0,\frac{\pi}{\kappa}]$. For Type (1a) terms this range is further divided by $2$ and hence $|\textup{Im}(z)|\leq \frac{\alpha}{\kappa}$ for $\alpha=\pi/2$. Since this $\alpha<\pi$, we can apply  \eqref{eq:errorbounds} from Proposition \ref{factorials} to show that for any $b\in (0,1)$ and $\e\in (0,1/2)$, there exists $C,\kappa_0>0$ such that for all $\kappa<\kappa_0$ and  $r\in [0,2N]$
\begin{equation}\label{eq:Errbgg}
\Big|\Error^{\pm}_{1}[\kappa,z]\Big| \leq C\left(\kappa(1+|z|)^2 + \kappa^b (1+|z|)^{1+2b+\e}\right)
\end{equation}
where $z=c \pm \iu \sqrt{N} \tfrac{\theta^{(N)}_r }{2}$. For the Type (1b) terms, observe that it is only $\Error^{+}_{1}$ that arises in $\Err^{\mathcal{E}}_N(r)$. Recalling from Proposition \ref{factorials} that the bound \eqref{eq:errorbounds} on $\Error^{+}_m[\kappa, z]$ holds  with $|\textup{Im}(z)|<\tfrac{2\alpha}{\kappa}$ (for $\alpha\in (0,\pi)$) we see that \eqref{eq:Errbgg} holds when $\pm$ is restricted to $+$ and $z=c \pm \iu \sqrt{N}\theta^{(N)}_r$.

It just remains to massage the bound in \eqref{eq:Errbgg} into the claimed form. To do this, note that for Type (1) choices of $z=z(r)$, with $r\in [0,2N]$, we have that $|z(r)|\leq c+ C\sqrt{r}$ for some choices of $c,C>0$. This implies that for any $b\in (0,1)$ and $\e\in (0,1/2)$,  there exists $C>0$ such that for all $r\in [0,2N]$,
\begin{equation}\label{eq:Errabs}
\Big|\Error^{\pm}_{1}[\kappa,z_c(r)]\Big| \leq C\left(N^{-1/2}(1+\sqrt{r})^2 + N^{-b/2} (1+\sqrt{r})^{1+2b+\e}\right).
\end{equation}
The second term on the right-hand side above is already of the form $CN^{-\chi}(1+\sqrt{r})^{\eta}$ where $\eta$ can be taken arbitrarily close to $1$ by turning $b$ and $\e$ close to zero, and where $\chi = b/2$. The first term $N^{-1/2}(1+\sqrt{r})^2$ can also be put in this form since for a large enough constant $C>0$,
\begin{equation}
N^{-1/2}(1+\sqrt{r})^2 = N^{-1/2}(1+\sqrt{r})^{2-\eta} (1+\sqrt{r})^{\eta} \leq C N^{-\frac{\eta-1}{2}} (1+\sqrt{r})^{\eta}
\end{equation}
where the inequality uses that $1+\sqrt{r} \leq C'N^{1/2}$ for some suitably large constant $C'>0$. Taking $\chi = \tfrac{\eta-1}{2}$ puts this bound in the form $CN^{-\chi}(1+\sqrt{r})^{\eta}$. Therefore, we conclude that the contribution of Type (1) terms to $|\Err^{\mathcal{E}}_N(r)|$ satisfies the bound in \eqref{eq:ErrAEbounds}. Combining this with the previous conclusion for Type (2) terms, we arrive at the bound in \eqref{eq:ErrAEbounds} on $|\Err^{\mathcal{E}}_N(r)|$.

\smallskip
\noindent {\bf Bounding $|\Err^{\AqP}_N(r)|$.} Recall that $\Err^{\AqP}_N(r)$ is defined in \eqref{eq:Err1def} in terms of $\AqP_N(a)$ defined in \eqref{eq:AqPlong}. From \eqref{eq:AqPlong}, we see that $\AqP_N(\sqrt{r}+\Err^{\theta}_N(r))-\AqP_N(\sqrt{r})$ involves many cancellations. All of the terms in \eqref{eq:AqPlong} which do not depend on the argument $a$ immediately cancel when taking this difference. Recalling the definition of $\AqP^{\pm}[\kappa,z]$ from \eqref{eq:AqPp} and \eqref{eq:AqPm}, we also see that the terms in those functions involving $\pi^2$ do not depend on the $z$ argument and hence also cancel upon taking a difference. Let us take an accounting of which terms remain. From $\AqP^+$, we need to account for (1) the $(z-\tfrac{1}{2})\log(\kappa)$ terms and (2) the  $\log\left[\tfrac{\Gamma(z)}{\sqrt{2\pi}}\right]$ terms, while from $\AqP^-$, we need only account for (3) the $(z-\tfrac{1}{2})\log(2)$ terms. Let us consider each of these types of term separately and show how their contributions can be bounded by expressions of the form of the right-hand side of  \eqref{eq:ErrAEbounds}.

Type (1) terms contribute to $\AqP_N(\sqrt{r}+\Err^{\theta}_N(r))-\AqP_N(\sqrt{r})$ expressions of the form $\iu \frac{\Err^{\theta}_N(r)}{2} \log(\kappa)$. Since $\kappa= \tfrac{2}{\sqrt{N}}$, the magnitude of such terms is proportional to $\Err^{\theta}_N(r) \log(N)$.
The bound we established in \eqref{eq:Err2bound} implies that the contribution to $|\Err^{\AqP}_N(r)|$ of Type (1) terms can be bounded by $ C N^{-\chi} (1+\sqrt{r})^{\eta}$ provided $\chi$ is small enough. This is precisely of the form of the right-hand side of  \eqref{eq:ErrAEbounds}. Note that Type (3) terms which arise from $\AqP^-$ involve $\Err^{\theta}_N(r) \log(2)$. Since for $N\geq 2$, $\log(2)\leq \log(N)$, the argument above immediately controls those terms by $ C N^{-\chi} (1+\sqrt{r})^{\eta}$ as well.

All that remains is to control the contribution to $|\Err^{\AqP}_N(r)|$ from the Type (2) terms coming from $\log\left[\tfrac{\Gamma(z)}{\sqrt{2\pi}}\right]$ in $\AqP^+$. These contributions are of the form
$\log\left[\Gamma(z'_c(r))\right] - \log\left[\Gamma(z_c(r))\right]$
where $z'_c(r) = c \pm \iu \frac{\sqrt{r} + \Err^{\theta}_N(r)}{2}$
and $z_c(r) = c \pm \iu \frac{\sqrt{r}}{2}$ (case (1)), or $z'_c(r) = c \pm \iu \left(\sqrt{r} + \Err^{\theta}_N(r)\right)$
and $z_c(r) = c \pm \iu \sqrt{r}$ (case (2)). The proof in both cases is
identical (just the constants change), so we focus on the first case.  Similarly, the argument we
present works just as well for $\pm=+$ and $\pm=-$, so we will just
address the $+$ case. We claim that for any $c\in \R$ fixed, there exists $N_0\in \Z_{\geq 1}$ and $C,\chi >0$ such that for all $N>N_0$ and $r\in \Supph^{(N),c}= [0,4N]$
\begin{equation}\label{eq:Err1bounds2}
\Big|\log[\Gamma(z'_c(r))] - \log[\Gamma(z_c(r))]\Big| \leq C N^{-\chi}(1+\sqrt{r})^\eta.
\end{equation}
We will show this separately for $r\in [r_0,4N]$ and for $r\in [0,r_0]$, where $r_0$ is specified momentarily. The purpose of this split is that for large $r$ we can use the asymptotic behavior of the Gamma function in the imaginary direction. For small $r$ we can use the uniform continuity of the gamma function provided $c\notin \Z_{\leq 0}$. If $c\in \Z_{\leq 0}$, we need to account for the divergence from the pole, but after doing that we can still show the desired bound. We proceed with this argument now.

Observe that by the fundamental theorem of calculus, for $r>0$,
\begin{equation}\label{eq:digammaint}
\log[\Gamma(z'_c(r))] - \log[\Gamma(z_c(r))] = \int^{\frac{\sqrt{r} + \Err^{\theta}_N(r)}{2}}_{\frac{\sqrt{r}}{2}} \!\!\!\!\partial_y \log[\Gamma(c+\iu y)] dy =\int^{\frac{\sqrt{r} + \Err^{\theta}_N(r)}{2}}_{\frac{\sqrt{r}}{2}} \!\!\!\!\iu \psi(c+\iu y) dy
\end{equation}
where in the second equality  uses the Cauchy-Riemann equation to write $\partial_y \log[\Gamma(c+\iu y)] = \iu \psi(c+\iu y)$ where $\psi$ is the digamma function. Thus,
\begin{equation}\label{eq:Logpsibdd}
\Big|\log[\Gamma(z'_c(r))] - \log[\Gamma(z_c(r))]\Big| \leq  \tfrac{\Err^{\theta}_N(r)}{2}\!\!\!\! \sup_{y\in \big[\frac{\sqrt{r}}{2},\frac{\sqrt{r} + \Err^{\theta}_N(r)}{2}\big]} |\psi(c+\iu y)|.
\end{equation}
It follows from the asymptotic expansion for the digamma function in \cite[equation (5.11.2)]{NIST:DLMF} or \cite[page 18]{MOS} that any $c$ there exists $C,y_0>0$ such that for $|y|>y_0$, $|\psi(c+\iu y)|\leq C \log(|y|)$. Let $r_0=4y_0^2$ (so that $\tfrac{\sqrt{r_0}}{2} =y_0$). Thus  there exists $C,C'>0$ such that for $r>r_0$,
\begin{equation}
\Big|\log[\Gamma(z'_c(r))] - \log[\Gamma(z_c(r))]\Big| \leq C \Err^{\theta}_N(r) \log[|z'_c(r)|] \leq  C' \Err^{\theta}_N(r) \log(N)
\end{equation}
where the second inequality comes from the fact that for $r\in \Supph^{(N),c}=[0,4N]$ we can bound $|z'_c(r)|\leq C'' \log(N)$ for some $C''>0$. Now we may appeal to \eqref{eq:Err2bound} to bound the right-hand side above by $CN^{-\chi}(1+\sqrt{r})^{\eta}$ as desired in \eqref{eq:Err1bounds2}.

It remains to demonstrate \eqref{eq:Err1bounds2} when $r\in [0,r_0]$. First assume that $c\notin\Z_{\leq 0}$. For $z$ in any compact set away from $\Z_{\leq 0}$, analyticity implies that  $|\psi(z)|$ is uniformly continuous. Combining this observation with \eqref{eq:Logpsibdd}, we find that $\big|\log[\Gamma(z'_c(r))] - \log[\Gamma(z_c(r))]\big| \leq  C \Err^{\theta}_N(r)$. Again, owing to \eqref{eq:Err2bound} we may bound this above by $CN^{-\chi}(1+\sqrt{r})^{\eta}$ as desired in \eqref{eq:Err1bounds2}.

For the case when $c\in \Z_{\leq 0}$ we must appeal to the behavior of $\psi(z)$ near its poles $\Z_{\leq 0}$. As in \cite[page 14]{MOS}, we have that for $z$ in a vertical strip with real part in $[c-1/2,c+1/2]$, the function $\psi(z) = -(z-c)^{-1} + \psi_c(z)$ where $\psi_c(z)$ is analytic in the strip. This and \eqref{eq:digammaint} imply
\begin{equation}
\log(\Gamma(z'_c(r))) - \log(\Gamma(z_c(r)))  =  \int^{\frac{\sqrt{r} + \Err^{\theta}_N(r)}{2}}_{\frac{\sqrt{r}}{2}} \!\!\!\! \left(-y^{-1} + \psi_c(c+\iu y)\right) dy.
\end{equation}
This shows that there exists a constant $C>0$ such that
\begin{equation}
\Big|\log(\Gamma(z'_c(r))) - \log(\Gamma(z_c(r))) \Big| \leq C \log\left(\tfrac{\sqrt{r}+\Err^{\theta}_N(r)}{\sqrt{r}}\right) + C\Err^{\theta}_N(r).
\end{equation}
The second term $\Err^{\theta}_N(r)|$ is bounded by appealing to  \eqref{eq:Err2bound}. For the first,
\begin{equation}
 \log\left(\tfrac{\sqrt{r}+\Err^{\theta}_N(r)}{\sqrt{r}}\right)  = \log\left(1+ \tfrac{\Err^{\theta}_N(r)}{\sqrt{r}}\right)  \leq \tfrac{\Err^{\theta}_N(r)}{\sqrt{r}} \leq C' \cdot \tfrac{r}{N},
\end{equation}
for some constants $C,C'>0$.
Since $r\in [0,r_0]$, for any choice of $\eta>1$ and $\chi<1$ there is a constant $C>0$ such that $\tfrac{r}{N} \leq CN^{-\chi} (1+\sqrt{r})^{\eta}$ . This bound is of the form of \eqref{eq:Err1bounds2}. Putting together the cases considered above we have shown that \eqref{eq:Err1bounds2} holds. This completes the proof of \eqref{eq:ErrAEbounds} for $|\Err^{\AqP}_N(r)|$ and hence completes Step 3 and the proof of Lemma \ref{lem:p_zero}.

\subsubsection{Proof of Lemma \ref{marginal_atoms}}

The convergence of the location of the atoms in \eqref{eq:lem83a} and  \eqref{eq:lem83aprime} follows immediately from \eqref{eq:yhatscale} by Taylor expanding in $N$.
The convergence of the masses of these atoms in \eqref{eq:lem83b} follows immediately from the bounds in \eqref{eq:lem83ab}. Thus, our problem reduces to showing \eqref{eq:lem83ab}. We will deal with the case when $v+t/2<0$ since the other case of $u-t/2<0$ proceeds exactly in the same manner. From \eqref{eq:AWdmasses} we can write down the weights of the atoms as (recall the parametrization for $a,b,c,d$ from \eqref{eq:abcdqvals})
\begin{align}\label{eq:computingthemasses}
\hat\pi^{(N), d}_t(\xvN_0(t))&=\dfrac{\left(q^{-2 v-t}, -q^{u+1-t} , -q^{1+u}, q^{2}\right)_\infty}{\left(-q^{1-v}, q^{u-v-t }, -q^{1-v-t }, q^{2+v+u}\right)_\infty},\\
\frac{\hat\pi^{(N), d}_t(\xvN_j(t))}{\hat\pi^{(N), d}_t(\xvN_0(t))}&=\dfrac{\left(q^{2 v+t},
    -q^{v+1+t}, q^{v+u}, -q^{v+1}\right)_j\cdot\left(1- q^{2 v+2 j+t}\right)}
{\left(q,- q^{v}, q^{1-u+v+t}, -q^{v+t}\right)_j \cdot \left(1-q^{2v+t}\right)} \cdot\left(q^{-v-u-1}\right)^j,
\end{align}
where $j\in \llbracket 1,-v-t/2\rrbracket$. Since $v+t/2<0$ and $u+v>0$, there is no chance of division by zero in the above formulas.

Recalling the notation from \eqref{eq:AzN} and  \eqref{eq:errorN}, Proposition
\ref{factorials} with $m=1$ yields (here $(\AqP^\pm_N+ \Err^\pm_N)[\cdots] :=\AqP^\pm_N[\cdots]+\Err^\pm_N[\cdots]$)
\begin{align}
  &\log\left( N^{u+v}\hat\pi^{(N), d}_t(\xvN_0(t))\right)= (u+v)\log N+(\AqP^+_N+\Err^+_N)[-2v-t, 2]\\
  &-(\AqP^+_N+\Err^+_N)[u-v-t ,2+ v+u ] +(\AqP^-_N+\Err^-_N)[u+1-t , u+1 ]\\
  &- (\AqP^-_N+\Err^-_N)[1-v, 1-v-t].\label{eq:Nuvpihf}
\end{align}
Simplifying the $\AqP$ terms, the $N$ dependence drops and
\begin{align}
&\AqP^+_N[-2v-t, 2]-\AqP^+_N[u-v-t ,2+ v+u ]+\AqP^-_N[u+1-t , u+1 ]\\
&- \AqP^-_N[1-v, 1-v-t]
+(u+v)\log N
=\log\left[\p_t^{d}(\xv_0(t))\right],\label{eq:NsimplifyA}
\end{align}
where $\p_t^{d}(\xv_0(t))$ is given in Definition \ref{def_marginal}.

Since $u, v$ and $t$ are fixed, it follows from \eqref{eq:errorbounds} that the four error terms in \eqref{eq:Nuvpihf} can be bounded in absolute value by $CN^{-\chi}$ for some $C>0$ and any $\chi\in (0,1/2)$ (recall that $\kappa = 2N^{-1/2}$).
Combining this observation with \eqref{eq:NsimplifyA} proves \eqref{eq:lem83ab} for $j=0$.


In the same manner, we can bound the asymptotic behavior  of the other $j\in \llbracket 1,-v-t/2\rrbracket$ masses. In order to do this we must take into account the additional multiplicative terms, all of which are of the form $1-q^a$, $1+q^a$, or $q^a$ for some choices of real $a$. We use the following error bound:
For $a\in \R$, if we write
\begin{equation}\label{eq:a123defa}
\frac{1-q^a}{1-q}= a \,\cdot\, e^{\Err^1_N(a)},\qquad\frac{1+q^a}{1+q}= e^{\Err^2_N(a)},\qquad q^a= e^{\Err^2_N(a)}
\end{equation}
then there exist constants $C_a>0$ such that
\begin{equation}\label{eq:a123defb}
\left| \Err^1_N(a)\right|,\left| \Err^2_N(a)\right|,\left| \Err^3_N(a)\right| \leq C_a N^{-1/2}.
\end{equation}
Applying these bounds to the formula for $\hat\pi^{(N), d}_t(\xvN_j(t))$ in \eqref{eq:computingthemasses} and rewriting the additional factors multiplying $\hat\pi^{(N), d}_t(\xvN_j(t))$ in the form of $1-q^a$, $1+q^a$ and $q^a$ for various choices of $a$, we see that
\begin{align}
\dfrac{\left(q^{2 v+t},
    -q^{v+1+t}, q^{v+u}, -q^{v+1}\right)_j\left(1- q^{2 v+2 j+t}\right)}
{\left(q,- q^{v}, q^{1-u+v+t}, -q^{v+t}\right)_j \left(1-q^{2v+t}\right)} \left(q^{-v-u-1}\right)^j \\
= \dfrac{(v+j+t/2)\cdot\left[2 v+t,v+u\right ]_j}{(v+t/2)\cdot j! \left[1-u+v+t\right ]_j} e^{\Err_N^{M_j}}
\end{align}
where
$
\left|\Err_N^{M_j}\right| \leq CN^{-1/2}
$
for some constant $C>0$.
Recognizing that
\begin{equation}
\dfrac{\Gamma(u-v-t,2+v+u)}{\Gamma(-2 v-t)}\cdot \dfrac{(v+j+t/2)\cdot\left[2 v+t,v+u\right ]_j}{(v+t/2)\cdot j! \left[1-u+v+t\right ]_j}=\p_t^{d}(\xv_j(t)),
\end{equation}
we arrive at \eqref{eq:lem83ab}, thus completing the proof of Lemma \ref{marginal_atoms}.

\subsubsection{Proof of Lemma \ref{lem:p}}
There are three parts to this lemma.

\medskip
\noindent{\bf Part 1.}
Observe that \eqref{eq:hatpNmrlim} immediately follows from \eqref{eq:hatpNmrbound}. It remains to prove the bound  \eqref{eq:hatpNmrbound}. The proof of this result very closely follows that of Lemma \ref{lem:p_zero}.
From \eqref{eq:hatppcc} it follows that for $m,r\in \Supph^{(N),c}=[0,4N]$,
\begin{equation}\label{eq:recallpihatdef}
\hat \pi^{(N), c,c}_{s,t}(m, r) =\dfrac{1}{2N}  \pi^{(N), c, c}_{q^{t}, q^{s} }\left (1-\tfrac{r}{2N}, 1-\tfrac{m}{2N}\right)\cdot \dfrac{\hat\pi^{(N), c}_{t}(r)}{\hat\pi^{(N), c}_{s}(m)}.
\end{equation}
Lemma \ref{lem:p_zero} controls $\hat\pi^{(N), c}_{t}(r)$ and
$\hat\pi^{(N), c}_{s}(m)$. Thus, we need only to control $$ \pi^{(N), c,
  c}_{q^{t}, q^{s}}\left (1-\tfrac{r}{2N},
  1-\tfrac{m}{2N}\right).$$
It is useful to factorize this in order to utilize certain symmetries. Define
\begin{align}\label{eq:piNccpzero1}
&f_{t;u,\tilde u}^{(N), c}( r):
=\bigg\lvert\left(-q^{\tilde u+t/2+\iu \sqrt N \theta_{r}/2},q^{u+t/2+\iu \sqrt N \theta_{r}/2}\right)_{\infty}\bigg \rvert^2,\\
&g_{s;u,\tilde u}^{(N), c}( m):
= \frac{
\bigg\lvert\left(q^{\iu \sqrt N \theta_{m}};q\right)_{\infty}\bigg \rvert^2}
{\bigg \lvert \left(-q^{\tilde u+s/2+\iu \sqrt N  \theta_m/2},q^{u+s/2+\iu \sqrt N \theta_{m}/2}\right)_{\infty}\bigg \rvert^2},\\
&h_{s,t;u,\tilde u}^{(N), c}(m, r): =\frac{1}{4\pi N \sqrt{\frac{m}{N}}\sqrt{1-\frac{m}{4 N}}}\\& \times\frac{\left(q, -q^{u+\tilde
                                      u+s},q^{t-s}
                                      \right)_\infty}{\left(-q^{u+\tilde
                                      u+t} \right)_\infty}\frac{1}{\bigg \lvert \left(q^{(t-s)/2+\iu \sqrt N (\theta_{r}+\theta_{m})/2} , q^{(t-s)/2+\iu \sqrt N (-\theta_{r}+\theta_{m})/2}\right)_{\infty}\bigg \rvert^2},\\
\end{align}
where we used the notation from \eqref{eq:thetardef} (suppressing the $N$ dependence above) that
$$\theta_a=\theta^{(N)}_a:=\arccos\left(1-\frac{a}{2N}\right).$$
Defining
\begin{align}\label{eq:piNccpzero_mixed}
&\pi^{(N), c, c}_{q^{t}, q^{s}, u, \tilde u }\left(x, y\right)
:= \\ & 2N f_{t;u,\tilde u}^{(N), c}\left(2N(1-x)\right) \cdot  g_{s;u,\tilde u}^{(N), c}\left(2N(1-y)\right)\cdot
                                           h_{s,t;u,\tilde u}^{(N), c}\left(2N(1-x),2N(1-y)\right)
                                           \end{align}
it follows from \eqref{eq:AW} that
\begin{align}\label{eq:piNccpzero_mixed}
&\pi^{(N), c, c}_{q^{t}, q^{s} }\left(1-\tfrac{r}{2N}, 1-\tfrac{m}{2N}\right)
=\pi^{(N), c, c}_{q^{t}, q^{s}; u, 1 }\left(1-\tfrac{r}{2N}, 1-\tfrac{m}{2N}\right).
\end{align}

Observe that if we send $m\mapsto 4N-r$, then $\theta_r\mapsto \pi -\theta_r$ and thus also $q^{\iu  \sqrt N \frac{\theta_r}{2}}\mapsto -q^{\iu  \sqrt N \frac{\theta_r}{2}}$ and $q^{\iu  \sqrt N \theta_r}\mapsto q^{\iu  \sqrt N \theta_r}$. These transformations imply that
\begin{equation}\label{eq:piinvar_mixed}
f_{t;u,\tilde u}^{(N), c}(4N-r)=f_{t;\tilde u, u}^{(N), c}(r),\quad
g_{s;u,\tilde u}^{(N), c}(4N-m)= g_{s;\tilde u, u}^{(N), c}(m).
\end{equation}

As in the proof of Lemma \ref{lem:p_zero}, by using the transformation \eqref{eq:piinvar_mixed} it will suffice to consider the asymptotics behavior of these $f$ and $g$ functions just for $m\in [0,2N]$ and $r\in [0,2N]$. 

%
%

Using the notation from
\eqref{eq:thetardef} and the notation from Section
\ref{sec:notforasym}, we may write

\begin{align}
&\log \left[\sqrt{1-\frac{m}{4N}}\,\cdot\,
                \frac{1}{2N}\,\cdot\,\pi^{(N), c, c}_{q^{t}, q^{s};
                u,\tilde u}\left(1-\tfrac{r}{2N}, 1-\tfrac{m}{2N}\right)\right]
=\\
&-\log\left[4\pi N\sqrt{\frac{r}{N}}\right]+\AqP_{N;u,\tilde{u}}(\sqrt m,\sqrt r)+\Err_{N; u,\tilde u}^{\AqP}(m ,r)+\Err_{N; u,\tilde u}^{\mathcal E}(m, r).\label{eq:rewritepiNcc}
\end{align}
The function $\AqP_{N, u,\tilde u}(a, b)$ is defined for $a,b\in \R_{\geq 0}$ by
\begin{equation}\label{eq:AqPlong_p}
\AqP_{N, u,\tilde u} (a, b):= \AqP_{N; u,\tilde u} ^f(a,
                                    b)+ \AqP_{N; u,\tilde u}^g (a,
                                    b)+ \AqP_{N; u,\tilde u}^h (a,
                                    b)
\end{equation}
where
\begin{align}
&  \AqP_{N; u,\tilde u} ^f(a, b):=\AqP_N^+\big[ u+\tfrac{t}{2}+\iu \tfrac{b}{2}, u+\tfrac{t}{2}-\iu \tfrac{b}{2}\big]+\AqP_N^-\big[ \tilde u+\tfrac{t }{2}+\iu  \tfrac{b}{2},\tilde u+\tfrac{t}{2}-\iu \tfrac{b}{2}\big],\\
&  \AqP_{N; u,\tilde u} ^g(a, b):=\AqP_N^+\big[ \iu a, -\iu a\big]-\AqP_N^+\big[ u+\tfrac{s}{2}+\iu \tfrac{b}{2}, u+\tfrac{t}{2}-\iu \tfrac{b}{2}\big]-\AqP_N^-\big[ \tilde u+\tfrac{s }{2}+\iu  \tfrac{b}{2},\tilde u+\tfrac{s}{2}-\iu \tfrac{b}{2}\big],\\
&  \AqP_{N; u,\tilde u} ^h(a, b):=\AqP_N^+\big[ 1, t-s\big]+\AqP_N^-[u+\tilde u+s]-\AqP_N^-[u+\tilde u+t]-\\
&  \qquad\qquad  \qquad  \AqP_N^+\big[\tfrac{t-s}{2}+\iu \tfrac{a+b}{2}, \tfrac{t-s}{2}-\iu \tfrac{a+b}{2}, \tfrac{t-s}{2}+\iu \tfrac{a-b}{2}, \tfrac{t-s}{2}-\iu \tfrac{a-b}{2}\big].\\
\end{align}
The error term $\Err_{N; u,\tilde u}^{\mathcal A}(m ,r)$ is defined as
 $$
\Err_{N; u,\tilde u}^{\mathcal A}(m ,r):=\Err_{N; u,\tilde u}^{\mathcal A, f}(m ,r) +\Err_{N; u,\tilde u}^{\mathcal A, g}(m ,r)
+\Err_{N; u,\tilde u}^{\mathcal A, h}(m ,r),
$$
where, for $\bullet \in \{f,g,h\}$ we have
$$
\Err_{N; u,\tilde u}^{\mathcal A, \bullet}(m ,r):=\mathcal A^{\bullet}_{N; u,\tilde u}\left(\sqrt m+\Err^\theta(m), \sqrt r+\Err^\theta(r)\right)-\mathcal A^{\bullet}_{N; u,\tilde u} (\sqrt m,\sqrt r)
.$$
Similarly, we define the error term 
$$
\Err^{\mathcal{E}}_{N, u, \tilde u}(m, r): = \mathcal{E}_{N;u,\tilde u}\left(\sqrt{m} +
  \Err^{\theta}_N(m), \sqrt{r} + \Err^{\theta}_N(r)\right),$$
where $\mathcal{E}_{N;u,\tilde u}(a, b)$ is defined exactly as in \eqref{eq:AqPlong_p}, except with the $\AqP$ symbol replaced by $\Err_N$, as in \eqref{eq:errorN}.

Next we use \eqref{eq:piinvar_mixed} and the arguments from Lemma \ref{lem:p_zero} to get an
analogue of \eqref{eq:ErrAEbounds}, i.e. for all $\eta>1$ there exists $N_0\in \Z_{\geq 1}$ and $C,\chi>0$ such that for all $N>N_0$ and $m,r\in \Supph^{(N),c}=[0,4N]$, and $\bullet\in \{f,g\}$
\begin{equation}
\left|\Err^{\AqP, \bullet}_{N; u, \tilde
    u}(m,r)\right|,\left|\Err^{\mathcal{E}}_{ N; u, \tilde u}(m,r)\right| \leq C N^{-\chi}(1+\sqrt{m})^{\eta}+C N^{-\chi}(1+\sqrt{r})^{\eta}.
\end{equation}
We can also prove such a bound for $\Err^{\AqP, h}_{N; u, \tilde  u}(m,r)$ directly by using Proposition \ref{factorials} (since $t-s>0$ we use the part of the proposition which assumes the condition that $\textup{dist}(\textup{Re}(z),\mathbb{Z}_{\leq 0})>r$ for some $r>0$).

Using the explicit formulas  for $\AqP^{\pm}$ from \eqref{eq:AqPp} and
\eqref{eq:AqPm}, when $\tilde u=1$ we may simplify the above calculation to see that
\begin{align}\label{eq:piNAterm}
&-\log\left[4\pi N\sqrt{\frac{r}{N}}\right]+\AqP_{N, u,1}(\sqrt m,\sqrt r) \\
&\qquad=\log\left[\frac{\left\vert \Gamma\left(u+\frac{s}{2}+\iu
  \frac{\sqrt m}{2},\frac{t-s}{2}+\iu\frac{\sqrt m+\sqrt r}{2}, \frac{t-s}{2}+\iu\frac{\sqrt m-\sqrt r}{2}\right)\right\vert^2}{8\pi\cdot \sqrt r\cdot \Gamma(t-s)\left\vert \Gamma\left( u+\frac{t}{2}+\iu
  \frac{\sqrt r}{2}, \iu  \sqrt m\right) \right\vert^2}
\right].
\end{align}
There is a $4\pi$ factor on the left-hand side above versus $8\pi$ on the right-hand side. This extra factor of two comes from an imbalance between the second order expansion in the $\AqP^+_N$ and $\AqP^{-}_N$.
From Lemma \ref{lem:p_zero}, we may write
\begin{equation}\label{eq:bound_err}
\log\left[\dfrac{\sqrt{1-\frac{r}{4N}}\hat\pi^{(N), c}_{t}(r)}{\sqrt{1-\frac{m}{4N}}\hat\pi^{(N), c}_{s}(m)}\right]= \log\left[\dfrac{\p^c_{t}(r)}{\p^c_{s}(m)} \right] + \ErrLem^{(N),c}_{t}(r)-\ErrLem^{(N),c}_{s}(m)
\end{equation}
where the error terms $\ErrLem^{(N),c}_{t}(r)$ and $\ErrLem^{(N),c}_{s}(m)$ are controlled by the bounds in \eqref{eq:lemma82ineq}.
In light of \eqref{eq:recallpihatdef}, we may combine this with \eqref{eq:piNAterm} and \eqref{eq:rewritepiNcc} to conclude that (recall Definition \ref{def_transition})
\begin{align}
\log\Big(\sqrt{1-\frac{r}{4N}}&\hat \pi^{(N), c,c}_{s,t}(m, r)\Big)
  = \\
  &\log\left[\frac{\left\vert \Gamma\left(u+\frac{s}{2}+\iu
 \frac{\sqrt m}{2}\right) \Gamma\left( \frac{t-s}{2}+\iu\frac{\sqrt m+\sqrt r}{2}\right)\Gamma\left( \frac{t-s}{2}+\iu\frac{\sqrt m-\sqrt r}{2}\right)\right\vert^2}{8\pi\cdot \sqrt r\cdot \Gamma(t-s)\left\vert \Gamma\left(u+\frac{t}{2}+\iu
 \frac{\sqrt r}{2}\right) \Gamma\left( \iu  \sqrt m\right) \right\vert^2} \dfrac{\p^c_{t}(r)}{\p^c_{s}(m)}
\right]\\
&\quad+\Err_{N; u, 1}^{\AqP}(m ,r)+\Err_{N; u, 1}^{\mathcal E}(m, r) +  \ErrLem^{(N),c}_{t}(r)-\ErrLem^{(N),c}_{s}(m)\\
&=\quad\log\left[\p^{c,c}_{s, t}(m, r)\right] + \ErrLem^{(N),c,c}_{s,t}(m,r),
\end{align}
where $\ErrLem^{(N),c,c}_{s,t}(m,r):=\Err_{N; u, 1}^{\AqP}(m
,r)+\Err_{N; u, 1}^{\mathcal E}(m, r) +  \ErrLem^{(N),c}_{t}(r)-\ErrLem^{(N),c}_{s}(m)$.
The simplification which produces $\p^{c,c}_{s, t}(m, r)$ above can be verified by appealing to the explicit formula for $\p^{c,c}_{s, t}(m, r)$ from Definition \ref{def_transition}.

Combining \ref{eq:bound_err}  with  \eqref{eq:lemma82ineq} in Lemma \ref{lem:p_zero} (that $| \ErrLem^{(N),c}_{t}(r)|\leq  C N^{-\chi}(1+\sqrt{r})^{\eta}$ and likewise with $r$ replaced by $m$) we see that $\ErrLem^{(N),c,c}_{s,t}(m,r)$ is likewise bound in absolute value by $C N^{-\chi}(1+\sqrt{m})^{\eta}+C N^{-\chi}(1+\sqrt{r})^{\eta}$ which completes the proof of \eqref{eq:hatpNmrbound} and hence part 1 of this lemma.

\medskip
\noindent{\bf Part 2.}
Since \eqref{eq:ucaselimit}  follows from \eqref{eq:ucasebound}, it remains to prove  \eqref{eq:ucasebound}.

We may rewrite \eqref{eq:hatppdc} as
\begin{equation}\label{eq:hatpidchapy}
\hat \pi^{(N), d, c}_{s,t}(\xvN_j(s), r) = \pi^{(N), c, d}_{q^{t}, q^{s} }(1-\tfrac{r}{2N},\xvNunscaled_j(s))\cdot \dfrac{\hat\pi^{(N), c}_{t}(r)}{\hat\pi^{(N),d}_{s}(\xvN_j(s))},
\end{equation}
where $r\in \Supph^{(N),c}=[0,4N]$ and $\xvNunscaled_j(s)$, for $j\in\llbracket 0,\lfloor -v-s/2\rfloor\rrbracket$, constitutes all of the atoms in $\Supp^{(N),d}_s$. (Recall that  $\xvN_j(s)$ and $\xvNunscaled_j(s)$ are related by \eqref{eq:yhatscale}.)
For $r\in \Supph^{(N),c}$, we may use \eqref{eq:def_p} to rewrite, for Borel $V\subset\R$,
\begin{equation}\label{eq:piasAW}
\pi^{(N)}_{q^{t}, q^{s} }\left (1-\tfrac{r}{2N},V\right) = AW\left(V;q^{v+\frac{s}{2}},-q^{1+\frac{s}{2}},q^{\frac{t-s}{2} + \iu \sqrt{N} \frac{\theta_r}{2}},q^{\frac{t-s}{2} - \iu \sqrt{N} \frac{\theta_r}{2}}\right).
\end{equation}
Based on the discussion about atoms in Section \ref{sec:awd} we observe that as long as $1+\tfrac{s}{2} >0$ and $\tfrac{t-s}{2}>0$ (both of which necessarily hold since we have assumed $s<t$ and $s,t\in (-2,2)$), the only atoms are those coming from the $q^{v+\tfrac{s}{2}}$ term. This term has absolute value exceeding $1$ and hence we see that the atoms of this measure are precisely $\xvNunscaled_j(s)$ for $j\in\llbracket 0,\lfloor -v-s/2\rfloor\rrbracket$.

By \eqref{eq:AWdmasses}, the weight $\pi^{(N),c,d}_{q^{t}, q^{s}}\left(1-\tfrac{r}{2N}, \xvNunscaled_0(s)\right)$ at $\xvNunscaled_0(s)$ is
\begin{equation}\label{eq:piyzero}
\dfrac{\left(q^{-2 v-s}, -q^{1+t/2+\iu\sqrt N \frac{\theta_{r}}{2}} , -q^{1+t/2-\iu\sqrt N \frac{\theta_{r}}{2}},q^{t-s}\right)_\infty}{\left(-q^{1-v},q^{-v+t/2-s+\iu\sqrt N \frac{\theta_{r}}{2}}, q^{-v+t/2-s-\iu\sqrt N \frac{\theta_{r}}{2}},-q^{v+t+1}\right)_\infty},
\end{equation}
while, for  $j\in\llbracket 1,\lfloor -v-s/2\rfloor\rrbracket$ the weights at $\xvNunscaled_j(s)$ are
\begin{equation}\label{eq:piygenj}
\pi^{(N),c,d}_{q^{t}, q^{s}}\left(1-\tfrac{r}{2N}, \xvNunscaled_j(s)\right)=\pi^{(N),c,d}_{q^{t}, q^{s}}\left(1-\tfrac{r}{2N}, \xvNunscaled_0(s)\right) \times M_j
\end{equation}
where the additional multiplicative factor $M_j$ is defined as
\begin{equation}\label{eq:Mj}
\dfrac{\left (q^{2 v+s}, -q^{1+v+s} , q^{v+t/2+\iu\sqrt N \frac{\theta_{r}}{2}}, q^{v+t/2-\iu\sqrt N \frac{\theta_{r}}{2}}\right)_{j} \left(1-q^{2v+2j+s}\right)}{\left( q, -q^{v}, q^{1+v+s-t/2+\iu\sqrt N\frac{\theta_{r}}{2}}, q^{1+s+v-t/2-\iu\sqrt N \frac{\theta_{r}}{2}}\right)_{j} \left(1-q^{2v+s}\right)}\cdot\left(-q^{-v-t}\right)^{j}.
\end{equation}
In \eqref{eq:piyzero} and \eqref{eq:Mj} there are no instances of division by zero. This would arise in \eqref{eq:piyzero} if $-v+t/2-s\in \Z_{\leq 0}$ and in \eqref{eq:Mj} if $1+v+s-t/2+j-1=0$ for $j\in\llbracket 1,\lfloor -v-s/2\rfloor\rrbracket$, or if $2v+s=0$. However, since $v+s/2<0$ and that $s<t$, none of these occur.

Using notation from Section \ref{sec:notforasym}, we may rewrite \eqref{eq:piyzero} as
\begin{align}
\log \left[\pi^{(N),c,d}_{q^{t}, q^{s}}\left(1-\tfrac{r}{2N}, \xvNunscaled_0(s)\right)\right]
=\AqP_N(\sqrt r)+\Err_N^{\AqP}(r)+\Err_N^{\mathcal E}(r).\quad\label{eq:rewritepiNcc84}
\end{align}
The function $\AqP_N(a)$ is now (compared to the proof of Lemma \ref{lem:p_zero}) defined for $a\in \R_{\geq 0}$ as
\begin{align}
\AqP_N(a):=&
\AqP_N^+\left[-2v-s, t-s\right]+ \AqP_N^-\left[1+\dfrac{t}{2}+\iu\dfrac{a}{2}, 1+\dfrac{t}{2}-\iu \dfrac{a}{2}\right]\\
&-\AqP_N^+\left[-v+\dfrac{t}{2}-s+\iu \dfrac{a}{2}, -v+\dfrac{t}{2}-s-\iu \dfrac{a}{2} \right]-\AqP_N^-\left[1-v,v+t+1\right].\label{eq:AqPlong_p84}
\end{align}
The error term
$
\Err_N^{\mathcal A}(r):=\mathcal A_N\left(\sqrt r+\Err^\theta(r)\right)-\mathcal A_N(\sqrt r).
$
The error term
$
\Err^{\mathcal{E}}_N(r) := \mathcal{E}_{N}\left(\sqrt{r} + \Err^{\theta}_N(r)\right),
$
where $\mathcal{E}_N(a)$ is defined as in \eqref{eq:AqPlong_p84}, except with the $\AqP$ replaced by $\Err_N$.

Using the explicit formulas  for $\AqP^{\pm}$ from \eqref{eq:AqPp} and \eqref{eq:AqPm}, we may simplify
\begin{equation}\label{eq:Aqpexpansiongamma}
\AqP_N(\sqrt r) =\log\Bigg[\dfrac{\Big\lvert\Gamma\left(\frac{t}{2}-s-v+\frac{\iu \sqrt r}{2}\right)\Big\rvert^2}{\Gamma(-2 v-s,t-s)}\Bigg].
\end{equation}
From Lemma \ref{lem:p_zero} and Lemma \ref{marginal_atoms} we may write
\begin{equation}
\log\left[\dfrac{\sqrt{1-\frac{r}{4N}}\hat\pi^{(N), c}_{t}(r)}{\hat\pi^{(N),d}_{s}\left(\xvN_j(s)\right)}\right]= \log\left[\dfrac{\p^c_{t}(r)}{\p^d_{s}(\xv_j)} \right] + \ErrLem^{(N),c}_t(r) -\ErrLem^{(N),d}_s(\xv_0(s)),
\end{equation}
where $\ErrLem^{(N),c}_t(r)$ and $\ErrLem^{(N),d}_s(\xv_0(s))$ are controlled by \eqref{eq:lemma82ineq} and \eqref{eq:lem83ab}.
In light of \eqref{eq:hatpidchapy}, we may combine this with \eqref{eq:Aqpexpansiongamma} and \eqref{eq:rewritepiNcc84} to conclude (with Definition \ref{def_transition}) that
\begin{equation}\label{eq:hatpiy0sr}
\log\left[\sqrt{1-\frac{r}{4N}}\hat \pi^{(N), d, c}_{s,t}(\xvN_0(s), r)\right]
=\log\left[\p^{d,c}_{s, t}(\xv_0, r)\right] + \ErrLem^{(N),d,c}_{s,t}(\xv_0(s),r),
\end{equation}
where $\ErrLem^{(N),d,c}_{s,t}(\xv_0(s),r):= \Err_N^{\AqP}(r)+\Err_N^{\mathcal E}(r)+ \ErrLem^{(N),c}_t(r) -\ErrLem^{(N),d}_s(\xv_0(s))$.
Just as in the proof of Lemma \ref{lem:p_zero}, since Proposition
\ref{factorials} can be applied directly  we can show that for all $\eta>1$ there exists $N_0\in \Z_{\geq 1}$ and $C,\chi>0$ such that for all $N>N_0$ and $r\in \Supph^{(N),c}=[0,4N]$, $|\ErrLem^{(N),d,c}_{s,t}(\xv_0(s),r)|\leq C N^{-\chi}(1+\sqrt{r})^{\eta}$. This completes the proof of \eqref{eq:ucasebound} and hence part 2 of this lemma when $j=0$.

%

\smallskip
When $j\in \llbracket 1,\lfloor -v-s/2\rfloor \rrbracket$, we must consider $M_j$ in \eqref{eq:Mj}. We claim that
\begin{equation}\label{eq:claimedMjexpansion}
M_j =  \frac{\left[2v+s,v+\frac{t}{2}+\iu \frac{\sqrt{r}}{2},v+\frac{t}{2}-\iu \frac{\sqrt{r}}{2}\right]_{j} (2v+2j+s)}{\left[1,1+v+s-\tfrac{t}{2} + \iu\frac{\sqrt{r}}{2},1+v+s-\tfrac{t}{2} - \iu\frac{\sqrt{r}}{2}\right]_j (2v+s)}(-1)^j \,\cdot\,e^{\Err_N^{M_j}(r)}
\end{equation}
where  $\Err_N^{M_j}(r)$ satisfies the following bound: For all $\eta>1$ there exists $N_0\in \Z_{\geq 1}$ and $C,\chi>0$ such that for all $N>N_0$ and $r\in \Supph^{(N),c}=[0,4N]$,
\begin{equation}\label{eq:claimedMjerror}
\left|\Err_N^{M,j}(r)\right|\leq C N^{-\chi}(1+\sqrt{r})^{\eta}.
\end{equation}
If we combine the above claim with our already established result for $j=0$, we may use Definition \ref{def_transition} to match our formula with that of $\p^{d,c}_{s, t}(x^u_j, r)$ so as to conclude that
\begin{equation}
\log\left[\sqrt{1-\frac{r}{4N}}\hat \pi^{(N), d, c}_{s,t}(\xvN_j(s), r)\right]
=\log\left[\p^{d,c}_{s, t}(\xv_j, r)\right] + \ErrLem^{(N),d,c}_{s,t}(\xv_j(s),r)
\end{equation}
where
$ \ErrLem^{(N),d,c}_{s,t}(\xv_j(s),r) :=  \ErrLem^{(N),d,c}_{s,t}(\xv_0(s),r) + \Err_N^{M_j}(r)$ is bounded in absolute value by  $C N^{-\chi}(1+\sqrt{r})^{\eta}$. This shows \eqref{eq:ucasebound} and hence part 2 of the lemma for $j\in \llbracket 1,\lfloor -v-s/2\rfloor \rrbracket$.

\smallskip
It remains to demonstrate \eqref{eq:claimedMjexpansion} with the error bound \eqref{eq:claimedMjerror}.
All of the terms in $M_j$ which involve $1-q^a$, $1+q^a$ or $q^a$ for real $a$ can be controlled via \eqref{eq:a123defa} and \eqref{eq:a123defb}, just as in the proof of Lemma \ref{marginal_atoms}.
%
The only terms in $M_j$ which are not controlled by these bounds are those involving $\iu \sqrt{N} \theta_r/2$. To deal with those terms we make use of a more general version of the first growth bound above: there exists a constant $C>0$ such that for all $z\in \C$
$
\frac{1-q^z}{1-q}= z \,\cdot\, e^{\Err^4_N(z)}
$
where the error bound satisfies
$
\left| \Err^4_N(z)\right|\leq C N^{-1/2} (1+|z|).
$
By combining this bound with the control on $\theta_r$ demonstrated earlier in Step 2 of the proof of Lemma \ref{lem:p_zero} we claim the following bound: for  $a\in \R$,
\begin{equation}\label{eq:a123}
\frac{1-q^{a+\iu \sqrt{N} \theta_r/2}}{1-q} = (a + \iu \sqrt{r}/2) \, \cdot\, e^{\Err^4_N(a,r)}
\end{equation}
where $\Err^4_N(a,r)$ satisfies the bound that for any fixed $a$ and for all $\eta>1$ there exists $N_0\in \Z_{\geq 1}$ and $C,\chi>0$ such that for all $N>N_0$ and $r\in \Supph^{(N),c}=[0,4N]$,
$
\left|\Err^4_N(a,r)\right|\leq C N^{-\chi}(1+\sqrt{r})^{\eta}.
$
Combining this bound with \eqref{eq:a123}, we can deduce that \eqref{eq:claimedMjerror} holds and hence complete the proof of this part of the lemma.
%
%
%
%
%

\medskip
\noindent{\bf Part 3.}
We may explicitly write $\hat \pi^{(N), d, c}_{s,t}(\xuN_j(s), r)$ exactly as in \eqref{eq:hatpidchapy} except with $u$ replacing $v$ there. This formula involves the measure $\pi^{(N)}_{q^{t}, q^{s} }\left (1-\tfrac{r}{2N},V\right)$ defined on Borel $V\subset \mathbb{R}$, and \eqref{eq:piasAW} provides a formula for this measure in terms of the Askey-Wilson measure. We are concerned presently with the atomic part of this measure. However, inspection of the $a,b,c,d$ parameters of that measure reveals that as long as $v+\tfrac{s}{2}>0$, $1+\tfrac{s}{2}>0$ and $\tfrac{t-s}{2}>0$, this measure has no atomic part. Our assumption in this part is that $u-\tfrac{s}{2}<0$. If $v+\tfrac{s}{2}\leq 0$ as well, this would imply that $u+v<0$ which violates our assumption that $u+v>0$. Thus, there is no atomic part and so $\hat \pi^{(N), d, c}_{s,t}(\xuN_j(s), r) = 0$. This completes the proof of this part and hence the entire lemma.

\subsubsection{Proof of Lemma \ref{transition_atoms}}
For parts 1 and 2 of this lemma, let us rewrite \eqref{eq:hatppdd} so that for $x\in \Supph^{(N),d}_s$ and $y\in \Supph^{(N),d}_t$
\begin{equation}\label{eq:hatpidchapyd}
\hat \pi^{(N), d, d}_{s,t}(x, y) = \pi^{(N), d, d}_{q^{t}, q^{s} }\left (1-\tfrac{y}{2N},1-\tfrac{x}{2N}\right)\cdot \dfrac{\hat\pi^{(N), d}_{t}(y)}{\hat\pi^{(N),d}_{s}\left(x\right)}.
\end{equation}

\medskip
\noindent{\bf Part 1.}
Since we have assumed that $v+s/2<0$, $\Supph^{(N),d}_s$ equals the set of $\xvN_j(s)$ such that $j\in\llbracket 0,\lfloor -v-s/2\rfloor\rrbracket$. There are three possibilities for $\Supph^{(N),d}_t$: (1) If $v+t/2<0$ then $\Supph^{(N),d}_t$  equals the set of $\xvN_k(t)$ such that $k\in\llbracket 0,\lfloor -v-t/2\rfloor\rrbracket$; (2) If $v+t/2>0$ and $u-t/2>0$ then $\Supph^{(N),d}_t$  is empty; (3) If $v+t/2>0$ and $u-t/2<0$ then $\Supph^{(N),d}_t$ contains $u$-atoms $\xuN_k(t)$. However, this third possibility is excluded since the condition $v+s/2$ implies that $u>0$ and by our assumption that $t<\Cuv$ it follows that $u-t/2>0$ (recall when $u>0$, $\Cuv = \min(2u,2)$ so $t<\Cuv$ means $t<2u$). The case of empty support $\Supph^{(N),d}_t$  requires no further argument, so from here on out we assume that we are in the first case where $v+t/2<0$.

In Lemma \ref{marginal_atoms} we have already controlled the convergence of the discrete marginal distribution masses. So, recalling that  $\xvN_k(t)$ and $\xvNunscaled_k(t)$ are related by \eqref{eq:yhatscale}, our problem now reduces to studying the behavior of $ \pi^{(N), d, d}_{q^{t}, q^{s} }$. We may use \eqref{eq:def_p} to rewrite, for any Borel subset $V\subset\R$,
\begin{equation}\label{eq:piasAWdd}
\pi^{(N)}_{q^{t}, q^{s} }\left (\xvNunscaled_k(t),V\right) = AW\left(V;q^{v+\frac{s}{2}},-q^{1+\frac{s}{2}},q^{\frac{t-s}{2} + v+k+\frac{t}{2}},q^{\frac{t-s}{2} - (v+k+\frac{t}{2})}\right).
\end{equation}

We will analyze the probability masses in \eqref{eq:piasAWdd} when $V = \{\xvNunscaled_j(s)\}$ for $j\in\llbracket 0,\lfloor -v-s/2\rfloor\rrbracket$. The support of the measure $\pi^{(N)}_{q^{t}, q^{s} }\left (\xvNunscaled_k(t),V\right)$ may actually include more atoms that just this set, namely coming from the fact that $q^{\frac{t-s}{2} + v+k+\frac{t}{2}}$ may have absolute value exceeding $1$. We do not, however, need to consider these atoms since in \eqref{eq:hatpidchapyd} we are restricting ourselves to having the first variable $x\in \Supph^{(N),d}_s$, which does not include these additional atoms.

We use \eqref{eq:AWdmasses} to write out these masses. The weight at $\xvNunscaled_0(s)$ is
\begin{equation}\label{eq:piyzerodd}
\pi^{(N),d,d}_{q^{t}, q^{s}}\left(\xvNunscaled_k(t), \xvNunscaled_0(s)\right)=
\frac{\left(q^{-2 v-s}, -q^{v+k+1+t} , q^{t-s}, -q^{1-v-k}\right)_\infty}{\left(-q^{1-v}, q^{k+t-s}, q^{-2v-k-s},-q^{v+t+1}\right)_\infty}
\end{equation}
while, for  $j\in\llbracket 1,\lfloor -v-s/2\rfloor\rrbracket$ the weights at $\xvNunscaled_j(s)$ are
\begin{equation}\label{eq:piygenjdd}
\pi^{(N),d,d}_{q^{t}, q^{s}}\left(\xvNunscaled_k(t), \xvNunscaled_j(s)\right)=\pi^{(N),d,d}_{q^{t}, q^{s}}\left(\xvNunscaled_k(t), \xvNunscaled_0(s)\right) \times M_j
\end{equation}
where the additional multiplicative factor $M_j$ is defined as
\begin{equation}\label{eq:Mjdd}
M_j:=\dfrac{\left (q^{2 v+s}, -q^{1+v+s} , q^{2v+k+t},
    q^{-k}\right)_{j} \left(1-q^{2v+2j+s}\right)}{\left( q, -q^{v}, q^{1-k-t+s}, q^{1+2v+k+s }\right)_{j} \left(1-q^{2v+s}\right)}\cdot\left(-q^{-v-t}\right)^{j}.
\end{equation}

The analysis of these formulas follows the same approach as the earlier lemmas, for example Part 2 of Lemma \ref{lem:p}. In fact, since all exponents are real, the analysis and control of error terms is even simpler. As such, we just record the limiting expressions which arise from applying Proposition  \ref{factorials}. 
Observe that
\begin{align}
&\lim_{N\to \infty} \pi^{(N),d,d}_{q^{t}, q^{s}}\left(\xvNunscaled_k(t), \xvNunscaled_0(s)\right) = \frac{\Gamma\left(k+t-s, -2v- k-s \right)}{\Gamma\left(-2 v-s,t-s\right)}\\
&\lim_{N\to \infty}M_j =\frac{\left[2v+s, 2v+k+t, -k\right]_{j}
    (2v+2j+s)}{\left[1,1-k-t+s, 1+2v+k+s\right]_{j}(2v+s)}(-1)^{j}.
\end{align}
Putting these expressions together with our knowledge of the asymptotic behavior of  $\hat\pi^{(N), d}_{t}(\xvN_k(t))$ and $\hat\pi^{(N), d}_{s}(\xvN_j(s))$, we readily confirm the expansion \eqref{eq:pinnexp} and error bound \eqref{eq:pinnexpbound}.

\medskip
\noindent{\bf Part 2.}
Since we have assumed that $u-s/2<0$, $\Supph^{(N),d}_s$ equals the set of $\xuN_j(s)$ such that $j\in\llbracket 0,\lfloor -u+s/2\rfloor\rrbracket$. Since $s<t$ it follows that $u-t/2<0$ as well thus $\Supph^{(N),d}_t$ equals the set of $\xuN_k(t)$ such that $k\in\llbracket 0,\lfloor -u+t/2\rfloor\rrbracket$. As we will see, the transition mass from  $\xuN_j(s)$  to  $\xuN_k(t)$ is zero when $j<k$.

In Lemma \ref{marginal_atoms} we have already controlled the convergence of the discrete marginal distribution masses. So, recalling that  $\xvN_k(t)$ and $\xvNunscaled_k(t)$ are related by \eqref{eq:yhatscale}, our problem now reduces to studying the behavior of $ \pi^{(N), d, d}_{q^{t}, q^{s}}$. We may use \eqref{eq:def_p} to rewrite, for any Borel subset $V\subset\R$,
\begin{equation}\label{eq:piasAWdd2}
\pi^{(N)}_{q^{t}, q^{s} }\left (\xuNunscaled_k(t),V\right) = AW\left(V;q^{v+\frac{s}{2}},-q^{1+\frac{s}{2}},q^{\frac{t-s}{2} + u+k-\frac{t}{2}},q^{\frac{t-s}{2} - (u+k-\frac{t}{2})}\right).
\end{equation}

Atoms in the Askey-Wilson measure on the right-hand side of \eqref{eq:piasAWdd2} arise when arguments exceed $1$ in absolute value. Since $u-s/2<0$, it follows (since $u+v>0$) that $v+s/2>0$ so $|q^{v+\frac{s}{2}}|<1$. Likewise, since we have assumed that $s>-2$, $|-q^{1+\frac{s}{2}}|<1$. The fourth argument necessarily satisfies $|q^{\frac{t-s}{2} - (u+k-\frac{t}{2})}|$ because for $k\in \llbracket 0,\lfloor -u+\frac{t}{2}\rfloor\rrbracket$, we have  that $t-s-u+\frac{s}{2}-k \geq \frac{t-s}{2} >0$. The absolute value of the third argument $|q^{\frac{t-s}{2} + u+k-\frac{t}{2}}|$ can exceed $1$ if $u+k-\frac{s}{2}<0$. In that case, this term will contribute atoms at $\xuNunscaled_{k+i}(s)$ for $i\in \llbracket 0, \lfloor -u+\frac{s}{2}-k\rfloor\rrbracket$, which is a subset of  $\Supph^{(N),d}_s$. For all other elements in $\Supph^{(N),d}_s$, there will be no atom in this transition probability.

It remains to compute the masses of the atoms in the Askey-Wilson process \eqref{eq:piasAWdd} at $\xuNunscaled_{k+i}(s)$ for $i\in \llbracket 0, \lfloor -u+\frac{s}{2}-k\rfloor\rrbracket$. For this we use \eqref{eq:AWdmasses}, noting that in this case we have to switch the $a$ and $c$ arguments in the formula for the masses. This yields
\begin{equation}
\pi^{(N), d, d}_{q^t, q^s}\left(\xuNunscaled_k(t), \xuNunscaled_{k}(s)\right)=\frac{\left(q^{-2 u-2k+s}, -q^{v+1+s} , q^{v-u-k+t},    -q^{1-u-k+t}\right)_\infty}{\left(q^{v-u-k+s}, -q^{1-u-k+s}, q^{-2u-2k+t },-q^{v+t+1}\right)_\infty},
\end{equation}
while for $i\in \llbracket 1, \lfloor -u+\frac{s}{2}-k\rfloor\rrbracket$
\begin{equation}
\pi^{(N), d, d}_{q^t, q^s}\left(\xuNunscaled_k(t), \xuNunscaled_{k+i}(s)\right)
=\pi^{(N), d, d}_{q^t, q^s}\left(\xuNunscaled_k(t), \xuNunscaled_{k}(s)\right)\times M_i
\end{equation}
where the additional multiplicative factor $M_i$ is defined as
\begin{equation}
M_i =\frac{\left (q^{2 u+2k-s}, q^{v+u+k} , -q^{u+1+k}, q^{t-s}\right)_{i} \left(1-q^{2u+2k+2i-s}\right)}{\left( q, q^{u-v+1+k-s}, -q^{u+k-s}, q^{2u+2k-t+1 }\right)_{i} \left(1-q^{2u+2k-s}\right)}\cdot\left(-q^{-v-2t+s}\right)^{i}.
\end{equation}
The analysis of these formulas follows the same approach as the earlier lemmas, for example Part 2 of Lemma \ref{lem:p}. In fact, since all exponents are real, the analysis and control of error terms is even simpler. As such, we just record the limiting expressions which arise from applying Proposition  \ref{factorials}.
Observe that
\begin{equation}
\lim_{N\to \infty} \pi^{(N),d,d}_{q^{t}, q^{s}}\left(\xuNunscaled_k(t), \xuNunscaled_k(s)\right) = \frac{\Gamma\left(v-u-k+s\right) \Gamma\left( -2u-2 k+t \right)}{\Gamma\left(-2 u-2 k+s\right) \Gamma\left(v-u-k+t\right)}
\end{equation}
and that
\begin{equation}
\lim_{N\to \infty}M_i =\frac{\left[2u+2k-s, v+u+k,t-s\right]_{i}
    (2u+2k+2i-s)}{\left[1,u-v+1+k-s, 2u+2k-t+1\right]_{j}(2u+2k-s)}(-1)^{i}.
\end{equation}
Putting these expressions together with our knowledge of the asymptotic behavior of  $\hat\pi^{(N), d}_{t}(\xuN_k(t))$ and $\hat\pi^{(N), d}_{s}(\xuN_{k+i}(s))$, we readily confirm the expansion \eqref{eq:pinnexp2}
and error bound \eqref{eq:pinnexpbound2}.

\medskip
\noindent{\bf Part 3.}
First consider the case when $u-t/2<0$. In this case, the discrete support of the measure $\hat\pi^{(N), c, d}_{s,t}(m, \cdot)$ is given by  $\Supph^{(N),d}_t$ which equals the set of $\xuN_k(t)$ such that $k\in\llbracket 0,\lfloor -u+t/2\rfloor\rrbracket$. We may rewrite \eqref{eq:hatppcd} as
\begin{equation}\label{eq:hatpidchapycdformula}
\hat \pi^{(N), c, d}_{s,t}\left(m, \xuN_k(t)\right) =\frac{1}{2N} \,\cdot\, \pi^{(N), d, c}_{q^{t}, q^{s} }\left (\xuNunscaled_k(t),1-\tfrac{m}{2N}\right)\cdot \dfrac{\hat\pi^{(N), d}_{t}\left( \xuN_k(t)\right)}{\hat\pi^{(N),c}_{s}(m)}.
\end{equation}
We have already studied the behavior of the marginal distribution terms on the right-hand side above, thus we focus now on $\pi^{(N), d, c}_{q^{t}, q^{s} }\left (\xuNunscaled_k(t),1-\tfrac{m}{2N}\right)$.
In \eqref{eq:piasAWdd2} we rewrote this in terms of the Askey-Wilson process. The density of the absolutely continuous part of that measure is given in \eqref{eq:AW} and using that expression and mimicking the  asymptotic analysis in the proof of Lemma \ref{lem:p}, we show that
\begin{align}
&\frac{ \sqrt{1-\frac{m}{4 N}}}{2N}\,\cdot\,\pi^{(N), d, c}_{q^{t}, q^{s} }\left (\xuNunscaled_k(t),1-\tfrac{m}{2N}\right)=\\
&\dfrac{\Big\lvert \Gamma\left(u+\frac{s}{2}+\iu \frac{\sqrt  m}{2}, u+j-\frac{s}{2}+\iu\frac{\sqrt m}{2}, -u-j+t-\frac{s}{2}+\iu\frac{\sqrt m}{2}\right)\Big\rvert^2}{8 \pi\cdot \sqrt m\cdot  \Gamma( v+u+j, v-u-j+t,t-s)\cdot\bigg \rvert\Gamma\left( \iu
       \sqrt  m \right) \bigg \rvert^2}\cdot  e^{\ErrLemtilde^{(N),d, c}_{s, t}(m,\xu_k(t))}
\end{align}
and that for all
$\eta>1$, there exists $N_0\in \Z_{\geq 1}$ and  $C,\chi\in \R_{>0}$ such that for all $N>N_0$ and all $m\in \Supph^{(N),c}= [0,4N]$
\begin{equation}\label{eq:ucaseboundsss}
\Big|\ErrLemtilde^{(N),d, c}_{s, t}(m,\xu_k(t))\Big|\leq CN^{-\chi}(1+\sqrt{m})^{\eta}.
\end{equation}

In light of \eqref{eq:hatpidchapycdformula}, combining the above error bound with our  bounds on \\$\hat\pi^{(N), d}_{t}\left( \xuN_k(t)\right)$ and $\hat\pi^{(N),c}_{s}(m)$ from Lemma \ref{marginal_atoms} and \ref{lem:p_zero}, we arrive at the claimed result from the lemma. Note that the factor of $\sqrt{1-\frac{m}{4 N}}$ cancels with a corresponding factor coming from our application of Lemma \ref{lem:p_zero} and that the matching to $\p^{c, d}_{s, t}\left(m, \xu_k(t)\right)$ can be seen from Definition \ref{def_transition}.

Now let us turn to the case when  $v+t/2<0$. In this case, the discrete support of the measure $\hat\pi^{(N), c, d}_{s,t}(m, \cdot)$ is given by  $\Supph^{(N),d}_t$ which equals the set of $\xvN_k(t)$ such that $k\in\llbracket 0,\lfloor -v-t/2\rfloor\rrbracket$. We may rewrite \eqref{eq:hatppcd} as in \eqref{eq:hatpidchapycdformula} with $u$ replaced by $v$. Focusing on $\pi^{(N), d, c}_{q^{t}, q^{s} }\left (\xvNunscaled_k(t),1-\tfrac{m}{2N}\right)$ we see that the explicit expression for this include $q^{-k}$ inside the $q$-Pochhammer symbols in the numerator. This implies that the numerator is zero. By inspect, the denominator is non zero. This implies that $\hat \pi^{(N), c, d}_{s,t}\left(m, \xuN_k(t)\right)$ is identically zero, completing the proof of part 3 and hence the lemma.

\section{Asymptotics of $(\pm q^z;q)_\infty$: Proof of Proposition \ref{factorials}}\label{sec:factorials}
Throughout  we will use the notation $s=\sigma + \iu t$ and let $\textup{Arg(s)}\in [-\pi,\pi]$ denote the argument of the complex number $s$.

\subsection{Preliminaries from analytic number theory}\label{sec_antprelims}


\subsubsection{Gamma function}
For $s\in \C$ with $\textup {Re}(s)>1$ define the gamma function as
  $$
  \Gamma(s) :=\int_0^\infty e^{-x}x^{s-1}dx.
  $$
This can be meromorphically continuated with
\begin{equation}\label{eq:gamma_pole}
\textup{simple poles at }\Z_{\leq 0}\textup{ and }\textup{Res}_{s=-k}\left[ \Gamma(s)\right]  = \frac{(-1)^k}{k!}\textup{ for }k\in \Z_{\leq 0}.
  \end{equation}
The Euler reflection formula shows that
  $
  \Gamma(1-s)\Gamma(s) = \frac{\pi}{\sin(\pi s)}.
  $
For $s = \sigma + i t$, \cite[(21.51)]{R} shows that:
  \begin{lemma}\label{lem:asympt_gamma}
  For any compact $K\subset \R$ and $\e>0$ there exists $t_0>0$ such that for all $\sigma\in K$ and $t$ with $|t|>t_0$,
  \begin{equation}\label{eq:asympt_gamma}
  \left\vert \frac{\big|\Gamma(s)\big|}{\sqrt{2\pi} |t|^{\sigma-1/2} e^{-\pi |t|/2}}-1\right\vert <\e.
  \end{equation}
  \end{lemma}

\subsubsection{Zeta and eta functions} (See \cite[Sections 1.3--1.4]{MOS}.) The Riemann zeta function $\zeta(s)$ is defined for $\textup{Re}(s)>1$ as
$$
\zeta(s) :=\sum_{n=1}^{\infty} \frac{1}{n^s}
$$
and is the $z=1$ specialization of the Hurwitz zeta function $\zeta(s, z)$, which is defined for $\textup{Re}(s)>1$ and $\textup{Re}(z)>0$ as
\begin{equation}\label{eq:hurwsum}
\zeta(s,z):= \sum_{n=0}^{\infty} \frac{1}{(n+z)^s}.
\end{equation}
Still assuming  $\textup{Re}(s)>1$ and $\textup{Re}(z)>0$, the Hurwitz zeta function admits an integral representation as
\begin{equation}\label{eq:hurwitz}
\zeta(s,z) = \frac{1}{\Gamma(s)} \int_0^\infty \frac{e^{-\rho z}\rho^{s-1}}{1-e^{-\rho}} d\rho.
\end{equation}

As functions of $s$, both $\zeta(s)$ and $\zeta(s,z)$ can be meromorphically extended to the complex plane and yield meromorphic functions having simple poles only at $s=1$ with residues
\begin{equation}\label{eq:zetaress}
 \textup{Res}_{s=1}\left[\zeta(s)\right] = \textup{Res}_{s=1}\left[\zeta(s,z)\right] = 1.
\end{equation}

 We will use of the following evaluation formulas for the zeta function
 \begin{equation}\label{eq:zeta_at_two}
\zeta(2)=\dfrac{\pi^2}{6}, \qquad \zeta(0, z)=\dfrac{1}{2}-z,
   \end{equation}
as well various derivatives and limits
\begin{align}\label{eq:derivative_zeta}
&\dfrac{d\zeta(s, z)}{ds}\bigg\lvert_{s=0}\!\!=\log\left[\dfrac{\Gamma(z)}{\sqrt{2\pi}}\right],
&&\zeta'(0)=-\dfrac{1}{2}\log(2 \pi), \\
&\lim_{s\rightarrow 1}{\zeta(s,z)-\dfrac{1}{s-1}}=-\psi(z),
&&\lim_{s\rightarrow 1}{\zeta(s)-\dfrac{1}{s-1}}=-\psi(1).
\end{align}
Above we have used the digamma function $\psi(z)=\dfrac{\Gamma'(z)}{\Gamma(z)}$ whose evaluation $\psi(1)=-\gamma$ is given by the Euler-Mascheroni constant $\gamma$.

Finally, the Dirichlet eta function $\eta(s)$ is defined for $\textup{Re}(s)>0$ as
$$
\eta(s) :=\sum_{n=1}^{\infty} \frac{(-1)^{n-1}}{n^s},
$$
and is related to the zeta function via
\begin{equation} \label{eq:eta_z}
  \eta(s) = (1-2^{1-s})\zeta(s).
\end{equation}
From this one sees that the eta function is an entire function (the $1-2^{1-s}$ factor cancels the first order pole of the zeta function at $s=1$).

\subsubsection{Bernoulli polynomials}\label{sec:bern}
The Bernoulli polynomials $\big\{B_n(x)\big\}_{n\in \Z_{\geq 0}}$ are be defined via the generating function expression
$$
\frac{te^{tx}}{e^t-1} = \sum_{n=0}^{\infty} B_n(x) \frac{t^n}{n!}.
$$
In particular, $B_n(x)$ is a degree $n$ polynomial in $x$ with the first few polynomials given by $B_0(x)=1$, $B_1(x)=x-\tfrac{1}{2}$, $B_2(x) = x^2-x+\tfrac{1}{6}$, and so on. We recall some results we will need from  \cite[Sections 12.11-12.12]{Apostol}.
The Bernoulli polynomials satisfy
\begin{equation}\label{eq:bsym}
  B_n(1-x) = (-1)^n B_n(x)\qquad \text{and}\qquad B_n(x+1)-B_n(x) = n x^{n-1},
\end{equation}
which implies that, through taking $x=0$, $B_n(1) = B_n(0)$ for $n\in \Z_{\geq 2}$ and $B_1(1) =\tfrac{1}{2} = - B_1(0)$.
The Bernoulli numbers are defined as
$
B_n := B_n(0).
$
Besides $B_1=-\tfrac{1}{2}$, all other odd indexed Bernoulli numbers are $0$.

In terms of the Bernoulli polynomials, we have that for $n\in
\Z_{\geq 0}$
\begin{equation}\label{eq:zetaressbern}
\zeta(-n,z) = -\frac{B_{n+1}(z)}{n+1},\qquad \text{ and }\qquad
\zeta(-n) = (-1)^n\frac{B_{n+1}}{n+1}.
\end{equation}

\subsubsection{Asymptotics of $\zeta(s)$ and $\zeta(s, z)$}

We will need the following asymptotic result for the Riemann zeta function which can be found in \cite[Section 43]{R}. For $\sigma\in [0,1]$, the bound proven below is suboptimal, though sufficient for our purposes. The Lindel\"of function $\mu(\sigma)$ determines the optimal growth exponent. The  Lindel\"of hypothesis posits that $\mu(1/2) = 0$, though this is far from proved. 
We use $\tilde\mu(\sigma)$ below to represent an upper bound on this exponent as $\sigma$ varies.

\begin{lemma}\label{prop_zeta}  For any compact $K\subset \R$, there exists a $C,t_0>0$ such that
$
\lvert \zeta(s) \rvert\leq C |t|^{\tilde\mu(\sigma)}
$
for all $s=\sigma + \iu t  $ with $\sigma\in K$ and $|t|>t_0$. Here $\tilde{\mu}(\sigma)$ is defined by
$$
\tilde\mu(\sigma) =
\begin{cases}
0&\textrm{for } \sigma\geq 1,\\
\frac{1-\sigma}{2}&\textrm{for } \sigma\in [0,1],\\
\frac{1}{2}-\sigma&\textrm{for } \sigma\leq 0.
\end{cases}
$$
\end{lemma}

Next, we prove a simple bound on the Hurwitz zeta function.

\begin{lemma}\label{lem:hurzpos}
For all $\sigma_0>1$ there exists a $C>0$ such that for all $s$ with $\sigma>\sigma_0$ and all  $z\in\C$ with $\textup{Re}(z)>0$, we have
\begin{equation}
|\zeta(s,z)| \leq C e^{|\textup{Arg}(z) \cdot t|}.
\end{equation}
\end{lemma}
\begin{proof}
From \eqref{eq:hurwsum} and the triangle inequality, $|\zeta(s,z)|\leq \sum_{n=0}^{\infty} |(n+z)^{-s}|$. Note that
$|(n+z)^{-s}| = |n+z|^{-\sigma} e^{\textup{Arg}(n+z)\cdot t}$. Since $\textup{Re}(z)>0$, $|n+z|>|n|$ and $|\textup{Arg}(n+z)|<|\textup{Arg}(z)|$. Thus we can further bound $|\zeta(s,z)| \leq \sum_{n=0}^{\infty} |n|^{-\sigma}e^{|\textup{Arg}(z)\cdot t|} \leq C e^{|\textup{Arg}(z)\cdot t|}$
where the constant $C$ can be taken as $1+\zeta(\sigma_0)$.
\end{proof}

Controlling $|\zeta(s,z)|$ when $\textup{Re}(s)\leq 1$ is considerably hard. It will be important to demonstrate bounds in that case which contain the $z$ dependence on the sub-leading polynomial terms. Such bounds are provided below as Proposition \ref{hurwitz_bound}. In the proof of the proposition, we will make use of an integral formulas for $\zeta(s,z)$. There are many related formulas available in the literature (cf. \cite[Section 1.4]{MOS}). We could note find a precise statement of the formula in Lemma \ref{lem:hurwitzint}, thus we prove it here. 

\begin{lemma}\label{lem:hurwitzint}
For $s=\sigma + \iu t \notin \Z_{\leq 0}$, $|\textup{Arg}(z)|<\pi$ and $d\in (0,1)$ with $\sigma +d\notin \Z_{\leq 0}$,
\begin{align}
  \zeta(s, z)=&\dfrac{1}{2
  z^s}+\dfrac{z^{1-s}}{s-1}+\dfrac{z^{-s}}{2 \pi \iu \Gamma(s)}\int\limits^{d+\iu
  \infty}_{d-\iu\infty}\Gamma(-u,u+s)z^{-u}\zeta(-u)du\\
&  -\dfrac{z^{-s}}{\Gamma(s)}\sum\limits_{k=0}^{\lfloor
  -(\sigma+d) \rfloor }\Gamma(k+s )\zeta (k+s) \frac{(-1)^k z^{k+s}}{k!},\label{eq:hureq}
\end{align}
(Recall $\Gamma(a,b)=\Gamma(a)\Gamma(b)$) where the summation in $k$ is dropped if $\sigma +d >0$.
\end{lemma}
\begin{proof}
We will make use of the following formula \cite[(3.3.9)]{ParisKam}
\begin{equation}\label{eq:pariskim}
\dfrac{1}{2\pi \iu}\int\limits_{c-\iu \infty}^{c+\iu \infty}\Gamma(s,a-s)x^{-s}ds=\frac{\Gamma(a)}{(1+x)^a},
\end{equation}
which is valid so long as $\textup{Re}(a)>c>0$ and $|\textup{Arg}(x)|<\pi$.

Start by assuming $\sigma>1$, $|\textup{Arg}(z)|<\pi$ and $c\in(1,
\sigma)$. Then from the definition of $\zeta(s, z)$ (this formula
appears as in \cite[(2.1)]{Paris})
\begin{align}
&z^s \zeta(s, z) = 1+\sum\limits_{n=1}^\infty
\left(1+\frac{n}{z}\right)^{-s}=1+\dfrac{1}{\Gamma(s)}\sum\limits_{n=1}^\infty
\dfrac{1}{2\pi \iu}\int\limits_{c-\iu \infty}^{c+\iu
  \infty}\Gamma(y,s-y) \left(\frac{n}{z}\right)^{-y}dy\\
& = 1+ \int\limits_{c-\iu \infty}^{c+\iu
  \infty}\frac{\Gamma(y,s-y) z^{y}\zeta(y) dy}{2\pi \iu\Gamma(s)}=1+
\int\limits_{-c-\iu \infty}^{-c+\iu
  \infty}\frac{\Gamma(-u,s+u) z^{-u}\zeta(-u) du}{2\pi \iu\Gamma(s)}\\
&=\dfrac{1}{2}+\dfrac{z}{s-1}+\int\limits_{d-\iu \infty}^{d+\iu
  \infty}\frac{\Gamma(-u,s+u) z^{u}\zeta(-u) du}{2\pi \iu\Gamma(s)},
\end{align}
 for $d\in(0, 1)$.
The first equality is by the definition of the Hurwitz zeta function; the second is by \eqref{eq:pariskim}; the third is from interchanging the summation and integration (justified by Fubini) and appealing to the definition of the zeta function (we assume $\sigma>c>1$ here);  the fourth equality is the simple change of variables $u=-y$; and the final equality follows from shifting the contour of integration to the right from $-c+\iu \R$ to $d+\iu \R$ for $d\in (0,1)$. In this shifting, we encounter two poles, one at $u=-1$ (from $\zeta(-u)$) and one at $u=0$ (from $\Gamma(-u)$). The first two terms in the final line come from evaluating these residues, see \eqref{eq:gamma_pole} and \eqref{eq:zetaress}.  To justify shifting the contours we must show that the integrand decays sufficiently fast for $|\textup{Im}(u)|$ large. Using the bounds from Lemmas \ref{lem:asympt_gamma} and \ref{prop_zeta} we can prove uniformly in the strip between $-c+\iu\R$ and $d+\iu \R$, exponential decay like $e^{(|\textup{Arg}(z)|-\pi)|\textup{Im}(u)|}$ (recall that $|\textup{Arg}(z)|<\pi$) as $|\textup{Im}(u)|\to \infty$.

We have shown that for $|\textup{Arg}(z)|< \pi$, $\sigma>1$ and any $d\in (0,1)$:
\begin{equation}\label{eq:firsthz}
\zeta(s, z)=\dfrac{1}{2
  z^s}+\dfrac{z^{1-s}}{s-1}+\dfrac{z^{-s}}{2 \pi \iu \Gamma(s)}\int\limits^{d+\iu
  \infty}_{d-\iu\infty}\Gamma(-u,u+s)z^{-u}\zeta(-u)du.
\end{equation}

By \eqref{eq:derivative_zeta} it follows that $\zeta(s,z)-\dfrac{z^{1-s}}{s-1}$ is an entire function. For a fixed value of $d\in (0,1)$, the integral in \eqref{eq:firsthz} is analytic in $s$ provided $\sigma+ d>0$. By analytic continuation, the formula \eqref{eq:firsthz} actually holds for all $\sigma>- d$.

In order to extend to a formula for $\sigma\leq - d$ we will need to make some contour deformations and account for some residues.

Our aim is now to establish a formula for $s=\sigma+\iu t$ when $\sigma <-d$. For the moment, let us assume that $t\neq 0$ and let us fix some $d\in (0,1)$ and assume that $\sigma<-d$ and that $\sigma +d \notin \Z_{\leq 0}$. Fix some $\tilde\sigma<-d$ and let $\varepsilon_0=|t|/2$ (which is non-zero by our temporary assumption) and $\varepsilon_1=(\lceil -\tilde \sigma\rceil +\tilde \sigma)/2$. By Cauchy's theorem without changing the value of the integral we can deform to the contour $C$:
\begin{equation}
  \begin{split}
    C=&[d-\iu \infty, d-\iu
(t+\varepsilon_0)]\cup [d-\iu (t+\varepsilon_0), \varepsilon_1-\tilde\sigma-\iu
(t+\varepsilon_0)]\cup  \\
&[\varepsilon_1-\tilde\sigma-\iu (t+\varepsilon_0), \varepsilon_1-\tilde\sigma-\iu
(t-\varepsilon_0)]\cup \\
&[ \varepsilon_1-\tilde\sigma-\iu (t-\varepsilon_0),  d-\iu (t-\varepsilon_0)]\cup [d-\iu (t-\varepsilon_0), d+\iu \infty].
\end{split}
\end{equation}
The purpose of this deformation is that the integral is now analytic in $s$ provided that $\sigma \geq \tilde\sigma$. Let us now assume that $\sigma=\tilde\sigma$. To reach a final formula we will deform $C$ back to the original contour $d+\iu \R$. In doing so, we cross poles from the $\Gamma(u+s)$ term. These occur when $u+s =-k$ for $k \in \{0,1,\ldots, \lfloor -(\sigma+d)\rfloor\}$. Taking into account the residues and the direction of the contours yields \eqref{eq:hureq}
when $t\neq 0$. Provided that $s \notin\Z_{\leq 0}$, we can use continuity of both side of \eqref{eq:hureq} in $t$ to extend to $t=0$.
\end{proof}


We come to our main bound on the Hurwitz zeta function.
\begin{proposition}\label{hurwitz_bound}
For any non-integer $\sigma<0$ and $d\in (0,1/2)$ chosen such that
$d+\sigma\notin\Z_{\leq 0}$, there exists $C>0$ such that for all
$z\in\C$ with $|\textup{Arg}(z)|<\pi$  and all $s\in \C$ with $s=\sigma+\iu t$,
\begin{align}\label{eq:asymp_hurwitz_z}
  \qquad  |\zeta(s,z)|\leq C \Big( e^{\lvert \textup{Arg}(z) \cdot t \rvert }\left( |z|^{-\sigma}+(1+|t|)^{-1}|z|^{1-\sigma}+ |t|^{1/2-\sigma}|z|^{-d-\sigma}\right)\qquad \\
   + \max(1,|t|^{1/2-\sigma}) \sum\limits_{k=0}^{\lfloor
  -(\sigma+d) \rfloor } |z|^k\Big),
  \end{align}
where the summation in $k$ is dropped if $\sigma +d >0$.

For any $t_0>0$, $a<b$ and $z\in\C$ with $|\textup{Arg}(z)|<\pi$, there exists a constant $C>0$ and $c<\pi/2$ such that for all $s\in\C$ with $\sigma\in (a,b)$ and $|t|\geq t_0$,
\begin{equation}\label{eq:asymp_hurwitz_z2}
|\zeta(s,z)| \leq C e^{c|t|}.
\end{equation}
\end{proposition}

\begin{proof}
We focus on proving \eqref{eq:asymp_hurwitz_z}. The proof of the bound \eqref{eq:asymp_hurwitz_z2} is simpler and proceeds in much the same manner (and thus is not provided here).

In this proof when we write $x\lesssim y$ we mean that $x\leq C y$ for
some constant which may depend on $d$ and $\sigma$, but nothing else. In turn, when we say that ``$x$ is bounded by $y$'', we mean that $x\lesssim y$ and when we say that ``$x$ is bounded'', we mean that $x$ is bounded by a constant. We also will make use of Lemmas \ref{lem:asympt_gamma} and \ref{prop_zeta} to deduce bounds when the imaginary part of the argument of the gamma or zeta function is small or large. For the rest of this proof, let $t_0$ be such that the bound in Lemma \ref{lem:asympt_gamma} holds for $\e = 1/2$, and such that the bound in Lemma \ref{prop_zeta} holds for some constant $C$ as specified in the lemma. We will use $t_0$ as the cutoff between small and large.

This proof relies on the integral representation for $\zeta(s,z)$ given in Lemma \ref{lem:hurwitzint}. The hardest term to bound in that representation is the contour integral. Let us address the other terms first.
For the first two terms in the $\zeta(s,z)$ representation in Lemma \ref{lem:hurwitzint}, we find that
$$\left|\frac{1}{2z^s}\right|\lesssim e^{\textup{Arg}(z) \cdot t} |z|^{-\sigma}\quad\text{and}\quad
\left|\frac{z^{1-s}}{s-1}\right| \lesssim e^{-\textup{Arg}(z) \cdot t} (1+|t|)^{-1} |z|^{1-\sigma},
$$
where, in both terms we have used the fact that
$|z^s| = |z|^\sigma e^{-\textup{Arg}(z)\cdot t}$,
and in the second inequality we use that for $s$ with negative real part, $|1/(1-s)|<C(1+|t|)^{-1}$ for some constant $C$. Since $e^{\textup{Arg}(z)\cdot  t}$ and $e^{-\textup{Arg}(z)\cdot  t}$ are both bounded by $e^{|\textup{Arg}(z)\cdot  t|}$, we find that the contribution of these two terms is upper bounded by the first two terms in the right-hand side of \eqref{eq:asymp_hurwitz_z}.

The $\zeta(s,z)$ representation in Lemma \ref{lem:hurwitzint} also involves terms indexed by $k\in \{0, \ldots, \lfloor -(\sigma+d) \rfloor\}$. Taking absolute values these terms contribute a constant times
$
|\Gamma(s)|^{-1}\,\cdot\, |\Gamma(k+s )|\,\cdot\, |\zeta (k+s)| |z|^{k}.
$
For large $t$, appealing to the asymptotics of Lemmas
\ref{lem:asympt_gamma} and \ref{prop_zeta}, we can show that the
expression above is bounded by $|t|^{1/2-\sigma} |z|^k$ where as for
$t$ small, since we have assumed that $s$ (and hence also $s+k$) is
not in $\Z_{\geq 0}$, the expression is bounded by $|z|^k$. These bounds produce the final terms in \eqref{eq:asymp_hurwitz_z}.

All that remains is to control the contour integral term in \eqref{eq:hureq}.
Taking the absolute value inside of the integral, we are left to control
\begin{equation}\label{eq:toboundd2}
\left|\dfrac{z^{-s}}{2 \pi \iu \Gamma(s)}\right|\int\limits^{\infty}_{-\infty}|\Gamma(-d-\iu r,d+\sigma+\iu(r+t))|\cdot|z^{-d-\iu r}|\cdot|\zeta(-d-\iu r)dr|.
\end{equation}
The rest of this proof is devoted to showing that \eqref{eq:toboundd2} is bounded by the right-hand side of \eqref{eq:asymp_hurwitz_z}. This is elementary, though requires the analysis of a number of cases and the use of the bounds from Lemmas \ref{lem:asympt_gamma} and \ref{prop_zeta} for large imaginary parts of the gamma and zeta functions, as well as constant bounds on the gamma and zeta functions for small imaginary parts (for the gamma function, this is where  $d+\sigma\notin \Z_{\leq 0}$ is important).

\smallskip
We split our analysis of \eqref{eq:toboundd2} into two cases --- $|t|\leq t_0$ and $|t|>t_0$.

\smallskip
\noindent {\bf Case 1: $|t|\leq t_0$.} Using Lemma \ref{lem:asympt_gamma} and the analyticity of $1/\Gamma(s)$ we bound the pre-factor
\begin{equation}\label{eq:firsttermbound}
\left|\dfrac{z^{-s}}{2 \pi \iu \Gamma(s)}\right| \lesssim |z|^{-\sigma} e^{\textup{Arg}(z)\cdot  t} \lesssim |z|^{-\sigma} e^{|\textup{Arg}(z)\cdot  t|}
\end{equation}
Using this and $|z^{-d-\iu r}| = |z|^{-d} e^{\textup{Arg}(z)\cdot  r}$ inside the integrand of \eqref{eq:toboundd2} yields
\begin{align}
&\eqref{eq:toboundd2} \lesssim (I)\times (II),\quad\textrm{where}\quad (I):=|z|^{-\sigma-d}\,\cdot\, e^{\textup{Arg}(z)\cdot  t}\quad\textrm{and}\\
&(II):=\int\limits^{\infty}_{-\infty}|\Gamma(-d-\iu r,d+\sigma+\iu(r+t))| \,\cdot\, e^{\textup{Arg}(z)\cdot  r}\,\cdot\,|\zeta(-d-\iu r)\,dr|.
\end{align}
We claim that $(II)\lesssim 1$. Assume this claim for the moment. Since $d\in (0,1/2)$, $|z|^{-\sigma-d} \leq (|z|^{-\sigma}+|z|^{1-\sigma})$. Thus $(I)\times (II)\lesssim  e^{|\textup{Arg}(z)|}(|z|^{-\sigma}+|z|^{1-\sigma})$ which is, itself, bounded by the right-hand side of \eqref{eq:asymp_hurwitz_z} as desired.

To bound $(II)\lesssim 1$, we split the integral into $|r|\leq t_0$
and  $|r|>t_0$. In the first case, since the integrands can be bounded
by constants, the total contribution is likewise bounded by a
constant. In the second case, to estimate the integral over $|r|>t_0$,
we may use of Lemmas \ref{lem:asympt_gamma} and \ref{prop_zeta} for
the gamma and zeta functions. The integrand in $(II)$ is thus bounded up to a constant factor
by
$|r|^{\sigma+d-\frac{1}{2}}e^{-\frac{\pi}{2}(|r|+|r+t|)+\textup{Arg}(z)\cdot
  r}$. Since $\sigma<0$ and $d\in(0,1/2)$, $\sigma+d-\frac{1}{2}<0$
and hence $|r|^{\sigma+d-\frac{1}{2}}\lesssim 1$ for $|r|>t_0$. Since
$|\textup{Arg}(z)|<\pi$ and $|t|\leq t_0$, for $r$ large enough, $-\frac{\pi}{2}(|r|+|r+t|)+\textup{Arg}(z)\cdot  r<-\delta r$ for some $\delta>0$. Thus the integral in $(II)$ can be bounded by a constant as claimed.

\smallskip
\noindent {\bf Case 2: $|t|>t_0$.} We proceed in a similar, albeit more involved, manner as in Case 1. In place of the bound \eqref{eq:firsttermbound} we get (using Lemma \ref{lem:asympt_gamma} to control the behavior of $|1/\Gamma(s)|$) that
\begin{equation}
\left|\dfrac{z^{-s}}{2 \pi \iu \Gamma(s)}\right| \lesssim |z|^{-\sigma} e^{\textup{Arg}(z)\cdot  t} |t|^{-\sigma+1/2} e^{\pi |t|/2}.
\end{equation}
As opposed to in Case 1, we do not want to throw away the possible decay that $e^{\textup{Arg}(z)\cdot  t}$ can provide. Instead, we write this as $e^{\textup{Arg}(z)\cdot  t} = e^{|\textup{Arg}(z)\cdot  t|} \,\cdot \,e^{\textup{Arg}(z)\cdot  t - |\textup{Arg}(z)\cdot  t|}$. The first term goes with $(I)$ below, while the second term goes with $(II)$.
Using $|z^{-d-\iu r}| = |z|^{-d} e^{\textup{Arg}(z)\cdot  r}$ we see that
	$$
\eqref{eq:toboundd2}
\lesssim (I)\times (II)\quad\textrm{where}\quad (I):=|z|^{-\sigma-d}\,\cdot\, e^{|\textup{Arg}(z) \cdot t|} \,\cdot\,|t|^{-\sigma+\frac{1}{2}}\quad\textrm{and}$$
$$
(II):=e^{\textup{Arg}(z)\cdot t -|\textup{Arg}(z)\cdot t|+\frac{\pi}{2} |t|}\int\limits^{\infty}_{-\infty} F_{II}(r;d,\sigma,t,z)dr.
$$
with
$$
F_{II}(r;d,\sigma,t,z) = |\Gamma(-d-\iu r,d+\sigma+\iu(r+t))| e^{\textup{Arg}(z)\cdot  r}|\zeta(-d-\iu r)|
$$
We claim that $(II)\lesssim 1$. Assuming this we see that $(I)\cdot
(II)$ is bounded by the right hand side in \eqref{eq:asymp_hurwitz_z}. Thus, it remains to show that $(II)\lesssim 1$.

\begin{figure}[h]
\centering
\scalebox{0.3}{\includegraphics{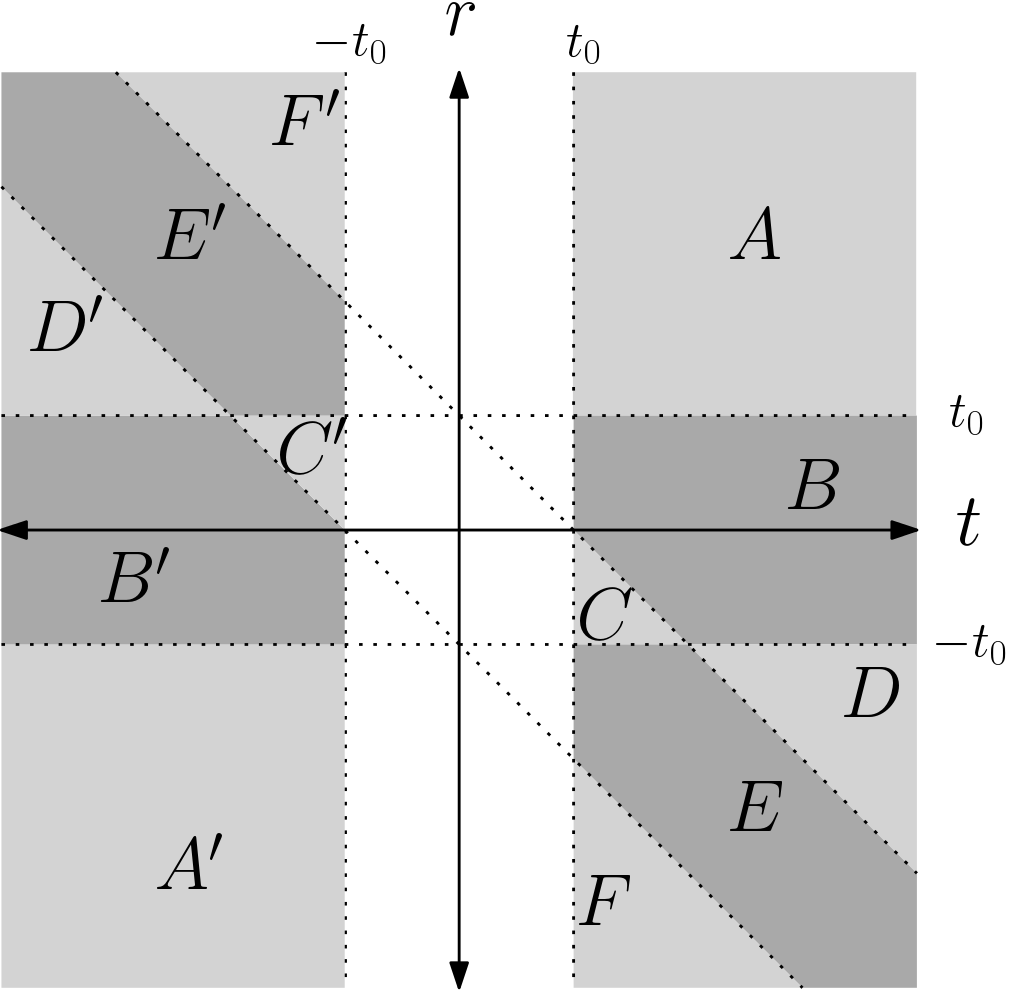}}
\end{figure}

In order to bound $(II)$, we split the integral depending on the size of $r$ and $r+t$. Assume that $t>t_0$ (the case $t<-t_0$ is completely analogous and involve the primed regions in the figure; we will not repeat the argument in that case though). We define six regions in the $(t,r)$ plane:
$A=\{(t,r):t>t_0,r>t_0\}$,
$B=\{(t,r):t>t_0,|r|\leq t_0,r+t>t_0\}$,
$C=\{(t,r):t>t_0,|r|\leq t_0,|r+t|\leq t_0\}$,
$D=\{(t,r):t>t_0,r<-t_0,r+t>t_0\}$,
$E=\{(t,r):t>t_0,r<-t_0,|r+t|\leq t_0\}$,
$F=\{(t,r):t>t_0,r<-t_0,t+r<-t_0\}$.
For $t$ given, we write $(II)_A$ to denote the expression given above for $(II)$ subject to the additional restriction that $(r,r+t)\in A$ (and likewise for $B,C,D,E,F$).

%
%
%
%
For a fixed  $t$ in each region $A,B,C,D,E,F$ we may upper bound the integrand $F_{II}(r;d,\sigma,t,z)$ defining $(II)$ either by constants if the imaginary part of the argument of
the gamma or zeta function is small, or by the asymptotics given in
Lemmas \ref{lem:asympt_gamma} and \ref{prop_zeta} if the imaginary
part of the argument of the gamma or zeta function are large. We can
then estimate the contribution of each region to the
integral. Depending on whether $t\in (t_0,2t_0]$ or $t>2 t_0$, the integral in $r$ will encounter a different set of regions. We consider these two cases.

\smallskip
\noindent{\bf Case 2.a: $t\in (t_0,2t_0]$.}
Here $(II)= (II)_A+(II)_B+(II)_C+(II)_E+(II)_F$. On
regions $B,C$ and $E$ the integrand $F_{II}(r;d,\sigma,t,z)\lesssim 1$ and
since the domain of integration for $r$ in these regions is bounded, it follows that $(II)_B,(II)_C,(II)_E\lesssim 1$. Regions $A$ and $F$ involve unbounded integrals and require
close inspection. On region $A$, $F_{II}(r;d,\sigma,t,z)\lesssim |r+t|^{d+\sigma-\frac{1}{2}}
e^{-\frac{\pi}{2}(|r|+|r+t|)+\textup{Arg}(z)\cdot r}$. We already
encountered such an integral in Case 1 when bounding $(II)$ there.
 Since $\sigma<0$ and $d\in(0,1/2)$, $\sigma+d-\frac{1}{2}<0$
and hence $|r+t|^{\sigma+d-\frac{1}{2}}\lesssim 1$. Since
$|\textup{Arg}(z)|<\pi$ and $t\in (t_0,2t_0]$, for $r$ large enough,
$-\frac{\pi}{2}(|r|+|r+t|)+\textup{Arg}(z)\cdot  r<-\delta r$ for some
$\delta>0$. Since $t\in (t_0,2t_0]$ the prefactor $e^{\textup{Arg}(z)\cdot t -|\textup{Arg}(z)\cdot t|+\frac{\pi}{2} |t|}$ in front of the
integral in $(II)$ is also bounded by a constant. Thus,  $(II)_A\lesssim 1$ as desired. The $F$ region follows similarly, as the integrand $F_{II}(r;d,\sigma,t,z)\lesssim|r+t|^{d+\sigma-\frac{1}{2}}
e^{-\frac{\pi}{2}(|r|+|r+t|)+\textup{Arg}(z)\cdot r}$. In both of these cases of region $A$ and $F$ we have used the fact that $t\in (t_0,2t_0]$ to bound the term $e^{\mp \frac{\pi}{2} t}$ in the prefactor by a constant. In the next case, this will not be true.

\noindent{\bf Case 2.b: $t>2t_0$.}
Now $(II)=(II)_A+(II)_B+(II)_D+(II)_E+(II)_F$. As in Case 2.a, for $(r,r+t)\in A$, $F_{II}(r;d,\sigma,t,z)\lesssim |r+t|^{d+\sigma-\frac{1}{2}} e^{-\frac{\pi}{2}(|r|+|r+t|)+\textup{Arg}(z)\cdot r}$. The term $|r+t|^{d+\sigma-\frac{1}{2}}\lesssim 1$, and the $e^{-\frac{\pi}{2}(|r|+|r+t|)+\textup{Arg}(z)\cdot r}\lesssim e^{-\frac{\pi}{2}t-\delta \cdot r}$ for some $\delta>0$. This implies that the integral of $F_{II}(r;d,\sigma,t,z)$ over $r$ such that $(r,r+t)\in A$ is bounded by a constant times $e^{-\pi t/2}$. This cancels the $e^{\pi |t|/2}$ pre-factor outside the integral in $(II)$. What is left is bounded by a constant times $e^{\textup{Arg}(z)\cdot t -|\textup{Arg}(z)\cdot t|}$ and since $\textup{Arg}(z)\cdot t -|\textup{Arg}(z)\cdot t|\leq 0$, we conclude that $(II)_A\lesssim 1$.

In region $B$, the integrand $F_{II}(r;d,\sigma,t,z)\lesssim |r+t|^{d+\sigma-\frac{1}{2}} e^{-\frac{\pi}{2}|r+t| + \textup{Arg}(z)\cdot  r}$. Since $r+t>t_0$ in this region, and since the $r$-variable is integrated from $-t_0$ to $t_0$, the contribution of the integral  of $F_{II}(r;d,\sigma,t,z)$ for $r$ in this region is bounded by a constant time $e^{-\pi t/2}$. Again, this cancels the  $e^{\pi |t|/2}$ pre-factor outside the integral in $(II)$ and thus $(II)_B\lesssim 1$.
Bounding the integral in regions $D, E$ and $F$ is more subtle.

Let us start by addressing $(II)_D$. Here,  $$F_{II}(r;d,\sigma,t,z)\lesssim |r+t|^{d+\sigma-\frac{1}{2}} e^{-\frac{\pi}{2}(|r|+|r+t|)+\textup{Arg}(z)\cdot r}.$$ Since $r\in (-t+t_0,-t_0)$, $|r|=-r$ and $|r+t|=r+t$. Also, since $d+\sigma-\frac{1}{2}<0$, $|r+t|^{d+\sigma-\frac{1}{2}}$ is bounded by a constant. Thus, the upper bound on the integrand $F_{II}(r;d,\sigma,t,z)$ reduces to $F_{II}(r;d,\sigma,t,z)\lesssim e^{-\frac{\pi}{2} t} e^{\textup{Arg}(z)\cdot  r}$. The magnitude of the integral $\int_{-t+1}^{-1} e^{\textup{Arg}(z)\cdot  r} dr$ depends on the sign of $\textup{Arg}(z)$.
If $\textup{Arg}(z)>0$, then the exponential $e^{\textup{Arg}(z)\cdot  r}$ decays and the integral is bounded by a constant. In this case, the pre-factor to the integral in $(II)$ is $e^{\frac{\pi}{2} t}$, which cancels the just demonstrated $e^{-\frac{\pi}{2} t}$ behavior of the integral. Thus, when $\textup{Arg}(z)>0$, $(II)_D\lesssim 1$.
If $\textup{Arg}(z)<0$, then the exponential $e^{\textup{Arg}(z)\cdot  r}$ grows and the integral is hence bounded by $e^{-\textup{Arg}(z)\cdot t}$. Combining this with the pre-factor in $(II)$ shows that when $\textup{Arg}(z)<0$, $(II)_D\lesssim e^{-|\textup{Arg}(z)\cdot t|}\lesssim 1$ since $t>t_0$.

Controlling the integral in region $E$ works similarly. Here $F_{II}(r;d,\sigma,t,z)\lesssim e^{-\frac{\pi}{2} |r|+\textup{Arg}(z)\cdot r}$. Since here $r\in (-t-t_0,-t+t_0)$, we can bound the integral of $F_{II}(r;d,\sigma,t,z)$ on this region  by a constant times $e^{-\frac{\pi}{2}t} e^{-\textup{Arg}(z)\cdot t}$. Putting this together with the pre-factors in $(II)$ shows that $(II)_E\lesssim e^{-|\textup{Arg}(z)\cdot t|}\lesssim 1$ since $t>2t_0$.

In region $F$, $F_{II}(r;d,\sigma,t,z)\lesssim|r+t|^{d+\sigma-\frac{1}{2}} e^{-\frac{\pi}{2}(|r|+|r+t|) + \textup{Arg}(z)\cdot r}$. Since $r<-t-t_0$, $|r|+|r+t| = -2r-t$ and since $d+\sigma-\frac{1}{2}<0$, $|r+t|^{d+\sigma-\frac{1}{2}}\lesssim 1$. Thus, $$F_{II}(r;d,\sigma,t,z)\lesssim e^{(\pi + \textup{Arg}(z))\cdot r + \frac{\pi}{2}t}.$$ We can bound the integral of this over the range $r\in (-\infty,-t-t_0)$ by a constant times $e^{-(\pi + \textup{Arg}(z))\cdot t + \frac{\pi}{2}t} = e^{- \textup{Arg}(z)\cdot t - \frac{\pi}{2}t}$. Combining this with the pre-factors in $(II)$ we find  that $(II)_F\lesssim e^{-|\textup{Arg}(z)\cdot t|}\lesssim 1$ since $t>2t_0$.
Thus, we have shown that $(II)\lesssim 1$ which implies the desired bound \eqref{eq:asymp_hurwitz_z} and hence completes the proof of the proposition.
\end{proof}

\subsubsection{Jacobi theta function}\label{sec:thetas}
Jacobi theta functions (see \cite[Chapter 10]{R}) are defined in the
following way for complex $\nu, \rho$ with
$\textup{Im}(\rho)>0$

\begin{align}
&\theta_1(\nu|\rho)=\dfrac{1}{\iu}\!\!\!\sum\limits_{k=-\infty}^\infty\!\!\!
(-1)^k e^{\pi \iu \rho (k+1/2)^2}e^{\pi \iu \nu (2k+1)},
&&\theta_2(\nu|\rho)=\!\!\!\sum\limits_{k=-\infty}^\infty\!\!\! e^{\pi \iu \rho (k+1/2)^2}e^{\pi \iu \nu (2k+1)},\\
&\theta_3(\nu|\rho)=\!\!\!\sum\limits_{k=-\infty}^\infty\!\!\! e^{\pi \iu \rho k^2}e^{2k\pi \iu \nu },
&&\theta_4(\nu|\rho)=\!\!\!\sum\limits_{k=-\infty}^\infty\!\!\!
    (-1)^k e^{\pi \iu \rho k^2}e^{2k \pi \iu \nu}.
\end{align}
For every value of $\rho$ in this half-plane the functions are entire
functions of $\nu$. Note that  $\theta_1(\nu|\rho)$ is an odd function in $\nu$ and
all others are even.
We also need the following identities\cite[(78.32), (78.33), (79.7),  (79.9) ]{R}:
\begin{equation*}
  \begin{split}
    \theta_1\left(\nu \lvert \rho\right)&= -\iu e^{\pi \iu
      \rho/4}\cdot e^{\pi \iu \nu}\cdot\left(e^{2\pi \iu \rho}, e^{2\pi \iu
        (\rho+\nu)}, e^{-2\pi\iu\nu}; e^{2\pi \iu \rho}\right)_\infty,\\
  \theta_2\left(\nu \lvert \rho\right)&= e^{\pi \iu
      \rho/4}\cdot e^{\pi \iu \nu}\cdot\left(e^{2\pi \iu \rho},- e^{2\pi \iu
        (\rho+\nu)},- e^{-2\pi\iu\nu}; e^{2\pi \iu \rho}\right)_\infty,\\
\theta_1\Big(\tfrac{\nu}{\rho}\Big \lvert
  -\tfrac{1}{\rho}\Big)&=-\iu\sqrt{\tfrac{\rho}{\iu}}\cdot e^{\iu \pi
  \nu^2/\rho}\cdot\theta_1(\nu|\rho),\\
\theta_4\Big(\tfrac{\nu}{\rho}\Big \vert
  -\tfrac{1}{\rho}\Big)&=\sqrt{\tfrac{\rho}{\iu}}\cdot e^{\iu \pi
  \nu^2/\rho}\cdot\theta_2(\nu|\rho),
\end{split}
\end{equation*}
 with the principal value of  the square root.
Using the $\theta_1$ identity,
 \begin{equation}\label{eq:theta}
\left(e^{-\kappa z}; e^{-\kappa}\right)_\infty= \frac{\sqrt {\frac{2 \pi}{\kappa}}
\exp\left(\frac{\kappa}{8}-\frac{\kappa z}{2}+\frac{\kappa
    z^2}{2}\right)}{\left(e^{-\kappa};
    e^{-\kappa}\right)_\infty \left(e^{-\kappa(1-z)}; e^{-\kappa}\right)_\infty}\cdot
\theta_1\Big(z \Big \vert \tfrac{2\pi \iu}{\kappa}\Big).
\end{equation}
Similarly, using the identities involve $\theta_4$ and $\theta_2$ we find that
%
%
 \begin{equation}\label{eq:theta_negative}
\left(-e^{-\kappa z}; e^{-\kappa}\right)_\infty=\frac{\sqrt {\frac{2 \pi}{\kappa}}
\exp\left(\frac{\kappa}{8}-\frac{\kappa z}{2}+\frac{\kappa
    z^2}{2}\right) }{\left(e^{-\kappa};
    e^{-\kappa}\right)_\infty \left(-e^{-\kappa (1-z)}; e^{-\kappa}\right)_\infty}\cdot
\theta_4\Big(z \Big \vert \tfrac{2\pi \iu}{\kappa}\Big).
\end{equation}


\subsubsection{Theta function bounds}
\begin{lemma} \label{theta_estimate_z}
For all $\alpha\in (0,\pi)$ there exist $C,c,\kappa_0>0$ such that for all $\kappa\in (0,\kappa_0)$ and all $\nu\in \C$ with $|\textup{Im}(\nu)| \leq \frac{\alpha}{\kappa}$
    \begin{equation}\label{eq:firstthetakappa}
\bigg\lvert \theta_1\left(\nu\bigg\lvert \dfrac{2\pi i}{\kappa}\right)\dfrac{e^{\pi^2/(2\kappa)}}{2 \sin(\pi \nu)} -1\bigg\rvert\,\,,\,\,
\bigg\lvert\theta_4\left(\nu\bigg\lvert
   \dfrac{2\pi i}{\kappa}\right)-1\bigg\lvert\leq C\cdot e^{-\frac{c}{\kappa}}.
\end{equation}
\end{lemma}
\begin{proof}
We can rewrite theta functions in the following way \cite[(76.2)]{R}:

\begin{align}
\theta_1(\nu|\rho)&=2\sum_{k=0}^{\infty}(-1)^k e^{(k+1/2)^2\pi \iu \rho}\sin\left((2k+1)\pi \nu\right),\\
\theta_4(\nu|\rho)&=1+2\sum_{k=1}^{\infty}(-1)^k e^{k^2\pi \iu \rho}\cos\left(2k \pi \nu\right).
\end{align}

Chebyshev polynomial of the first kind $T_k(x)$ and  the second kind  $U_k(x)$ with $x\in[-1,1]$ are defined in the
following way (see \cite[Section 1.8.2]{KS2}): For $x=\cos(\nu)$,
$
T_k(x):=\cos(k \nu)$ and $U_k(x):=\frac{\sin((k+1)\nu)}{\sin(\nu)}$.
Their coefficients are explicitly given by
\begin{equation*}
T_k(x) =\dfrac{k}{2}\sum\limits_{n=0}^{\lfloor k/2\rfloor}\dfrac{
  (-1)^k}{k-n}{k-n \choose n}(2 x)^{k-2n},\qquad U_k(x)=\sum\limits_{n=0}^{\lfloor k/2\rfloor} (-1)^n{k-n \choose n}(2 x)^{k-2n}.
  \end{equation*}

Inserting the Chebyshev polynomials into these expressions we arrive at
\begin{align}\label{eq:thetacheb}
  \dfrac{\theta_1(\nu|\rho)}{2 e^{\pi\iu \rho/4}\sin(\pi \nu)}-1&= 2\sum\limits_{k=1}^\infty(-1)^k  e^{\pi \iu \rho
  (k^2+k)}U_{2k}(\cos(\pi \nu)),\text{ and }\\
\theta_4(\nu|\rho)-1 &= 2\sum_{k=1}^{\infty}(-1)^k e^{\pi \iu
  \rho k^2}T_{2k}(\cos(\pi\nu)).
\end{align}

We now claim that with $x=\cos(\pi\nu)$, for all $\nu\in \C$
\begin{equation}\label{eq:ch_ineq}
|U_k(x)|,\, |T_k(x)|\,\leq \dfrac{(1+\sqrt 5)^{k+1}}{\sqrt 5} \cdot \max\Big(|\cos(\pi\nu)|,2^{-1}\Big)^k.
\end{equation}
To see this, consider the case $|x|\geq 1/2$ and $|x|<1/2$ separately. When $|x|\geq 1/2$
$$
|U_k(x)|,\,|T_k(x)|\,\leq |2 x|^{k}\sum_{n=0}^{\lfloor k/2\rfloor} {k-n
    \choose n}=|2 x|^{k}F(k+1),
$$
where $F(k+1)$ is the Fibonacci number. When $|x|< 1/2$
$$|U_k(x)|,\, |T_k(x)|\,\leq |2 x|^{k-2\lfloor k/2\rfloor}\sum_{n=0}^{\lfloor k/2\rfloor} {k-n
  \choose n}\leq F(k+1).$$
Putting these together we see that $|U_k(x)|,\, |T_k(x)|\,\leq F(k+1)\max(|2x|,1)^k$.
Since the Fibonacci number $F(k+1)$ equals the nearest integer to
$\tfrac{(1+\sqrt 5)^{k+1}}{2^{k+1} \sqrt 5}$, multiplying this by $2$ clearly yields an upper bound on $F(k+1)$ and combined with our earlier bounds on $|U_k(x)|$ and $|T_k(x)|$ we arrive at \eqref{eq:ch_ineq}.

Inserting \eqref{eq:ch_ineq} into \eqref{eq:thetacheb} we find that
\begin{align}
\bigg\lvert\dfrac{\theta_1(\nu|\rho)}{2 e^{\pi\iu \rho/4}\sin(\pi
  \nu)}-1\bigg\rvert &\leq \dfrac{2 (1+\sqrt 5)}{\sqrt 5} \sum\limits_{k=1}^\infty e^{\pi \iu \rho (k^2+k)}(1+\sqrt 5)^{2k}\cdot \max(|\cos(\pi \nu)|, 2^{-1})^{2k},  \label{eq:inequality_theta2}\\
\bigg\lvert \theta_4(\nu|\rho) -1\bigg\rvert&\leq \dfrac{2 (1+\sqrt 5)}{\sqrt 5} \sum\limits_{k=1}^\infty e^{\pi \iu \rho
  k^2}(1+\sqrt 5)^{2k}\cdot \max(|\cos(\pi \nu)|, 2^{-1})^{2k}.\label{eq:inequality_theta3}
\end{align}
Since $\rho = 2\pi \iu /\kappa$ it follows that $e^{\pi \iu \rho} = e^{-2 \pi^2/\kappa}<1$. Thus for $k\in \Z_{\geq 1}$ we can bound $e^{\pi \iu \rho (k^2+k)},  e^{\pi \iu \rho k^2} \leq  e^{\pi \iu \rho k}$. This shows that for some $C>0$,
\begin{align}\label{eq:combinedineq}
&\qquad\bigg\lvert\dfrac{\theta_1(\nu|\rho)}{2 e^{\pi\iu \rho/4}\sin(\pi \nu)}-1\bigg\rvert \,\, , \,\, \bigg\lvert \theta_4(\nu|\rho) -1\bigg\rvert\leq \\
& C  e^{-\frac{2\pi^2}{\kappa}} \cdot \max(|\cos(\pi \nu)|, 2^{-1})^{2}
\sum\limits_{k=0}^\infty e^{-\frac{2\pi^2}{\kappa}a k}(1+\sqrt 5)^{2k}\cdot \max(|\cos(\pi \nu)|, 2^{-1})^{2k}.
\end{align}
Observe that $|\cos(\pi \nu)| \leq e^{|\textup{Im}(\pi\nu)|}$. We have assumed that $|\textup{Im}(\nu)| \leq \frac{\alpha}{\kappa}$ with $\alpha\in (0,\pi)$ and thus it follows that there exists some $c>0$ such that
$$e^{-\frac{2\pi^2}{\kappa}} \cdot \max(|\cos(\pi \nu)|, 2^{-1})^{2}\leq e^{-\frac{c}{\kappa}}.$$ Plugging this bound into \eqref{eq:combinedineq} yields \eqref{eq:firstthetakappa} and hence the lemma.
\end{proof}

\subsubsection{Mellin transform} \label{sec:mellin}
For a function $f(x)$ on $(0, +\infty)$, and  $s\in\C$, define
\begin{equation} \label{eq:Mellin} F(s)=M[f;s]:=\int\limits_0^\infty f(x) x^{s-1} dx,
\end{equation}

The largest open strip  $-\infty\leq a<\textup{Re}(s)\leq
  b\leq\infty)$ in which the integral converges is called the {\em fundamental
strip} or the {\em strip of analyticity} of $M\left[f; s\right]$. Note
that if $g$ is defined by the relation $f(x)=g(-\log x),$
$$\int\limits_{0}^\infty x^{s-1} g(-\log x)
dx=\int\limits_{-\infty}^\infty e^{-t s}g(t)dt=M\left[f; s\right].$$
Therefore, all basic properties of the Mellin transform follow from those of the Laplace transform.
The following inversion formula can be found as \cite[Theorem 11.2.1.1]{BBO} (other similar statements abound).
\begin{proposition}\label{prop:melinver}
Assume that the function $F(s)$ is analytic in the strip $(a, c)$ and satisfies
  $\lvert F(s)\rvert\leq K\cdot |s|^{-2}$
for some $K>0$. Then, for $b\in (a,c)$,
    $$f(x)=\dfrac{1}{2\pi\iu}\int\limits_{b-\iu\infty}^{b+\iu\infty}
    F(s)x^{-s}ds$$
  is a continuous function of the variable $x\in(0;\infty)$ and does not depend on the choice of $b.$ Furthermore, $F(x) = M\left[f; s\right],$ and we then say that $f(x)$ is the inverse Mellin transform of $F(s)$.
\end{proposition}

\subsection{Proof of Proposition \ref{factorials}}

The following is the key to being able to use the above derived asymptotics on special function to access $q$-Pochhammer asymptotics.

\begin{lemma}\label{lem:mellin}
For all $c>1,$ $\kappa>0,$ $\textup{Re}(z)>0$ we have that
\begin{align}
  \log\left(e^{-\kappa z}; e^{-\kappa}\right)&=-\dfrac{1}{2\pi
    \iu}\int\limits_{c-\iu\infty}^{c+\iu\infty}\Gamma(s)\zeta(s+1)\zeta(s,
  z)\dfrac{ds}{\kappa^s},\label{eq:mellinqpoch}\\
  \log\left(-e^{-\kappa z}; e^{-\kappa}\right)&=-\dfrac{1}{2\pi
    \iu}\int\limits_{c-\iu\infty}^{c+\iu\infty}(2^{-s}-1)\Gamma(s) \zeta(s+1)\zeta(s,
  z)\dfrac{ds}{\kappa^s}\label{eq:mellinqpochneg}
  \end{align}
\end{lemma}

\begin{proof}
Our goal is to show that $\log\left(e^{-\kappa z};  e^{-\kappa}\right)$ and $\log\left(-e^{-\kappa z}; e^{-\kappa}\right)$ are the inverse Mellin transforms of the corresponding functions on the right-hand side of \eqref{eq:mellinqpoch}. To do this we compute the Mellin transforms of the left-hand sides in \eqref{eq:mellinqpoch} and then show they can be inverted.

For $a\in \C$ with $|a|<1$ and $q\in (0,1)$ we may write
  \begin{equation}\label{eq:log}
-\log(a; q)_\infty=-\sum\limits_{k=0}^\infty \log(1-a
q^k)=\sum\limits_{k=0}^\infty\sum\limits_{n=1}^\infty\dfrac{(a q^k)^n}{n}=\sum\limits_{n=1}^\infty\dfrac{a^n}{n}\sum\limits_{k=0}^\infty q^{n k}=\sum\limits_{n=1}^\infty\dfrac{a^n}{n(1-q^n)}.
    \end{equation}
The interchange of the order of summations is possible due to Fubini's theorem since for each $q\in (0,1)$ there exists $C=C(q)>0$ such that
$$\sum\limits_{n=1}^\infty\bigg\lvert \dfrac{a^n}{n(1-x^n)} \bigg \rvert \leq  C\sum\limits_{n=1}^\infty|a^n|<\infty.$$

Let $f_1(x)=-\log\left((e^{-\kappa z}; e^{-\kappa})_\infty\right)$. For $\textup {Re} (z)>0$ and $\textup{Re}(s)>1$,
\begin{align}
&M\left[f_1; s\right]=-\int\limits_0^\infty \log\left(( e^{-\kappa z};e^{-\kappa})_\infty\right)
\kappa^{s-1}d\kappa=\int\limits_0^\infty \sum\limits_{n=1}^\infty \dfrac{(
  e^{-\kappa z} )^n}{n (1-e^{-n \kappa})} \kappa^{s-1}d\kappa\\
&=\int\limits_0^\infty \sum\limits_{n=1}^\infty
\dfrac{1}{n^{s+1}}\dfrac{e^{-\rho z} \rho^{s-1} }{ (1-e^{- \rho})}
d\rho =\sum\limits_{n=1}^\infty\dfrac{1}{n^{s+1}}\int\limits_0^\infty \dfrac{e^{-\rho z} \rho^{s-1} }{ (1-e^{- \rho})}
d\rho= \Gamma(s)\zeta(s+1)\zeta(s ,z).\label{eq:Mf1}
\end{align}
The first equality in \eqref{eq:Mf1} is by the definition of the Mellin transform. The second equality in \eqref{eq:Mf1} uses \eqref{eq:log}  with $a=e^{-\kappa z}$ and $q=e^{-\kappa}$. The third equality in \eqref{eq:Mf1} comes from the change of variables $\kappa=\rho/n$. The fourth equality in \eqref{eq:Mf1} is valid because
$$\sum\limits_{n=1}^\infty
\dfrac{1}{n^{\textup{Re}(s)+1}} \int\limits_0^\infty \bigg\lvert\dfrac{
  e^{-\rho z} \rho^{s-1} }{ (1-e^{- \rho})}d\rho\bigg \rvert\leq
C\sum\limits_{n=1}^\infty\dfrac{1}{n^{\textup{Re}(s)+1}}<\infty.
$$
The constant term $C>0$ above comes from bounding the integral: Since  $\textup{Re}(s)>1$, the behavior near $\rho=0$ is like $\rho^{\textup{Re}(s)-2}$ which is integrable since $\textup{Re}(s)>1$; and near $\rho=\infty$ the integrand decays exponentially because $\textup{Re}(z)>0$. The final equality in \eqref{eq:Mf1} uses \eqref{eq:hurwitz} for $\zeta(s,z)$.

A similar computation for $f_2(x)=-\log\left((-e^{-\kappa z}; e^{-\kappa})_\infty\right)$, yields
\begin{align}
M[f_2;s]&=-\int\limits_0^\infty \log\left((- e^{-\kappa z};e^{-\kappa})_\infty\right) \kappa^{s-1}d\kappa=
\sum\limits_{n=1}^\infty \dfrac{(-1)^n}{n^{s+1}}\int\limits_0^\infty \dfrac{e^{-\rho z} \rho^{s-1} }{ (1-e^{- \rho})}d\rho\\
&=-\Gamma(s)\eta(s+1)\zeta(s ,z)=(2^{-s}-1)\Gamma(s)\zeta(s+1)\zeta(s ,z).
\end{align}
The last equality uses \eqref{eq:eta_z}.

Having computed the Mellin transform for $f_1$ and $f_2$ we now verify that they can be inverted using Proposition \ref{prop:melinver}. We must check analyticity and the quadratic decay estimate.
Due to the analyticity of $\Gamma(s), \zeta(s)$ and $\zeta(s ,z)$, the analyticity of these Mellin transforms holds for $\textup{Re}(s)>1$.

Now we argue that there is quadratic decay. To apply Proposition \ref{prop:melinver}, it suffices to have this decay on any vertical strip. Fix $a=1$ and any $c>1$.  We claim that there exists a constant $K>0$ such that for all $s\in \C$ with $\textup{Re}(s)\in (a,c)$ and all $z\in \C$ with $\textup{Re}(z)>0$,
$
\lvert\Gamma(s)\zeta(s+1)\zeta(s ,z)\rvert\leq K |s|^{-2}.
$
This follows by appealing to the gamma function decay bound in Lemma \ref{lem:asympt_gamma}, the boundedness of $\zeta(s+1)$ for $\textup{Re}(s)\in(a,c)$,  and the bound on $|\zeta(s,z)|$ from Lemma \ref{lem:hurzpos}.
With this, we can invert the Mellin transform of $f_1$ and thus prove the desired formula. The case for $f_2$ is similar since the factor $(2^{-s}-1)$ is bounded by a constant  provided $\textup{Re}(s)\in (a,c)$.
%
%
%
%
%
\end{proof}

\subsubsection{Proof of Proposition \ref{factorials}: asymptotics of
  $(e^{-\kappa z}; e^{-\kappa})_\infty$}\label{sec:qqasy}

 We will first consider the case when $\textup{Re}(z)\geq 1/2$. This is addressed in three steps. Then, in a fourth step we will use the functional identity \eqref{eq:theta} to address the case when $\textup{Re}(z)< 1/2$. In the fifth and final step, we will combine these two bounds into a common bound.

\smallskip
\noindent {\bf Step 1.} We start with the representation for $(e^{-\kappa z}; e^{-\kappa})_\infty$ from Lemma \ref{lem:mellin} and shift the contour of integration to the left of zero, picking up some residues. We claim the following formula: For any non-integer $a<0$,
\begin{align}\label{eq:error}
    -\log(e^{-\kappa z}; e^{-\kappa})_\infty=&\dfrac{1}{2\pi\iu}\int\limits_{a-\iu\infty}^{a+\iu\infty}\!\!\!\Gamma(s)\zeta(s+1)\zeta(s
    ,z)\kappa^{-s} ds \\
    &+\sum\limits^{\ell=1}_{\ell=\ceil{a}}\textup{Res}_{s=\ell}\left[ \Gamma(s)\zeta(s+1)\zeta(s
    ,z)\kappa^{-s}\right].
  \end{align}
The starting point for this is \eqref{eq:mellinqpoch} where the contour has real part $c>1$. The idea is to shift the contour to the left until it lies on a vertical line with real part $a<0$. In doing this, we encounter poles at $s=1, 0, -1, \ldots, \lceil a\rceil$ whose residues must be accounted for. The summation in \eqref{eq:error} is precisely the contribution of those residues. To justify the deformation we use Lemmas \ref{lem:asympt_gamma} and \ref{prop_zeta} with \eqref{eq:asymp_hurwitz_z2}  from Proposition \ref{hurwitz_bound} to show that  for $z$ fixed with $\textup{Re}(z)>0$ there exists some $\varepsilon>0$ and $C>0$ such that for all $s$ with $\textup{Re}(s)=\sigma\in [a,c]$ and $|t|>1$,
 $
 \left|\Gamma(s)\zeta(s+1)\zeta(s ,z)\kappa^{-s}  \right| \leq C e^{-\varepsilon |t|}.
 $

\smallskip
\noindent{\bf Step 2.} Next, we compute the residues in \eqref{eq:error}. In order to do that this, we make use of the following Taylor series expansions. We first address the residue at $s=0$. By using the results in \eqref{eq:derivative_zeta} we see that around $s=0$
    \begin{align}
\Gamma(s)\zeta  (s+1)\zeta(s,z) \kappa^{-s}&=\tfrac{1}{s^2}\Gamma(s+1)\cdot \left(s
        \zeta(s+1)\right)\cdot \zeta(s,z)\cdot \kappa^{-s}=\\
      &=(\tfrac{1}{2}-z)s^{-2}+\left(\left(z-\tfrac{1}{2}\right)\log \kappa + \log \left[\tfrac{\Gamma (z)}{\sqrt{2
            \pi}}\right]
    \right)s^{-1}+\cdots,
      \end{align}
    where $\cdots$ here represents lower order terms in $s$. From this expansion it immediately follows that the residue at $s=0$ is $\left(z-\frac{1}{2}\right)\log \kappa + \log \left[\frac{\Gamma (z)}{\sqrt{2\pi}}\right]$. Turning to the residue at $s=1$, from \eqref{eq:zeta_at_two} we have that
      \begin{equation*}
    \underset{s=1}{\textup{Res}}\left[\Gamma(s)\zeta(s+1)\zeta(s,z) \kappa^{-s}\right]=\dfrac{\Gamma(1)\zeta(2)}{\kappa}\cdot \underset{s=1}{\textup{Res}}\left[\zeta(s,z)\right]=\dfrac{\pi^2}{6 \kappa}.
    \end{equation*}
The residue at $s=-n$ for $n\in \Z_{\geq 1}$ is evaluated by \eqref{eq:gamma_pole} and \eqref{eq:zetaressbern} as
    $$
    \underset{s=-n}{\textup{Res}}\left[\Gamma(s)\zeta(s+1)\zeta(s,z) \kappa^{-s}\right]=\underset{s=-n}{\textup{Res}}\left[\Gamma(s)\right]\cdot
    \zeta(-n+1)\zeta(-n, z)\cdot \kappa^{n}=\dfrac{B_{n}B_{n+1}(z)}{n(n+1)!}\cdot \kappa^{n}.
      $$
Recall that Bernoulli numbers are zero for odd integers. Combining these deductions, we conclude for $\textup{Re}(z)>0$ and non-integer $a<0$ we have
    \begin{align}\label{eq:step216}
    -\log(e^{-\kappa z}; e^{-\kappa})_\infty&=\dfrac{1}{2\pi
      \iu}\int\limits_{a-\iu\infty}^{a+\iu\infty}\Gamma(s)\zeta(s+1)\zeta(s
    ,z)\kappa^{-s}ds\\
   +\dfrac{\pi^2}{6 \kappa}+ &\left(z-\dfrac{1}{2}\right)\log \kappa+\log \left[\dfrac{\Gamma (z)}{\sqrt{2
          \pi}}\right ]+\sum\limits_{n=1}^{\lfloor -a\rfloor}\dfrac{B_{n+1}(z)B_{n}}{n(n+1)!}\kappa^{n} .
    \end{align}

\smallskip
Recall $m\in \Z_{\geq 1}$ in the statement of Proposition \ref{factorials}. For any $a\in (-m,-m+1)$
we may compare the right-hand sides of \eqref{eq:step216} and \eqref{eq:factorialspos} to see that the error term $\Error_m^+[\kappa, z]$ in \eqref{eq:factorialspos} is precisely give by
\begin{equation}\label{eq:eplusint}
\Error_m^+[\kappa, z] = \dfrac{-1}{2\pi
      \iu}\int\limits_{a-\iu\infty}^{a+\iu\infty}\Gamma(s)\zeta(s+1)\zeta(s
    ,z)\kappa^{-s}ds.
\end{equation}
Therefore, our problem reduces to bounding the absolute value of the above integral. Fix some $a\in (-m,-m+1)$. To estimate $|\Error_m^+[\kappa, z] |$, we bring the absolute value inside the integral and utilize the bounds given in Propositions \ref{prop_zeta} and  \ref{hurwitz_bound}. Using the notation $s = a + \iu t$, we will divide the integral into small $|t|$ and large $|t|$. On account of the just mentioned lemma and propositions, for all $d\in (0,1/2)$ such that $a+d\notin\Z_{\leq 0}$, there exists a $t_0>0$ and $C>0$ such that for all
$t<t_0$ and $z\in \C$ with $\textup{Re}(z) \geq 1/2$,
\begin{align}\label{eq:ttnotsmall}
  \left|\Gamma(a+\iu t)\zeta(a+\iu t+1)\zeta(a+\iu t,z)\right| &\leq\\
  C& \Big( e^{\lvert \textup{Arg}(z) \cdot t \rvert }\big( |z|^{1-a}+ |t|^{1/2-a}|z|^{-d-a}\big) +|z|^{-d-a}\Big),
\end{align}
while for all $t>t_0$ and $z\in \C$ with $\textup{Re}(z) \geq 1/2$,
\begin{align}\label{eq:ttnotbig}
&\qquad\left|\Gamma(a+\iu t)\zeta(a+\iu t+1)\zeta(a+\iu t,z)\right| \leq C  e^{-\frac{\pi}{2} |t|}\times\\
& \times\Big( e^{\lvert \textup{Arg}(z) \cdot t \rvert }\big( |z|^{-a} + (1+|t|)^{-1} |z|^{1-a}+ |t|^{1/2-a}|z|^{-d-a}\big) +|t|^{1/2-a}|z|^{-d-a}\Big),
\end{align}
In deriving the above we made some simplifications from the bound in Proposition \ref{hurwitz_bound}.
For $|t|\leq t_0$, we bounded $\max(1,|t|^{1/2-a}) \leq  C$
while
for $|t|\geq t_0$, we bounded $\max(1,|t|^{1/2-a}) \leq  C |t|^{1/2-a}$
where $C>0$ depends on $a$ and $t_0$.
Since we are presently assuming that $\textup{Re}(z) \geq 1/2$ (and hence $|z|\geq 1/2$) we find a constant only dependent on $a$ and $t_0$ such that for $|t|\leq t_0$, $|z|^{-a}+(1+|t|)^{-1}|z|^{1-a} \leq C |z|^{1-a}$ and likewise find $C>0$ only dependent on $d+a$ such that $\sum_{k=0}^{\lfloor -d-a \rfloor } |z|^k \leq  C |z|^{-d-a}$.

%

With the $t_0$ above we may bound
$
|\Error_m^+[\kappa, z]| \leq (\pi)^{-1} \kappa^{-a} (\mathbf{I}+\mathbf{II})
$
where
\begin{align}
\mathbf{I} &= \int_0^{t_0}\left|\Gamma(a+\iu t)\zeta(a+\iu t+1)\zeta(a+\iu t,z)\right|dt,\\
\mathbf{II} &= \int_{t_0}^{\infty} \left|\Gamma(a+\iu t)\zeta(a+\iu t+1)\zeta(a+\iu t,z)\right|dt.
\end{align}
Estimating $\mathbf{I}$ from \eqref{eq:ttnotsmall} is done easily since $t<t_0$  and $|z|\geq 1/2$. The main contribution is from the term $|z|^{1-a}$ and all other terms can be bounded by it. Thus, we find that there exists $C>0$ depending on $a$ and $t_0$ such that
$
\mathbf{I} \leq C |z|^{1-a}.
$
To control $\mathbf{II}$ requires a bit more. Let us recall two facts. The first is an immediate consequence of the gamma function integral formula: For any $\alpha>-1$ and any $\e>0$,
\begin{equation}\label{eq:gammabound}
\int_0^{\infty} e^{-\e t}t^{\alpha} dt = \e^{-\alpha -1} \Gamma(\alpha+1).
\end{equation}
The second is that for all $\textup{Re}(z) \geq 1/2$, there exists a $C>0$ such that
$\frac{\pi}{2}-|\textup{Arg}(z)|\geq C |z|^{-1}$.
With these facts we may show that there exists $C>0$ depending only on $a$ such that
\begin{equation}
\int_{t_0}^{\infty} e^{-(\frac{\pi}{2}-\textup{Arg}(z)) |t|}\big(|z|^{-a} +(1+|t|)^{-1} |z|^{1-a}+ |t|^{1/2-a}|z|^{-d-a}\big)dt \leq
C \left(|z|^{1-a} + |z|^{\frac{3}{2} -2a -d}\right).
\end{equation}
To derive this inequality we first extended the integration to $(0,\infty)$, and then used \eqref{eq:gammabound} with $\e = \frac{\pi}{2}-\textup{Arg}(z)\geq C |z|^{-1}$. Similarly, we find that
\begin{equation}
\int_{t_0}^{\infty} e^{-\frac{\pi}{2}|t|}|t|^{1/2-a}|z|^{-d-a} dt \leq
C |z|^{-d-a}.
\end{equation}
Thus, in light of \eqref{eq:ttnotbig} and the above bounds, we have shown that
$$
\mathbf{II} \leq C  \left(|z|^{1-a} + |z|^{\frac{3}{2} -2a -d}+ |z|^{-d-a}\right),
$$
and combining this with the earlier bound on $\mathbf{I}$ we see that there exists  $C>0$ dependent only on $a$ and $d$ such that
$$
|\Error_m^+[\kappa, z]| \leq C \kappa^{-a} \left(|z|^{1-a} + |z|^{\frac{3}{2} -2a -d}+ |z|^{-d-a}\right)\leq C \kappa^{-a} \left(|z|^{1-a} + |z|^{\frac{3}{2} -2a -d}\right).
$$
Since we were allowed to take $a\in (-m,-m+1)$ arbitrary, and $d\in (0,1/2)$ provided $a+d\notin\Z_{\leq 0}$, we may try to optimize the right-hand side above. Let $b=-a\in (m-1,m)$. Then, we have shown that for any $\e\in (0,1/2)$ and any $b\in (m-1,m)$, there exists $C>0$ such that
\begin{equation}\label{eq:errorzpos}
|\Error_m^+[\kappa, z]| \leq C \kappa^{b} \left(|z|^{1+b} + |z|^{1+2b+\e}\right)\leq  C \kappa^{b}|z|^{1+2b+\e}.
\end{equation}
This completes the proof of the proposition when $\textup{Re}(z)\geq 1/2$. Notice that there is no restriction on $\textup{Im}(z)$ assumed here (in accordance with the statement of the proposition).

\smallskip
\noindent{\bf Step 4.} We will now make use of the functional identity \eqref{eq:theta} to relate the $\textup{Re}(z)<1/2$ behavior to the $\textup{Re}(z)\geq 1/2$ result we have proven already above.
The two Pochhammer symbols, $(e^{-\kappa}; e^{-\kappa})_\infty$ and $(e^{-\kappa (1-z)};    e^{-\kappa})_\infty$, which arise on the right-hand side of \eqref{eq:theta} are both of the form $(e^{-\kappa\tilde{z}}; e^{-\kappa})_\infty$ for $\tilde{z} = 1$ and $1-z$ respectively. In both cases (in the second case, due to the assumption that $\textup{Re}(z)<1/2$), we have that $\textup{Re}(\tilde{z}) \geq 1/2$, hence the result we have proved above can be applied. In particular, for any $m\in \Z_{\geq 1}$, recalling  $\AqP^+[\kappa, z]$ from \eqref{eq:AqPp},
\begin{align}\label{eq:minlogqq}
  -\log(e^{-\kappa(1-z)}; e^{-\kappa})_\infty&=\\
-\AqP^+[\kappa, 1-z]  &+\sum\limits_{n=1}^{m-1}\dfrac{B_{n+1}(1-z)B_{n}}{n(n+1)!}\kappa^{n}+\Error^+_{m}[\kappa, 1-z].\\
  -\log \left(e^{-\kappa}; e^{-\kappa}\right)_\infty=&\\
-\AqP^+[\kappa, 1]  &+\sum\limits_{n=1}^{m-1}\dfrac{B_{n+1}(1)B_{n}}{n(n+1)!}\kappa^{n}+\Error^+_{m}[\kappa, 1].
\end{align}
Recall that from \eqref{eq:bsym} we have $B_{n+1}(1-z)=(-1)^{n+1}B_{n+1}(z)$ and that, in particular, $B_{n+1}=B_{n+1}(1)=0$ for $n$ even. Thus, we may replace $B_{n+1}(1-z)B_n$ by $-B_{n+1}(z)B_n$ (when $n$ is even, both sides are zero). Also observe that the product $B_{n+1}B_{n} =0$ for all $n$ except $n=1$.
By the Euler reflection equation
$\log\left [ \frac{\Gamma(1-z)} {\sqrt 2 \pi}\right ]=-\log\left[2\sin(\pi z)\right]-\log\left[\frac{\Gamma(z)}{\sqrt {2 \pi}}\right]$.
Plugging this into \eqref{eq:theta} we find that we can write
$$\log\left(e^{-\kappa z}; e^{-\kappa}\right)_\infty=
\AqP^+[\kappa, z]-\sum\limits_{n=1}^{m-1}\dfrac{B_{n+1}(z)B_{n}}{n(n+1)!}\kappa^{n}+\Error_m^+[\kappa, z]
$$
where
\begin{align}
\Error_m^+[\kappa, z] &:= \sum\limits_{n=1}^{m-1}\dfrac{B_{n+1}(1)B_{n}}{n(n+1)!}\kappa^{n}+\dfrac{\kappa}{8}-\dfrac{\kappa z}{2}+\dfrac{\kappa
    z^2}{2}
    +\Error ^+_{m} (\kappa, 1)+\Error^+ _{m} (\kappa, 1-z)
    \\&
\qquad+\log\left[\theta_1\left(z\bigg\lvert \dfrac{2\pi\iu}{\kappa}\right)\dfrac{e^{\pi^2/(2\kappa)}}{2 \sin(\pi z)}\right].
\end{align}
This provides an explicit formula for the error term $\Error_m^+[\kappa, z]$ in \eqref{eq:factorialspos}.
The first bound in Lemma \ref{theta_estimate_z} shows that for any $\alpha\in (0,\pi)$ there exist $C,c,\kappa_0>0$ such that for all $\kappa\in (0,\kappa_0)$ and all $z\in\C$ with $|\textup{Im}(z)|<\tfrac{\alpha}{\kappa}$
\begin{equation*}
\left|\log\left(\theta_1\left(z\bigg\lvert \dfrac{2\pi\iu}{\kappa}\right) \dfrac{e^{\pi^2/(2\kappa)}}{2 \sin(\pi z)}\right)\right|
\leq \log\big(1+C\cdot e^{-\frac{c}{\kappa}}\big)\leq C\cdot e^{-\frac{c}{\kappa}}.
\end{equation*}
The two terms $\Error ^+_{m} (\kappa, 1)$ and $\Error^+ _{m} (\kappa, 1-z)$ were bounded in the first three steps of this proof, see \eqref{eq:errorzpos}.  The summation involving Bernoulli numbers has only one non-zero term when $n=1$ and is proportional to $\kappa$. Thus, to summarize so far we have shown that for any $\alpha\in (0,\pi)$, $\e\in (0,1/2)$ and $b\in (m-1,m)$ there exist $C,c,\kappa_0>0$ such that for all $\kappa\in (0,\kappa_0)$ and all $z\in\C$ with $|\textup{Im}(z)|<\tfrac{\alpha}{\kappa}$
\begin{equation}
|\Error_m^+[\kappa, z]| \leq C\left(\kappa +\kappa |1-2z|^2 + \kappa^{b}|1-z|^{1+2b+\e} + \kappa^b + e^{-\frac{c}{\kappa}}\right).
\end{equation}
For each $b$, provided $\kappa$ is small enough, we can bound $e^{-\frac{c}{\kappa}}\leq \kappa^{b}$. Since $\textup{Re}(z)<1/2$ it follows that there exist a constant $C$ which depends on $b$ such that $1\leq C|1-z|^{1+2b+\e}$. Thus, we can simplify our bound to the following: For any $\alpha\in (0,\pi)$, $\e\in (0,1/2)$ and $b\in (m-1,m)$ there exist $C,c,\kappa_0>0$ such that for all $\kappa\in (0,\kappa_0)$ and all $z\in\C$ with $|\textup{Im}(z)|<\tfrac{\alpha}{\kappa}$
\begin{equation}\label{eq:errorzneg}
|\Error_m^+[\kappa, z]| \leq C\left(\kappa +\kappa |1-2z|^2 + \kappa^{b}|1-z|^{1+2b+\e}\right).
\end{equation}

\smallskip
\noindent{\bf Step 5.} In this final step, we combine the bound \eqref{eq:errorzpos} shown in Step 3 for all $\textup{Re}(z)\geq 1/2$ with the bound just shown at the end of Step 4, in \eqref{eq:errorzneg} which holds for  $\textup{Re}(z)< 1/2$ (with some additional condition on the imaginary part). Notice that if  $\textup{Re}(z)\geq 1/2$, then there exists a constant $C$ such that  $|z| \leq C(1+|z|)$, and similarly, if $\textup{Re}(z)< 1/2$, then there exists a constant $C$ such that $|1-2z|\leq C (1+|z|)$. Making these replacements in the respective bounds  \eqref{eq:errorzpos}  and \eqref{eq:errorzneg} immediately leads to the error bound claimed in \eqref{eq:errorbounds} for $\Error_m^+[\kappa, z]$. This completes the proof of the asymptotics of $\log(q^z;q)_{\infty}$ provided $|\textup{Im}(z)|<\tfrac{\alpha}{\kappa}$. We will extend this asymptotic to $|\textup{Im}(z)|<\tfrac{2\alpha}{\kappa}$ using the result for  $\log(-q^z;q)_{\infty}$. To avoid confusion we defer this until the end of the proof of the  $\log(-q^z;q)_{\infty}$ asymptotics.

\subsubsection{Proof of Proposition \ref{factorials}: asymptotics of $(-q^z; q)_\infty$}
Comparing the formula \eqref{eq:mellinqpoch} for $(q^z;q)_{\infty}$ to the formula \eqref{eq:mellinqpochneg} for $(-q^z;q)_{\infty}$, the only difference is the inclusion of the factor $2^{-s}-1$ in the later. This will have very minor effect on the argument, as compared to our earlier study of the asymptotics of $(q^z;q)_{\infty}$ in Section \ref{sec:qqasy}. As such, we just summarize the outcome of each of the five steps from that proof, subject to inclusion now of the multiplicative factor $2^{-s}-1$. As a sixth and final step we include the extension of the  $\log(q^z;q)_{\infty}$ asymptotic to only assume $|\textup{Im}(z)|<\tfrac{2\alpha}{\kappa}$.

\smallskip
\noindent{\bf Step 1.} Assume that $\textup{Re}(z)\geq 1/2$. Under this condition we may use \eqref{eq:mellinqpochneg} and decay estimates to show, as in \eqref{eq:error}, that for non-integer $a<0$,
\begin{align}\label{eq:error2}
    -\log(-e^{-\kappa z};e^{-\kappa})_\infty=
    &\dfrac{1}{2\pi\iu}\!\!\!\int\limits_{a-\iu\infty}^{a+\iu\infty}\!\!\!(2^{-s}-1)\Gamma(s)\zeta(s+1)\zeta(s
    ,z)\kappa^{-s} ds \\
   + &\sum\limits^{\ell=1}_{\ell=\ceil{a}}\!\!\textup{Res}_{s=\ell}\left[(2^{-s}-1)\Gamma(s)\zeta(s+1)\zeta(s
    ,z)\kappa^{-s}\right].
  \end{align}

\smallskip
\noindent{\bf Step 2.} We compute the residues in \eqref{eq:error2}. Notice that around $s=0$, we have $2^{-s}-1 = -s \log(2) +\cdots$. From this we see that around $s=0$,
$$
(2^{-s}-1) \Gamma(s) \zeta(s+1) \zeta(s,z) \kappa^{-s} = \frac{(z-\frac{1}{2})\log 2}{s}+\cdots
$$
where the $\cdots$ represent lower order terms. Hence the residue at $s=0$ is $(z-\tfrac{1}{2})\log 2$. At $s=1$, the factor $2^{-s}-1$ is simply evaluated to $-\tfrac{1}{2}$ and hence the residue there becomes $-\tfrac{\pi^2}{12 \kappa}$. Finally, for the residues at  $-n$ for $n\in \Z_{\geq 1}$, the $2^{-s}-1$ factor contributes a new multiplicative factor $2^n-1$ to the residue. Thus, \eqref{eq:step216} is replaced by (recall  $\AqP^-[\kappa, z]$ from \eqref{eq:AqPm})
    \begin{align}\label{eq:step2162}
    -\log(e^{-\kappa z}; e^{-\kappa})_\infty=&\dfrac{1}{2\pi
      \iu}\int\limits_{a-\iu\infty}^{a+\iu\infty}(2^{-s}-1)\Gamma(s)\zeta(s+1)\zeta(s
    ,z)\kappa^{-s}ds\\
    &-\AqP^-[\kappa, z]+\sum\limits_{n=1}^{\lfloor -a\rfloor}(2^n-1)\dfrac{B_{n+1}(z)B_{n}}{n(n+1)!}\kappa^{n}.
    \end{align}

\smallskip
\noindent{\bf Step 3.} We fix $m\in \Z_{\geq 1}$ and let $a\in (-m,-m+1)$. Comparing the right-hand sides of \eqref{eq:step2162} and \eqref{eq:factorialsneg} we see that the error term $\Error_m^-[\kappa, z]$ in \eqref{eq:factorialsneg} is precisely give by
\begin{equation}
\Error_m^-[\kappa, z] = \dfrac{-1}{2\pi
      \iu}\int\limits_{a-\iu\infty}^{a+\iu\infty}(2^{-s}-1)\Gamma(s)\zeta(s+1)\zeta(s
    ,z)\kappa^{-s}ds.
\end{equation}
Bounding the absolute value of the error term $|\Error_m^-[\kappa, z]|$ proceeds quite similar to that of $|\Error_m^-[\kappa, z]|$ in \eqref{eq:eplusint}.  Since along our integration contour, $\textup{Re}(s)=a$, it follows that we can bound $|2^{-s}-1| \leq 2^{-a}+1$ which can be absorbed into the constant $C$. Thus, we arrive as the same bound as in \eqref{eq:errorzpos}. In particular, we see that for any $\e\in (0,1/2)$ and $b\in(m-1,m)$, there exists a constant $C>0$ such that
\begin{equation}\label{eq:errorzneg2}
|\Error_m^-[\kappa, z]| \leq C \kappa^{b}|z|^{1+2b+\e}.
\end{equation}
This result still assumes that $\textup{Re}(z)\geq 1/2$.

\smallskip
\noindent{\bf Step 4.} We now return to $\textup{Re}(z)< 1/2$ and use \eqref{eq:theta_negative} to deduce the asymptotics in that regime from those derived above for $\textup{Re}(z)\geq 1/2$.
Since $\textup{Re}(1-z)\geq 1/2$, we may conclude from the results of Steps 1--3 above that for any $m\in \Z_{\geq 1}$, 
\begin{equation}
-\log(-e^{-\kappa (1-z)}; e^{-\kappa})_\infty\!\!\!=-\AqP^-[\kappa, 1-z]+\sum\limits_{n=1}^{m-1}\left (2^{n}-1 \right)\dfrac{B_{n+1}(1-z)B_{n}}{n(n+1)!}\kappa^{n}+\Error^-_{m}[\kappa, 1-z].
\end{equation}
Combining this with the asymptotics of $-\log\left(e^{-\kappa}; e^{-\kappa}\right)_\infty$
in \eqref{eq:minlogqq}  and the form of
the right-hand sides of \eqref{eq:factorialsneg} in
\eqref{eq:factorialsneg} yields
\begin{align}\label{eq:theta_negative_1}
\log\left(-e^{-\kappa z}; e^{-\kappa}\right)_\infty =\AqP^-[\kappa, z]-\sum\limits_{n=1}^{m-1} (2^n-1)\dfrac{B_{n+1}(z)B_{n}}{n(n+1)!}\kappa^{n}+\Error_m^-[\kappa, z]
\end{align}
where $\Error_m^-[\kappa, z]$ is given by
\begin{align} \Error_m^-[\kappa, z]=&\sum\limits_{n=1}^{m-1}\dfrac{B_{n+1}(1)B_{n}}{n(n+1)!}\kappa^{n}+\dfrac{\kappa}{8}-\dfrac{\kappa z}{2}+\dfrac{\kappa z^2}{2}+\Error^+_{\lfloor -d\rfloor}[\kappa,1]+\Error^-_{\lfloor -d\rfloor}[\kappa,1- z]\\
&\quad
+\log\Big(\theta_4 \Big(z \Big \vert \tfrac{2\pi \iu}{\kappa}\Big)\Big).
\end{align}
We may now invoke the second bound in Lemma \ref{theta_estimate_z} which shows that for any $\alpha\in (0,\pi)$ there exists $C,c,\kappa_0>0$  such that for all $\kappa\in (0,\kappa_0)$ and all $z\in\C$ with $|\textup{Im}(z)|<\tfrac{\alpha}{\kappa}$ 
$$
\left|\log\left(\theta_4\left(z\bigg\lvert \frac{2\pi\iu}{\kappa}\right)\right)\right| \leq \log\big(1+C\cdot e^{-\frac{c}{\kappa}}\big)\leq C\cdot e^{-\frac{c}{\kappa}}.
$$
The term $\Error ^+_{m} (\kappa, 1)$ was bounded in Section \ref{sec:qqasy} whereas $\Error^-_{m} (\kappa, 1-z)$ has been already bounded in the first three steps of this proof, see \eqref{eq:errorzneg2}. Thus, in the same manner as in Section \ref{sec:qqasy} we conclude that: For any $\alpha\in (0,\pi)$, $\e\in (0,1/2)$ and $b\in (m-1,m)$ there exist $C,c,\kappa_0>0$ such that for all $\kappa\in(0,\kappa_0)$ and all $z\in\C$ with $|\textup{Im}(z)|<\tfrac{\alpha}{\kappa}$ (again, this is assuming $\textup{Re}(z)< 1/2$)
$$
|\Error_m^-[\kappa, z]| \leq C\left(\kappa +\kappa |1-2z|^2 + \kappa^{b}|1-z|^{1+2b+\e}\right).$$

\smallskip
\noindent{\bf Step 5.} The synthesis of the $\textup{Re}(z)\geq 1/2$  and $\textup{Re}(z)< 1/2$ bounds on the error are precisely the same as above, hence the proof is complete.

\smallskip
\noindent{\bf Step 6.} 
The fact that the bound on $\Error^{+}_m[\kappa, z]$ holds with $|\textup{Im}(z)|<\tfrac{2\alpha}{\kappa}$ follows from \eqref{eq:errorbounds} for $\Error^{\pm}_m[\kappa, z]$ with $|\textup{Im}(z)|<\tfrac{\alpha}{\kappa}$  by using the transformation \eqref{eq:transfo}. To see this, consider $z$ with $\textup{Im}(z)\in[\tfrac{\alpha}{\kappa},\tfrac{2\alpha}{\kappa})$ (the negative imaginary case is the same). Then
\begin{align}
\log (q^z;q)_{\infty} &= \log (-q^{z-\frac{\pi}{\kappa} \iu};q)_{\infty} \\
&= \AqP^-[\kappa, z-\frac{\pi}{\kappa} \iu] -\sum\limits_{n=1}^{m-1}\left(2^{n}-1\right)\dfrac{B_{n+1}( z-\frac{\pi}{\kappa} \iu)B_{n}}{n(n+1)!}\kappa^{n}+\Error_m^-[\kappa,  z-\frac{\pi}{\kappa} \iu].
\end{align}
Comparing this to \eqref{eq:factorialspos} shows that
\begin{align}
\Error_m^+[\kappa, z] &= \AqP^-[\kappa, z-\frac{\pi}{\kappa} \iu] - \AqP^+[\kappa, z] \\ &\quad +\sum\limits_{n=1}^{m-1}\left(\dfrac{B_{n+1}(z)B_{n}}{n(n+1)!}\kappa^{n}-\left(2^{n}-1\right)\dfrac{B_{n+1}( z-\frac{\pi}{\kappa} \iu)B_{n}}{n(n+1)!}\kappa^{n}\right)+\Error_m^-[\kappa,  z-\frac{\pi}{\kappa} \iu].
\end{align}
For
$\textup{Im}(z)\in[\tfrac{\alpha}{\kappa},\tfrac{2\alpha}{\kappa})$ we
have that $\textup{Im}[z-\frac{\pi}{\kappa} \iu]\in
[\frac{\alpha-\pi}{\kappa},\frac{2\alpha-\pi}{\kappa})$ which is
contained in $(-\frac{\beta}{\kappa},\frac{\beta}{\kappa})$ for some
$\beta\in (0,\pi)$. Thus we can apply the bound \eqref{eq:errorbounds}
to control $\Error_m^-[\kappa,  z-\frac{\pi}{\kappa} \iu]$. All of the
other terms are also easily controlled (for the Gamma function, use
the asymptotics from Lemma \ref{lem:asympt_gamma}) and doing so we
verify that  \eqref{eq:errorbounds} holds for $\Error_m^+[\kappa,
z]$. The reason why we cannot do the same extension of the range of
imaginary part for $\Error_m^-[\kappa, z]$ is that in replicating that
above argument we encounter $\AqP^+[\kappa, z-\frac{\pi}{\kappa} \iu]
- \AqP^-[\kappa, z]$ and the first term may have singularities from
the Gamma function.

Note that the above argument works for obtaining a bound on $\Error_m^{\pm}[\kappa, z]$ for $z$ with $\textup{dist}(\textup{Re}(z),\mathbb{Z}_{\leq 0})>r$ for any $r>0$ since we will avoid the singularities from the Gamma function.
This implies the final claim of the proposition.

\ack
I. Corwin was supported by a Packard Fellowship in Science and Engineering and a W.~M.~Keck Foundation Science and Engineering Grant and NSF DMS:1811143, DMS:1664650 and DMS:1937254.  A. Knizel was supported by NSF DMS:1704186.
The authors wish to thank Wlodek Bryc and Yizao Wang for discussions surrounding Askey-Wilson processes and their tangent processes, and in particular to Wlodek Bryc for visiting Columbia and inspiring us to work on this subject and Yizao Wang for illuminating discussions around the atoms which occur in our work (in particular, see Remark \ref{remark:ASEPgen}). We would also like to thank Wlodek Bryc for asking clarifying questions about Proposition \ref{factorials} which prompted us to uncover and remedy a mistake in its original formulation and application.
The authors also wish to thank Alex Dunlap, Alexandra Florea, Hao Shen and Kevin Yang for helpful discussions and comments related to this work, as well as two referees who provided very detailed and helpful remarks on an earlier version of this paper.



\frenchspacing
\bibliographystyle{cpam.bst}
\bibliography{refs}

\end{document}